\long\def\qg#1\ik{\cy{#1}}\long\def\qo#1\wj{}{\catcode`\/=0\catcode`\\=12/long/gdef/qz#1\bytes#2|#3/wj{/expandafter/gdef/expandafter/hj/expandafter{/hj #1}/if/relax/detokenize{#3}/relax/expandafter/qo/fi/qz#3/wj}/long/gdef/qf#1{/gdef/hj{}/qz#1\bytes|/wj}}\long\def\cy#1{\endgroup\gdef\DocumentHash{}\ifx\empty#1\empty{\def\x{\multiply~by2\x}\x}\fi\long\gdef\nz{#1}\qf{#1}\xdef\DocumentHash{\pdfmdfivesum{\hj}}\everyeof{}\newlinechar=13\relax\expandafter\scantokens\expandafter{\nz}}\def\qh{\begingroup\catcode`\\=12\catcode`\{=12\catcode`\}=12\catcode`\$=12\catcode`\&=12\catcode`\#=12\catcode`\^=12\catcode`\_=12\catcode`\%=12\catcode`\~=12\catcode`\ =12\catcode`\^^I=12\catcode`\^^M=12\endlinechar=`\^^M\qg}\everyeof{\ik}\qh\tracingmacros=0\tracingcommands=0\tracingonline=0\let\tracingmacros\relax\let\tracingcommands\relax\newcount\lor\newcount\gv\newcount\vh\newcount\vl\newcount\vt\newtoks\wz\newbox\qB\newbox\qC\newbox\qE\newcount\qN\def\wj{}\def\qr#1#2\pj{\gv=`#1\relax\edef\jz{#2#1}}\def\qa{\expandafter\qr\jz\pj}\def\qu#1{\edef\jz{#1}\qa}{\catcode`\~=12\gdef\qb#1{\lccode`\~=#1\relax\lowercase{\toks0{~}}\wz=\expandafter\expandafter\expandafter{\expandafter\the\expandafter\wz\the\toks0}}}\def\qy#1{\lor=#1\relax\advance\lor by-\gv\relax\ifnum\lor<0\advance\lor by256\relax\fi\ifnum\lor=0\else\ifnum\lor=10\qb{13}\else\qb\lor\fi\fi\qa}\def\cv#1{\vt=`#1\relax\ifnum\vt>92\advance\vt by-34\relax\else\advance\vt by-33\relax\fi}\def\qs#1{\cv#1\vl=\numexpr\vl*85+\vt\relax\vt=\vl\divide\vt by65536\relax\vh=\numexpr\vh*85+\vt\relax\advance\vl by-\numexpr\vt*65536\relax}\def\qd#1#2#3#4#5{\vh=0\relax\vl=0\relax\qs#1\qs#2\qs#3\qs#4\qs#5\vt=\vh\divide\vt by256\relax\qy\vt\advance\vh by-\numexpr\vt*256\relax\qy\vh\vt=\vl\divide\vt by256\relax\qy\vt\advance\vl by-\numexpr\vt*256\relax\qy\vl}\def\ql#1#2#3#4#5{\ifx\wj#1\relax\else\qd#1#2#3#4#5\expandafter\ql\fi}\def\qk{\begingroup\catcode`\{=12\catcode`\}=12\catcode`\$=12\catcode`\&=12\catcode`\#=12\catcode`\^=12\catcode`\_=12\catcode`\%=12\catcode`\~=12\catcode`\ =12\catcode`\^^I=12\catcode`\^^M=12\endlinechar=`\^^M\qm}\long\def\qm#1#2#3|{\endgroup\qu\DocumentHash\cv#2\loop\ifnum\vt>0\qa\advance\vt by-1\relax\repeat\wz{}\ql#3\wj\wj\wj\wj\wj\newlinechar=13\relax\expandafter\scantokens\expandafter{\the\wz}}\expandafter\let\csname bytes\endcsname\qk\documentclass[10pt,reqno]{amsart}\makeatletter\RequirePackage{mathtools}\RequirePackage{amsfonts}\RequirePackage{mathrsfs}\RequirePackage{amssymb}\RequirePackage{amsthm}\RequirePackage{theoremref}\RequirePackage{microtype}\RequirePackage{bm}\RequirePackage{tikz-cd}\RequirePackage{listings}\@ifclassloaded{beamer}{}{\@ifclassloaded{mdwlist}{\RequirePackage[shortlabels]{enumitem}}{\RequirePackage[shortlabels,inline]{enumitem}}\RequirePackage{xr-hyper}\RequirePackage{hyperref}}\RequirePackage{graphicx}\RequirePackage{iftex}\ifPDFTeX\RequirePackage[T1]{fontenc}\fi\RequirePackage{accsupp}\RequirePackage{pdfrender}\RequirePackage{iftex}\AtBeginDocument{ \@ifpackageloaded{hyperref}{\hypersetup{pdfauthor={},pdftitle={},pdfsubject={},pdfkeywords={},pdfcreator={},pdfproducer={}}}{}\ifPDFTeX \pdfinfo{/Author () /Title () /Subject () /Keywords () /Creator () /Producer ()}\fi \ifLuaTeX \pdfextension info{/Author () /Title () /Subject () /Keywords () /Creator () /Producer ()}\fi }\ifPDFTeX\pdfinfoomitdate1\pdfsuppressptexinfo-1\pdftrailerid{}\fi\ifLuaTeX\pdfvariablesuppressoptionalinfo-1\pdfvariabletrailerid{}\fi\let\e\boxtimes\newsavebox{\f}\newdimen{\g}\newdimen{\h}\newcommand{\m}[2]{{\mathsurround=0pt \sbox{\f}{$#1#2$} \g=\dimexpr(\ht\f+\dp\f)*13/16\relax \h=\dimexpr\g/22\relax \vcenter{\hbox{\rlap{\makebox[\wd\f]{\begin{tikzpicture}[ x=\g, y=\g, line width=\h, line cap=butt, line join=miter ] \draw(0,0)rectangle(1,1); \draw(0,0)--(1,1)(0,1)--(1,0);\end{tikzpicture}}}\makebox[\wd\f]{\smash{\pdfrender{TextRenderingMode=Invisible}{$#1\e$}}}}}}}\newcommand{\n}{\varepsilon}\newcommand{\p}[1]{\widetilde{#1}}\usetikzlibrary{decorations,positioning}\pgfkeys{/tikz/.cd,alt double distance/.initial=5pt,alt double step/.initial=1pt,}\pgfdeclaredecoration{double deco}{initial}{\state{initial}[width=\pgfkeysvalueof{/tikz/alt double step},next state=cont]{ \pgfmoveto{\pgfpoint{\pgfkeysvalueof{/tikz/alt double step}}{\pgfkeysvalueof{/tikz/alt double distance}/2}} \pgfpathlineto{\pgfpoint{0.3\pgflinewidth}{\pgfkeysvalueof{/tikz/alt double distance}/2}} \pgfpathmoveto{\pgfpoint{0.3\pgflinewidth}{-\pgfkeysvalueof{/tikz/alt double distance}/2}} \pgfpathlineto{\pgfpoint{1pt}{-\pgfkeysvalueof{/tikz/alt double distance}/2}} \pgfcoordinate{lastup}{\pgfpoint{1pt}{\pgfkeysvalueof{/tikz/alt double distance}/2}} \pgfcoordinate{lastdown}{\pgfpoint{1pt}{-\pgfkeysvalueof{/tikz/alt double distance}/2}} } \state{cont}[width=\pgfkeysvalueof{/tikz/alt double step}]{ \pgfmoveto{\pgfpointanchor{lastup}{center}} \pgfpathlineto{\pgfpoint{\pgfkeysvalueof{/tikz/alt double step}}{\pgfkeysvalueof{/tikz/alt double distance}/2}} \pgfcoordinate{lastup}{\pgfpoint{\pgfkeysvalueof{/tikz/alt double step}}{\pgfkeysvalueof{/tikz/alt double distance}/2}} \pgfmoveto{\pgfpointanchor{lastdown}{center}} \pgfpathlineto{\pgfpoint{\pgfkeysvalueof{/tikz/alt double step}}{-\pgfkeysvalueof{/tikz/alt double distance}/2}} \pgfcoordinate{lastdown}{\pgfpoint{\pgfkeysvalueof{/tikz/alt double step}}{-\pgfkeysvalueof{/tikz/alt double distance}/2}} } \state{final}[width=0pt]{\pgfmoveto{\pgfpointdecoratedpathlast} } }\@ifclassloaded{beamer}{}{\def\@setdate{\datename\ \@date \ifx\s\@empty\else .\ \q\ \s \fi \@addpunct.} \def\q{\textit{Last Update}:} {} \let\s\@empty }\RequirePackage{xparse}\theoremstyle{plain}\newtheorem{theorem}{Theorem}[subsection]\newtheorem{lemma}[theorem]{Lemma}\newtheorem{corollary}[theorem]{Corollary}\newtheorem{proposition}[theorem]{Proposition}\theoremstyle{definition}\newtheorem{definition}[theorem]{Definition}\newtheorem{example}[theorem]{Example}\theoremstyle{remark}\newtheorem{remark}[theorem]{Remark}\RequirePackage{tikz}\usetikzlibrary{backgrounds}\usetikzlibrary{arrows}\usetikzlibrary{shapes,shapes.geometric,shapes.misc}\pgfkeys{/tikz/tikzit fill/.initial=0}\pgfkeys{/tikz/tikzit draw/.initial=0}\pgfkeys{/tikz/tikzit shape/.initial=0}\pgfkeys{/tikz/tikzit category/.initial=0}\pgfdeclarelayer{edgelayer}\pgfdeclarelayer{nodelayer}\pgfsetlayers{background,edgelayer,nodelayer,main}\tikzstyle{every loop}=[]\newcommand{\w}{\mathscr{U}}\newcommand{\x}{\mathfrak{I}}\DeclareFontFamily{U}{MnSymbolC}{}\DeclareFontShape{U}{MnSymbolC}{m}{n}{ <-6> MnSymbolC5 <6-7> MnSymbolC6 <7-8> MnSymbolC7 <8-9> MnSymbolC8 <9-10> MnSymbolC9 <10-12> MnSymbolC10 <12-> MnSymbolC12}{}\DeclareFontShape{U}{MnSymbolC}{b}{n}{ <-6> MnSymbolC-Bold5 <6-7> MnSymbolC-Bold6 <7-8> MnSymbolC-Bold7 <8-9> MnSymbolC-Bold8 <9-10> MnSymbolC-Bold9 <10-12> MnSymbolC-Bold10 <12-> MnSymbolC-Bold12}{}\DeclareSymbolFont{MnSyC}{U}{MnSymbolC}{m}{n}\DeclareMathSymbol{\smallstar}{\mathbin}{MnSyC}{128}\newcommand{\y}{\mathop{-}\nolimits}\newcommand{\z}{\mathcal{DR}}\newcommand{\B}{\mathcal{K}}\usepackage[labelformat=empty]{caption}\usepackage{subcaption}\usepackage[textsize=tiny]{todonotes}\usepackage{tensor}\usepackage{eqparbox}\usepackage[letterpaper,margin=1in,headsep=0.25in,marginparwidth=2.1cm]{geometry}\parindent=0.2in\mathtoolsset{showonlyrefs,showmanualtags}\numberwithin{equation}{section}\usetikzlibrary{arrows,calc,decorations.pathreplacing,decorations.pathmorphing,backgrounds,decorations.markings,shapes.geometric,ext.paths.arcto}\hypersetup{colorlinks=true,allcolors=blue,linktoc=all,}\title{The crystal of the categorified quantum group}\author[S.-H. Lee]{Soo-Hong Lee}\address{Department of Mathematical Sciences, Seoul National University, Seoul 08826, Korea}\email{shlee@crystalline.site}\thanks{This research was supported by Basic Science Research Program through the National Research Foundation of Korea (NRF) funded by the Ministry of Education (RS-2025-25412049).}
\begin{document}\begin{abstract}A categorification of the crystal $B(\widetilde{U}_q(\mathfrak{g}))$ of the modified quantized enveloping algebra associated to a symmetrizable Kac--Moody algebra $\mathfrak{g}$ is given. A crystal is constructed from a categorical $\mathfrak{g}$-action on abelian categories without finiteness conditions, and is used to construct the bicrystal structure on the set of self-dual indecomposable objects of the categorified quantum group. Crystals from Webster's generalized tensor product categorification are identified with tensor products of lowest and highest weight crystals. As an application, a categorification of the extremal weight module is obtained.\end{abstract}\maketitle\makeatletter \let\D\contentsline\renewcommand*{\contentsline}[4]{\D{#1}{\textcolor{black}{#2}}{#3}{#4}}\makeatother \tableofcontents\section{Introduction}In \cite{_} Chuang and Rouquier introduced the notion of $\mathfrak{sl}_2$-categorification. It consists of a pair of biadjoint functors $E$ and $F$ acting on an abelian category and inducing an $\mathfrak{sl}_2$-action on the Grothendieck group, together with natural transformations satisfying the relations of affine Hecke algebras. They showed that this structure, while ubiquitous in representation theory, has a rich theory and used it to prove Brou\'e's abelian defect group conjecture for symmetric groups. Among the structures it entails is the tight control of the endomorphism rings of the objects $F^{(m)}M$ for a simple $M$, which, in retrospect, induces an $\mathfrak{sl}_2$-crystal structure on the set of isomorphism classes of simple objects, generalizing \cite{A} and \cite{B}.\par In \cite{C} and \cite{D}, Khovanov and Lauda introduced a family of algebras, now called the Khovanov--Lauda--Rouquier or quiver Hecke algebras, categorifying the negative half of the quantized enveloping algebra. As a step in establishing the categorification, they considered a graph on the set of isomorphism classes of simple modules, and Lauda and Vazirani showed that it is isomorphic to the crystal $B(\infty)$, and that the one for cyclotomic quotients is isomorphic to the crystal $B(\lambda)$ \cite{E}.\par In \cite{F}, Rouquier introduced the 2-category $\mathscr{U}(\mathfrak{g})$ associated to a symmetrizable Kac--Moody algebra $\mathfrak{g}$, by which an $\mathfrak{sl}_2$-categorification can be understood as a 2-representation of $\mathscr{U}(\mathfrak{sl}_2)$ with certain finiteness conditions. Khovanov and Lauda \cite{G} introduced a 2-category with a different presentation, later shown to be isomorphic to Rouquier's by Brundan \cite{H}, and observed that $\mathscr{U}(\mathfrak{g})$ categorifies the modified quantized enveloping algebra $\widetilde{U}(\mathfrak{g})$ under the so-called \emph{nondegeneracy hypothesis}, now known to hold in full generality due to \cite{I} and \cite{J}.\par In \cite{K}, Webster constructed the categorification of highest weight modules $V(\lambda)$ (independently established by Kang and Kashiwara \cite{L}) as well as their tensor products as 2-representations of $\mathscr{U}(\mathfrak{g})$. The crystal induced by this construction was shown to be isomorphic to the tensor product of the corresponding crystals $B(\lambda)$ by Losev and Webster \cite{M}, building on the work of Losev \cite{N}.\par In this paper, we categorify the crystal $B(\widetilde{U}(\mathfrak{g}))$ of the modified quantized enveloping algebra (\thref{_}). It is a crystal over $\mathfrak{g}\oplus\mathfrak{g}$, and as noted by Kashiwara in \cite{O}, where the crystal was first studied, this is a manifestation of the bimodule structure of $\widetilde{U}(\mathfrak{g})$ over itself. In our construction, the two crystal structures come from the two Yoneda 2-representations $\mathscr{U}(\lambda,-)$ and $\mathscr{U}(-,\mu)$. Each crystal element parametrizes a self-dual indecomposable object of (the additive Karoubi envelope of) $\mathscr{U}(\lambda,\mu)$ for some $\lambda,\mu\in X$, and thus we regard the crystal as a structure intrinsic to the categorified quantum group. As an application, we obtain a categorification of the extremal weight module $V(\lambda)$, the generalization of integrable highest weight modules to arbitrary weights $\lambda\in X$ (\thref{A}).\par Together with earlier results, this work may be viewed as completing a categorical analogue of the theory of crystals developed in a series of works by Kashiwara \cite{P,Q,O}. The boundedness of the weights of connected components of $B(\widetilde{U}(\mathfrak{g}))$ \cite[\S9]{O} has no counterpart in this paper. It would be interesting to find one, and to explore how further structures of the crystal, such as the Peter--Weyl type decomposition of \cite{R}, interact with the categorification.\subsection{Outline}The first half of the paper develops a framework that yields a crystal from a categorical $\mathfrak{g}$-action without the usual finiteness conditions.\par In Section \ref{B}, we fix notations and conventions for 2-categories following Rouquier \cite{F} and Johnson--Yau \cite{S}, as well as for their graded versions and idempotent completions. Next, in Section \ref{C}, we recall the definitions of quiver Hecke algebras, the categorified quantum group $\mathscr{U}$, and related structures.\par In Section \ref{D}, we generalize Proposition 5.20 of \cite{_} to arbitrary abelian categories. This is the statement that controls $\operatorname{End}(F^{(m)}M)$ for a simple $M$, and that forms the basis of the construction of the categorical crystal. Existing versions of the statement are not suitable for our purpose: It is not clear whether the functors $E_i$ and $F_i$ preserve finite-length modules, and thus there is no ``$K_0$'' to work on. Our proof is based on the thick calculus of categorified quantum $\mathfrak{sl}_2$ developed in \cite{T}, and in particular, relies on the Sto\v si\'c formula.\par In Section \ref{E}, we consider a family of humorous categories equipped with a $\mathscr{U}$-action, and show that their orthodox bases carry a natural crystal structure. This provides a setting in which the crystal of the humorous categories $\dot{\w}(\lambda,\mu)$ is understood. Although $E_i$ and $F_i$ do not act locally nilpotently on $\mathscr{U}(\lambda,\mu)$, we show that they do on its graded simple modules. We then give a compatibility condition between the $\mathscr{U}$-action and the duality functors, and show that the crystal operators restrict to self-dual objects.\par The second half of the paper is devoted to identifying the crystal of the orthodox bases of $\dot{\w}(\lambda,\mu)$ with $B(\widetilde{U}(\mathfrak{g}))$. We obtain this by categorifying the construction of the canonical (or global crystal) basis of $\widetilde{U}(\mathfrak{g})$ by Lusztig \cite{U} and Kashiwara \cite{O}, in which each component of $\widetilde{U}(\mathfrak{g})$ is approximated by an inverse system of tensor products of a lowest weight module and a highest weight module. Here, the role of the tensor products is played by algebras introduced by Webster \cite{V}, which we refer to as generalized tensor product categorifications.\par In Section \ref{F}, we show that the crystal of the categorification of $V(\lambda_1)\otimes\cdots\otimes V(\lambda_m)$, where $\lambda_1,\ldots,\lambda_k\in-X^+$ and $\lambda_{k+1},\ldots,\lambda_m\in X^+$, is isomorphic to the tensor product crystal $B(\lambda_1)\otimes\cdots\otimes B(\lambda_m)$. We give a definition of an upper finite fully stratified category of Brundan--Stroppel \cite{W} adapted to our purpose. Then we extend the axiomatic approach of Losev and Webster \cite{M} to \cite{V} so that lowest and highest weight modules can be mixed. We generalize the Ext vanishing lemma of \cite{N,M} to this setting. At the end of the section, we revisit the construction of \cite{V}, which is the first statement in the paper that depends on the nondegeneracy hypothesis.\par Combining the results established so far, we prove the main theorem of the paper in Section \ref{G}: we show that the $\mathfrak{g}\oplus\mathfrak{g}$-crystal structure on the orthodox bases of $\dot{\w}(\lambda,\mu)$ for all $\lambda,\mu\in X$ is isomorphic to $B(\widetilde{U}(\mathfrak{g}))$, independently of the polynomials $Q_{ij}$ (the quiver Hecke datum) and the base field. By an argument of Kashiwara \cite[\S8]{O} with global bases replaced by orthodox bases, we obtain in Section \ref{H} a categorification of the extremal weight module, as well as the fundamental weight modules in affine type. This immediately yields a symmetric bilinear form on the extremal weight module for an arbitrary symmetrizable Kac--Moody algebra, which was not known beyond affine type, to the best of the author's knowledge.\par\section{Preliminaries on 2-categories}\label{B}\par Let $\mathcal{C}$ be a category. For objects $x,y\in\mathcal{C}$, we denote the set of morphisms from $x$ to $y$ by $\mathcal{C}(x,y)$, and the identity morphism of $x$ by $1_x\in\mathcal{C}(x,x)$. When the category $\mathcal{C}$ is evident from the context, we also write $\operatorname{Hom}(x,y)$ for $\mathcal{C}(x,y)$.\par Unless otherwise specified, we will work over a commutative ring with unit $\Bbbk$.\par\subsection{Graded categories}\label{I}\par A graded category $\mathcal{K}$ is a category equipped with an auto-equivalence $q\colon\mathcal{K}\to\mathcal{K}$. We call the elements of $\mathcal{K}_d(x,y)\coloneqq\varinjlim_n\mathcal{K}(q^{n+d}x,q^ny)$ morphisms from $x$ to $y$ of degree $d$, where the direct limit is taken with respect to the maps $f\mapsto q(f)$. We suppress the action of $q$ on morphisms and denote $q^k(f)\colon q^kx\to q^ky$ simply by $f$, with the degree shift understood from the context.\par A graded $\Bbbk$-linear category is a graded category with a $\Bbbk$-linear structure such that $q$ is $\Bbbk$-linear. Given a graded $\Bbbk$-linear category $\mathcal{K}$, we denote by $\mathcal{K}|_{q=1}$ the category enriched over graded $\Bbbk$-modules, whose class of objects is the same as that of $\mathcal{K}$, whose hom-spaces are given by $\mathcal{K}|_{q=1}(x,y)=\bigoplus_{d\in\mathbb{Z}}\mathcal{K}_d(x,y)$, and whose composition is induced by that of $\mathcal{K}$. Upon a choice of an adjoint pair $(q,q^{-1})$ of functors, which is unique up to unique isomorphism, we have a canonical isomorphism $\mathcal{K}_d(x,y)\cong\mathcal{K}(q^dx,y)$, and $\mathcal{K}|_{q=1}$ is the category denoted by $\widehat{\mathcal{K}}$ in \cite[Def.~5.1]{X}. When $\mathcal{K}$ is clear from the context, we also write $\operatorname{Hom}_q(x,y)$ and $\operatorname{End}_q(x)$ for $\mathcal{K}|_{q=1}(x,y)$ and $\mathcal{K}|_{q=1}(x,x)$, respectively.\par For graded $\Bbbk$-linear categories $\mathcal{K}$ and $\mathcal{L}$, we denote by $\mathcal{K}\mathbin{\BeginAccSupp{method=hex,unicode,ActualText=22A0}\mathpalette{\m}{\otimes}\EndAccSupp{}}\mathcal{L}$ the graded $\Bbbk$-linear category whose objects are symbols $x\mathbin{\BeginAccSupp{method=hex,unicode,ActualText=22A0}\mathpalette{\m}{\otimes}\EndAccSupp{}}y$ for $x\in\mathcal{K}$ and $y\in\mathcal{L}$, and whose hom-spaces are $(\mathcal{K}\mathbin{\BeginAccSupp{method=hex,unicode,ActualText=22A0}\mathpalette{\m}{\otimes}\EndAccSupp{}}\mathcal{L})(x\mathbin{\BeginAccSupp{method=hex,unicode,ActualText=22A0}\mathpalette{\m}{\otimes}\EndAccSupp{}}y,x'\mathbin{\BeginAccSupp{method=hex,unicode,ActualText=22A0}\mathpalette{\m}{\otimes}\EndAccSupp{}}y')\coloneqq\bigoplus_{d\in\mathbb{Z}}\mathcal{K}_d(x,x')\otimes_\Bbbk\mathcal{L}_{-d}(y,y')$, where $q(x\mathbin{\BeginAccSupp{method=hex,unicode,ActualText=22A0}\mathpalette{\m}{\otimes}\EndAccSupp{}}y)\coloneqq qx\mathbin{\BeginAccSupp{method=hex,unicode,ActualText=22A0}\mathpalette{\m}{\otimes}\EndAccSupp{}}y$. For $\Bbbk$-linear functors $M$ (resp. $N$) from $\mathcal{K}^{\operatorname{op}}$ (resp.\ $\mathcal{L}^{\operatorname{op}}$) to the category $\Bbbk\operatorname{\!-\!}\operatorname{mod}$ of $\Bbbk$-modules, we denote by $M\mathbin{\BeginAccSupp{method=hex,unicode,ActualText=22A0}\mathpalette{\m}{\otimes}\EndAccSupp{}}N$ the $\Bbbk$-linear functor $(\mathcal{K}\mathbin{\BeginAccSupp{method=hex,unicode,ActualText=22A0}\mathpalette{\m}{\otimes}\EndAccSupp{}}\mathcal{L})^{\operatorname{op}}\to\Bbbk\operatorname{\!-\!}\operatorname{mod}$ given by $(M\mathbin{\BeginAccSupp{method=hex,unicode,ActualText=22A0}\mathpalette{\m}{\otimes}\EndAccSupp{}}N)(x\mathbin{\BeginAccSupp{method=hex,unicode,ActualText=22A0}\mathpalette{\m}{\otimes}\EndAccSupp{}}y)\coloneqq\bigoplus_{d\in\mathbb{Z}}M(q^dx)\otimes_\Bbbk N(q^{-d}y)$ and $(M\mathbin{\BeginAccSupp{method=hex,unicode,ActualText=22A0}\mathpalette{\m}{\otimes}\EndAccSupp{}}N)(f\otimes g)\coloneqq M(f)\otimes N(g)$ for $f\in\mathcal{K}(x,q^ex')$ and $g\in\mathcal{L}(q^ey,y')$.\subsection{2-categories}\label{J}\par Our terminology, notation, and conventions concerning 2-categories mostly follow \cite{S}. Let us briefly recall the 2-categorical notions we use. Among the data of a definition and the coherence conditions imposed on the data, we describe only the data, and refer to \cite{S} for the complete list of coherence conditions. Similarly, when we construct a 2-categorical structure in this paper, we specify only the data and omit the routine verification of the coherence conditions.\par A 2-category is a $\mathrm{Cat}$-enriched category, or equivalently, a bicategory with identity associators and unitors. Let $\mathscr{C}$ be a 2-category. For an object $x\in\mathscr{C}$, the identity object of $\mathscr{C}(x,x)$ is denoted by $\textnormal{\fontfamily{pbk}\selectfont1}_x$. Let $c_{x,y,z}\colon\mathscr{C}(y,z)\times\mathscr{C}(x,y)\to\mathscr{C}(x,z)$ be the composition functor of $\mathscr{C}$. Given 2-morphisms $\delta$ and $\gamma$ in $\mathscr{C}(y,z)$ and $\mathscr{C}(x,y)$, $c_{x,y,z}((\delta,\gamma))$ is called the horizontal composition of $\delta$ and $\gamma$, and is denoted by $\delta\mathbin{\boldsymbol{\cdot}}\gamma$.\par A \emph{pseudofunctor} $\mathcal{F}\colon\mathscr{C}\to\mathscr{D}$ consists of a function on objects, functors $\mathcal{F}\colon\mathscr{C}(x,y)\to\mathscr{D}(\mathcal{F}(x),\mathcal{F}(y))$ for each pair of objects $x,y\in\mathscr{C}$, and natural isomorphisms $\mathcal{F}^2$ and $\mathcal{F}^0$ as in the diagram below, subject to the coherence conditions of \cite[Def.~4.1.2]{S}. \begin{equation}\label{K}\begin{tikzpicture}[xscale=2.3,yscale=.8,baseline=-3pt]\node(A)at(-1,1){$\mathscr{C}(y,z)\times\mathscr{C}(x,y)$};\node(B)at(1,1){$\mathscr{C}(x,z)$};\node(C)at(-1,-1){$\mathscr{D}(\mathcal{F}(y),\mathcal{F}(z))\times\mathscr{D}(\mathcal{F}(x),\mathcal{F}(y))$};\node(D)at(1,-1){$\mathscr{D}(\mathcal{F}(x),\mathcal{F}(z))$};\draw[->](A)--(B)node[midway,above,scale=.9]{$c_{x,y,z}$};\draw[->](A)--(C)node[midway,left]{$\mathcal{F}\times\mathcal{F}$};\draw[->](B)--(D)node[midway,right]{$\mathcal{F}$};\draw[->](C)--(D)node[midway,below,scale=.9]{$c$};\draw[double distance=2pt,-{Implies}](-.15,-.15)--(.15,.15)node[midway,auto=right,]{$\mathcal{F}^2_{x,y,z}$};\end{tikzpicture}\qquad\begin{tikzpicture}[xscale=1.1,yscale=.8,baseline=-3pt]\node(A)at(-1,1){$\textnormal{\fontfamily{pbk}\selectfont1}$};\node(B)at(1,1){$\mathscr{C}(x,x)$};\node(C)at(1,-1){$\mathscr{D}(\mathcal{F}(x),\mathcal{F}(x))$};\draw[->](A)--(B)node[midway,above]{$\textnormal{\fontfamily{pbk}\selectfont1}_x$};\draw[->](A)--(C)node[midway,auto=right,inner sep=.7pt]{$\textnormal{\fontfamily{pbk}\selectfont1}_{\mathcal{F}(x)}$};\draw[->](B)--(C)node[midway,right]{$\mathcal{F}$};\draw[double distance=2pt,-{Implies}](.25,.25)--(.5,.5)node[midway,auto=right,inner sep=1.5pt,scale=.9]{$\mathcal{F}^0_{x}$};\end{tikzpicture}\end{equation} A \emph{2-functor} is a pseudofunctor in which $\mathcal{F}^2$ and $\mathcal{F}^0$ are the identity natural transformations.\par A \emph{strong transformation} $\alpha\colon\mathcal{F}\to\mathcal{G}$, where $\mathcal{F},\mathcal{G}\colon\mathscr{C}\to\mathscr{D}$ are pseudofunctors, consists of 1-morphisms $\alpha_x\in\mathscr{D}(\mathcal{F}(x),\mathcal{G}(x))$ for each object $x\in\mathscr{C}$, and a natural isomorphism $\alpha\colon(\alpha_x)^\ast\circ\mathcal{G}\to\mathcal{F}\circ(\alpha_y)_\ast$, where $(\alpha_x)^\ast\colon\mathscr{D}(\mathcal{G}(x),\mathcal{G}(y))\to\mathscr{D}(\mathcal{F}(x),\mathcal{G}(y))$ and $(\alpha_y)_\ast\colon\mathscr{D}(\mathcal{F}(x),\mathcal{F}(y))\to\mathscr{D}(\mathcal{F}(x),\mathcal{G}(y))$ are functors given by precomposition with $\alpha_x$ and postcomposition with $\alpha_y$, respectively. This data is subject to the coherence conditions of \cite[Def.~4.2.1]{S}. In more explicit terms, for each $x,y\in\mathscr{C}$ and $f\in\mathscr{C}(x,y)$, we are given 1-morphisms $\alpha_x,\alpha_y$ and an invertible 2-morphism $\alpha_f$ fitting into the diagram below. \[\begin{tikzpicture}[xscale=.9,yscale=.8]\node(A)at(-1,1){$\mathcal{F}(x)$};\node(B)at(1,1){$\mathcal{F}(y)$};\node(C)at(-1,-1){$\mathcal{G}(x)$};\node(D)at(1,-1){$\mathcal{G}(y)$};\draw[->](A)--(B)node[midway,above,scale=.9]{$\mathcal{F}(f)$};\draw[->](A)--(C)node[midway,left]{$\alpha_x$};\draw[->](B)--(D)node[midway,right]{$\alpha_y$};\draw[->](C)--(D)node[midway,below,scale=.9]{$\mathcal{G}(f)$};\draw[double distance=2pt,-{Implies}](-.15,-.15)--(.15,.15)node[midway,auto=right,inner sep=2pt]{$\alpha_f$};\end{tikzpicture}\vspace{-.5em}\] A \emph{2-natural transformation} is a strong transformation in which $\alpha_f$ for every 1-morphism $f$ in $\mathscr{C}$ is the identity 2-morphisms. A \emph{modification} $\gamma\colon\alpha\to\beta$ between two strong transformations $\alpha,\beta\colon\mathcal{F}\to\mathcal{G}$ consists of 2-morphisms $\gamma_x\colon\alpha_x\to\beta_x$ for each object $x\in\mathscr{C}$, subject to the coherence conditions of \cite[Def.~4.4.1]{S}.\par Given a 2-category $\mathscr{C}$, there exists a 2-category $\mathscr{C}^{\operatorname{op}}$ satisfying $\mathscr{C}^{\operatorname{op}}(x,y)=\mathscr{C}(y,x)$ for $x,y\in\mathscr{C}$, and a 2-category $\mathscr{C}^{\operatorname{co}}$ satisfying $\mathscr{C}^{\operatorname{co}}(x,y)=\mathscr{C}(x,y)^{\operatorname{op}}$ (cf.\ \cite[\S 2.6]{S}). We denote $\mathscr{C}^{\operatorname{coop}}\coloneqq(\mathscr{C}^{\operatorname{op}})^{\operatorname{co}}=(\mathscr{C}^{\operatorname{co}})^{\operatorname{op}}$. Given a pseudofunctor $\mathcal{F}\colon\mathscr{C}\to\mathscr{D}$, there exist pseudofunctors $\mathcal{F}^{\operatorname{op}}\colon\mathscr{C}^{\operatorname{op}}\to\mathscr{D}^{\operatorname{op}}$ and $\mathcal{F}^{\operatorname{co}}\colon\mathscr{C}^{\operatorname{co}}\to\mathscr{D}^{\operatorname{co}}$. Given a strong transformation $\alpha\colon\mathcal{F}\to\mathcal{G}$ of pseudofunctors, there exist strong transformations $\alpha^{\operatorname{op}}\colon\mathcal{G}^{\operatorname{op}}\to\mathcal{F}^{\operatorname{op}}$ and $\alpha^{\operatorname{co}}\colon\mathcal{F}^{\operatorname{co}}\to\mathcal{G}^{\operatorname{co}}$.\par There exists a 2-category $\mathrm{Bicat}(\mathscr{C},\mathscr{D})$, whose objects are pseudofunctors, 1-morphisms are strong transformations, and 2-morphisms are modifications.\footnote{In \cite[Cor.~4.4.13]{S}, this 2-category is denoted by $\mathsf{Bicat}^{\mathsf{ps}}(\mathscr{C},\mathscr{D})$, and $\mathsf{Bicat}(\mathscr{C},\mathscr{D})$ denotes a larger 2-category of lax functors.} Given a pseudofunctor $\mathcal{F}\colon\mathscr{C}'\to\mathscr{C}$, there exists a canonical 2-functor $\mathcal{F}^\ast\colon\mathrm{Bicat}(\mathscr{C},\mathscr{D})\to\mathrm{Bicat}(\mathscr{C}',\mathscr{D})$, such that $\mathcal{F}^\ast(\mathcal{G})=\mathcal{G}\mathcal{F}$ on objects, and $\mathcal{F}^\ast(\alpha)=\alpha\mathbin{\smallstar}\mathcal{F}$ on 1-morphisms, where $\alpha\mathbin{\smallstar}\mathcal{F}$ denotes the \emph{pre-whiskering} (as in \cite[Def.~11.1.1]{S}). Given a pseudofunctor $\mathcal{H}\colon\mathscr{D}\to\mathscr{D}'$, there exists a canonical pseudofunctor $\mathcal{H}_\ast\colon\mathrm{Bicat}(\mathscr{C},\mathscr{D})\to\mathrm{Bicat}(\mathscr{C},\mathscr{D}')$ which is a 2-functor if $\mathcal{H}$ is a 2-functor, such that $\mathcal{H}_\ast(\mathcal{G})=\mathcal{H}\mathcal{G}$ on objects, and $\mathcal{H}_\ast(\alpha)=\mathcal{H}\mathbin{\smallstar}\alpha$ on 1-morphisms, where $\mathcal{H}\mathbin{\smallstar}\alpha$ denotes the \emph{post-whiskering} (as in \cite[Def.~11.1.1]{S}).\par Let us remark on the commutativity of whiskerings. Given pseudofunctors $\mathcal{F}'\colon\mathscr{C}''\to\mathscr{C}'$ and $\mathcal{H}'\colon\mathscr{D}'\to\mathscr{D}''$, we have $(\mathcal{F}\circ\mathcal{F}')^\ast=\mathcal{F}'^\ast\circ\mathcal{F}^\ast$, $(\mathcal{H}'\circ\mathcal{H})_\ast=\mathcal{H}'_\ast\circ\mathcal{H}_\ast$, and $\mathcal{H}_\ast\circ\mathcal{F}^\ast=\mathcal{F}^\ast\circ\mathcal{H}_\ast$. Thus, for instance, for a strong transformation $\alpha\colon\mathcal{G}_1\to\mathcal{G}_2$ of functors $\mathcal{G}_1,\mathcal{G}_2\colon\mathscr{C}\to\mathscr{D}$, we write $\mathcal{H}'\mathbin{\smallstar}\mathcal{H}\mathbin{\smallstar}\alpha\mathbin{\smallstar}\mathcal{F}\mathbin{\smallstar}\mathcal{F}'$ for the strong transformation $\mathcal{H}'_\ast\mathcal{H}_\ast\mathcal{F}'^{\ast}\mathcal{F}^\ast(\alpha)$ and it is unambiguous without parentheses.\par Let $\mathcal{F},\mathcal{F}'\colon\mathcal{B}\to\mathcal{C}$ and $\mathcal{G},\mathcal{G}'\colon\mathcal{C}\to\mathcal{D}$ be pseudofunctors, and $\alpha\colon\mathcal{F}\to\mathcal{F}'$ and $\beta\colon\mathcal{G}\to\mathcal{G}'$ be strong transformations. There is an invertible modification \begin{equation}\label{L}\gamma\colon(\mathcal{G}'\mathbin{\smallstar}\alpha)\circ(\beta\mathbin{\smallstar}\mathcal{F})\to(\beta\mathbin{\smallstar}\mathcal{F}')\circ(\mathcal{G}\mathbin{\smallstar}\alpha)\end{equation} whose component at an object $x\in\mathcal{B}$ is given by $\gamma_x\coloneqq\beta_{\alpha_x}$ (cf.\ \cite[\S 11]{S}). We will only need this when $\beta$ is a 2-natural transformation, in which case $\gamma$ is the identity.\subsection{$\Bbbk$-linear and graded 2-categories}We call a 2-category $\mathscr{C}$ \emph{$\Bbbk$-linear} if each Hom-category has the structure of a $\Bbbk$-linear category, and each composition functor $c_{x,y,z}$ is $\Bbbk$-bilinear. A pseudofunctor or a 2-functor $\mathcal{F}\colon\mathscr{C}\to\mathscr{D}$ between $\Bbbk$-linear 2-categories is \emph{$\Bbbk$-linear} if the functor $\mathcal{F}\colon\mathscr{C}(x,y)\to\mathscr{D}(\mathcal{F}(x),\mathcal{F}(y))$ is $\Bbbk$-linear for all $x,y\in\mathscr{C}$.\par Let $\mathscr{C}$ and $\mathscr{D}$ be $\Bbbk$-linear 2-categories. There exists a sub-2-category $\mathrm{Bicat}_\Bbbk(\mathscr{C},\mathscr{D})$ of $\mathrm{Bicat}(\mathscr{C},\mathscr{D})$, whose objects are $\Bbbk$-linear pseudofunctors, and having the same 1-morphisms and 2-morphisms as $\mathrm{Bicat}(\mathscr{C},\mathscr{D})$. The constructions in the previous subsection restrict to give 2-functors $\mathrm{Bicat}_\Bbbk(\mathscr{C},\mathscr{D})\allowbreak\to\mathrm{Bicat}_\Bbbk(\mathscr{C}',\mathscr{D})$ and $\mathrm{Bicat}_\Bbbk(\mathscr{C},\mathscr{D})\to\mathrm{Bicat}_\Bbbk(\mathscr{C},\mathscr{D}')$.\par We call $\mathscr{C}$ a \emph{graded $\Bbbk$-linear} 2-category if it is $\Bbbk$-linear, and each Hom-category is a graded $\Bbbk$-linear category with an autoequivalence $q$, such that the composition functor $c_{x,y,z}$ satisfies $c_{x,y,z}\circ(q\times q)=q\circ c_{x,y,z}$ for all $x,y,z\in\mathscr{C}$. For a graded $\Bbbk$-linear 2-category $\mathscr{C}$, we denote by $\mathscr{C}|_{q=1}$ the 2-category obtained by replacing Hom-categories $\mathscr{C}(x,y)$ with $\mathscr{C}(x,y)|_{q=1}$, and by naturally extending the composition functors.\subsection{Idempotent completion of 2-categories}\label{M}\par For a $\Bbbk$-linear category $\mathcal{K}$, we denote by $\dot{\mathcal{K}}$ the \emph{additive Karoubi envelope} of $\mathcal{K}$, that is, the Karoubi envelope of the additive closure of $\mathcal{K}$. Objects of $\dot{\mathcal{K}}$ are pairs $(\bigoplus_{i=1}^nx_i,(e_{ij})_{1\le i,j\le n})$ for $n\in\mathbb{Z}_{\ge0}$, $x_i\in\mathcal{K}$, and $e_{ij}\in\mathcal{K}(x_j,x_i)$ such that $(e_{ij})$ is an idempotent matrix.\par Let $\mathcal{K}$ be a small $\Bbbk$-linear category, and $\iota_{\mathcal{K}}\colon\mathcal{K}\to\dot{\mathcal{K}}$ be the canonical inclusion functor. Let $\mathcal{L}$ be a $\Bbbk$-linear category, and $\iota_{\mathcal{K}}^\ast\colon\mathrm{Cat}_\Bbbk(\dot{\mathcal{K}},\mathcal{L})\to\mathrm{Cat}_\Bbbk(\mathcal{K},\mathcal{L})$ be the functor given by precomposition with $\iota_{\mathcal{K}}$. Then we have (cf.\ \cite[Prop.~6.5.9]{Y}): (i) $\iota_{\mathcal{K}}^\ast$ is fully faithful. (ii) If $\mathcal{L}$ is additive and idempotent-complete, then $\iota_{\mathcal{K}}^\ast$ is surjective on objects.\par Let $\mathscr{C}$ be a $\Bbbk$-linear 2-category. We define $\dot{\mathscr{C}}$ to be the $\Bbbk$-linear 2-category whose objects are the same as $\mathscr{C}$, and whose Hom-categories $\dot{\mathscr{C}}(x,y)$ are given by the additive Karoubi envelope of $\mathscr{C}(x,y)$ for $x,y\in\mathscr{C}$. The 2-category structure on $\mathscr{C}$ naturally extends to $\dot{\mathscr{C}}$. Let $\iota_{\mathscr{C}}\colon\mathscr{C}\to\dot{\mathscr{C}}$ be the canonical 2-functor. The proof of the following lemma is routine, and is omitted.\par\begin{lemma}\thlabel{N} Let $\mathscr{C}$ be a small $\Bbbk$-linear 2-category, and $\iota_{\mathscr{C}}\colon\mathscr{C}\to\dot{\mathscr{C}}$ be the canonical 2-functor. Let $\mathscr{D}$ be a $\Bbbk$-linear 2-category, and $\iota_{\mathscr{C}}^\ast\colon\mathrm{Bicat}_\Bbbk(\dot{\mathscr{C}},\mathscr{D})\to\mathrm{Bicat}_\Bbbk(\mathscr{C},\mathscr{D})$ be the 2-functor given by precomposition with $\iota_{\mathscr{C}}$. \begin{enumerate}[(i)\ensuremath{^{\prime}},font=\normalfont]\item $\iota_{\mathscr{C}}^\ast$ induces isomorphisms on Hom-categories.\item If $\mathcal{F}\in\mathrm{Bicat}_\Bbbk(\mathscr{C},\mathscr{D})$ is such that $\mathscr{D}(\mathcal{F}(x),\mathcal{F}(y))$ is additive and idempotent-complete for all $x,y\in\mathscr{C}$, then $\mathcal{F}$ has an inverse image under $\iota_{\mathscr{C}}^\ast$. \label{O}\end{enumerate}\end{lemma}\par\begin{remark}In \thref{N} \ref{O}, an inverse image of a 2-functor may only exist as a pseudofunctor.\end{remark}\par\subsection{2-representations}\label{P}\par Let $\mathscr{C}$ be a $\Bbbk$-linear 2-category. A \emph{2-representation} of $\mathscr{C}$ is a $\Bbbk$-linear pseudofunctor $\mathcal{V}\colon\mathscr{C}\to\mathrm{Cat}_\Bbbk$. In this case, we say that $\mathscr{C}$ \emph{acts} on a collection of categories $\{\mathcal{V}_x\coloneqq\mathcal{V}(x)\}_{x\in\mathscr{C}}$, in analogy with the standard notation for weight spaces in representation theory.\par\begin{remark}By the strictification theorems (e.g.\ \cite[\S 8]{S}), any $\Bbbk$-linear pseudofunctor $\mathcal{V}\colon\mathscr{C}\to\mathrm{Cat}_\Bbbk$ is equivalent to a 2-functor, and hence statements on 2-representations that depend only on isomorphism classes of objects of $\mathcal{V}_x$ can be proved by assuming that $\mathcal{V}$ is a 2-functor, which is the reduction we will apply throughout. For the 2-category $\mathscr{U}$ in \thref{Q}, there is no relation imposed on generating 1-morphisms, and this implies that any pseudofunctor $\mathcal{V}\colon\mathscr{U}\to\mathrm{Cat}_\Bbbk$ can be strictified to a 2-functor by an icon instead of a strong transformation.\end{remark}\par By a morphism between two 2-representations $\mathcal{V},\mathcal{V}'\colon\mathscr{C}\to\mathrm{Cat}_\Bbbk$, we mean a morphism in the 2-category $\mathrm{Bicat}_\Bbbk(\mathscr{C},\mathrm{Cat}_\Bbbk)$, that is, a strong transformation $\phi\colon\mathcal{V}\to\mathcal{V}'$. Similarly, we say that $\mathcal{V}$ and $\mathcal{V}'$ are equivalent if there exist 1-morphisms $f\colon\mathcal{V}\to\mathcal{V}'$ and $g\colon\mathcal{V}'\to\mathcal{V}$ in $\mathrm{Bicat}_\Bbbk(\mathscr{C},\mathrm{Cat}_\Bbbk)$ such that $fg$ and $gf$ are isomorphic to the identity 1-morphisms.\par\begin{lemma}\thlabel{R} Suppose that $\mathcal{V}\colon\mathscr{C}\to\mathrm{Cat}_\Bbbk$ is a $\Bbbk$-linear pseudofunctor, and $\mathcal{V}_x$ is additive and idempotent-complete for all $x\in\mathscr{C}$. Then there exists a $\Bbbk$-linear pseudofunctor $\dot{\mathcal{V}}\colon\dot{\mathscr{C}}\to\mathrm{Cat}_\Bbbk$ such that $\mathcal{V}=\dot{\mathcal{V}}\circ\iota_{\mathscr{C}}$.\end{lemma} \begin{proof}Given an idempotent in $\mathrm{Cat}_\Bbbk(\mathcal{V}_x,\mathcal{V}_y)$, its splitting is given by a coequalizer of the idempotent and the identity, which can be computed objectwise. Since $\mathcal{V}_y$ is idempotent-complete, $\mathrm{Cat}_\Bbbk(\mathcal{V}_x,\mathcal{V}_y)$ is idempotent-complete. Noting that $\mathrm{Cat}_\Bbbk(\mathcal{V}_x,\mathcal{V}_y)$ is additive, we may apply \thref{N} \ref{O}.\end{proof}\par For a 1-morphism $F\in\mathscr{C}(x,y)$ and an object $M\in\mathcal{V}_x$, we set $FM\coloneqq\mathcal{V}(F)(M)$. Also, we compose $\mathcal{V}$ with the 2-functor $\dot{(-)}\colon\mathrm{Cat}_\Bbbk\to\mathrm{Cat}_\Bbbk$ taking additive Karoubi envelopes, and implicitly choose its lift $\dot{\mathcal{V}}\colon\dot{\mathscr{C}}\to\mathrm{Cat}_\Bbbk$ by applying \thref{R}. Then we set $\dot{F}M\coloneqq\dot{\mathcal{V}}(\dot{F})(M)$ for $\dot{F}\in\dot{\mathscr{C}}(x,y)$.\par Finally, let us briefly describe relations between these notions and 2-representations found in the literature. When $\mathscr{C}$ is $\mathscr{U}$ or $\mathscr{U}|_{q=1}$ defined in \thref{Q} below, a 2-functor $\mathscr{C}\to\mathrm{Cat}_\Bbbk$ coincides with the notion of \emph{categorical $\mathfrak{g}$-action} found in the literature \cite{_,F,M,H}. A morphism between 2-representations coincides with a \emph{strongly equivariant functor} in \cite{X} (see \cite[Rem.~4.7--4.8]{Z}). The convention for $\dot{F}M$ aligns with e.g. \cite{_} where $\dot{F}M$ for a certain 1-morphism $\dot{F}$ living in an idempotent completion is defined by choosing a splitting of $FM$.\section{Categorification of quantum groups}\label{C}\subsection{Quantized enveloping algebras}\label{S}\par A Cartan datum is a tuple $(I,\cdot)$ of a set $I$ and a bilinear form $\cdot\colon\mathbb{Z}[I]\otimes\mathbb{Z}[I]\to\mathbb{Z}$, such that $i\cdot i\in2\mathbb{N}$ and $2\frac{i\cdot j}{i\cdot i}\in\mathbb{Z}_{\le0}$ holds for any $i\ne j$ in $I$, and $i\cdot j=0$ for all but finitely many $j\in I$ for each $i\in I$.\par\begin{definition}A \emph{weight datum} of type $(I,\cdot)$ is a pair $(X,\{\langle i,-\rangle\}_{i\in I})$ consisting of a set $X$ equipped with an action $X\times\mathbb{Z}[I]\to X$ \,($(\lambda,\nu)\mapsto\lambda+\nu$), and a family of maps $\langle i,-\rangle\colon X\to\mathbb{Z}$ $(i\in I)$, satisfying \begin{enumerate}\item $\langle i,\lambda+j\rangle=\langle i,\lambda\rangle+2\frac{i\cdot j}{i\cdot i}$ for all $i,j\in I$ and $\lambda\in X$,\item $\langle i,\lambda\rangle=0$ for all but finitely many $i\in I$ for each $\lambda\in X$.\end{enumerate} For $\epsilon\in\{-,+\}$, we set $X^{\epsilon}\coloneqq\left\{\kern.07em\lambda\in X\,\middle|\,\epsilon\langle i,\lambda\rangle\ge0\text{ for all }i\in I\kern.07em\right\}$.\end{definition}\par\begin{remark}\thlabel{T} A root datum $(Y,X,\langle-,-\rangle)$ of type $(I,\cdot)$ in the sense of \cite{a} gives a weight datum $(X,\{\langle i,-\rangle\}_{i\in I})$, where the action of $i\in I$ on $\lambda\in X$ is given by $\lambda+i'$ (the sum in $X$), and the maps $\langle i,-\rangle$ are restrictions of the pairing $\langle-,-\rangle\colon Y\times X\to\mathbb{Z}$.\end{remark}\par Given a weight datum, let $\widetilde{U}$ be the associated idempotent-completed quantized enveloping algebra. It is a $\mathbb{Q}(q)$-linear category, whose set of objects is $X$, and whose morphisms are generated by $E_i1_\lambda\colon\lambda\to\lambda+i$ and $F_i1_\lambda\colon\lambda\to\lambda-i$ for $i\in I$ and $\lambda\in X$ with relations $E_iF_j1_\lambda-F_jE_i1_\lambda=\delta_{ij}[\langle i,\lambda\rangle]_i1_\lambda$ and $0=\sum_{m+n=a_{ij}}(-1)^mE_i^{(m)}E_jE_i^{(n)}1_\lambda=\sum_{m+n=a_{ij}}(-1)^mF_i^{(m)}F_jF_i^{(n)}1_\lambda$, where $a_{ij}=1-2\frac{i\cdot j}{i\cdot i}$ and standard notations for $q$-integers and divided powers are used. Let $\widetilde{U}_\mathbb{Z}$ be the $\mathbb{Z}[q^{\pm}]$-linear subcategory of $\widetilde{U}$ generated by $E_i^{(n)}1_\lambda$ and $F_i^{(n)}1_\lambda$ for $i\in I$, $\lambda\in X$, and $n\in\mathbb{N}$. We also denote $E_i$ and $F_i$ by $E_i^+$ and $E_i^-$ as usual.\par For $\lambda\in X^\epsilon$, we denote by $V(\lambda)$ the integrable $\widetilde{U}$-module generated by a vector $u_\lambda$ of weight $\lambda$ with $E_i^\epsilon u_\lambda=0$ $(i\in I)$, and by $V_\mathbb{Z}(\lambda)$ its $\widetilde{U}_\mathbb{Z}$-submodule generated by $u_\lambda$.\par\subsection{Quiver Hecke algebras}A \emph{quiver Hecke datum} of type $(I,\cdot)$ is a matrix $Q=(Q_{ij})_{i,j\in I}$ with entries in $\Bbbk[u,v]$ satisfying: \begin{enumerate}[(i)]\item $Q_{ii}=0$ for all $i\in I$,\item $Q_{ij}(u,v)=Q_{ji}(v,u)$ for all $i,j\in I$,\item $Q_{ij}(u^{i\cdot i},0)=t_{ij}u^{-2i\cdot j}$ for some $t_{ij}\in\Bbbk^\times$ for all $i\ne j$, $i,j\in I$,\item $Q_{ij}(u^{i\cdot i},v^{j\cdot j})$ is a homogeneous polynomial of degree $-2i\cdot j$.\end{enumerate} We set $\overline{Q}_{ij}(u,v,w)\coloneqq\frac{Q_{ij}(u,v)-Q_{ij}(w,v)}{u-w}\in\Bbbk[u,v,w]$.\par Associated to $Q$, we define two quiver Hecke algebras $\mathcal{R}^-$ and $\mathcal{R}^+$. These will be regarded as categorifications of the negative and positive parts $U^-$ and $U^+$ of the quantized enveloping algebra (cf.\ \cite{D}).\par\begin{definition}For $\epsilon\in\{-,+\}$, a graded $\Bbbk$-linear strict monoidal category $\mathcal{R}^\epsilon$ is generated by objects $i\in I$, and by morphisms $x_i\colon q^{i\cdot i}i\to i$ and $\tau_{ij}\colon q^{-i\cdot j}j\otimes i\to i\otimes j$ ($i,j\in I$), subject to relations \begin{enumerate}[(a)]\item $(1_i\otimes x_j)\tau_{ij}-\tau_{ij}(x_j\otimes1_i)=\tau_{ij}(1_j\otimes x_i)-(x_i\otimes1_j)\tau_{ij}=\epsilon\delta_{ij}1_{j\otimes i}$,\item $\tau_{ji}\tau_{ij}=Q_{ij}(1_j\otimes x_i,x_j\otimes1_i)$,\item $(1_i\otimes\tau_{jk})(\tau_{ik}\otimes1_j)(1_k\otimes\tau_{ij})-(\tau_{ij}\otimes1_k)(1_j\otimes\tau_{ik})(\tau_{jk}\otimes1_i)=\epsilon\delta_{i,k}\overline{Q}_{ij}(1_{k\otimes j}\otimes x_i,1_k\otimes x_j\otimes1_i,x_k\otimes1_{j\otimes i})$.\end{enumerate} For an element $u=q^di_1\otimes\cdots\otimes i_m\in\mathcal{R}^\epsilon$, we set $|u|\coloneqq i_1+\cdots+i_m\in\mathbb{N}[I]$. Given an element $\nu\in\mathbb{N}[I]$, we denote by $\mathcal{R}^\epsilon(\nu)$ the full subcategory of $\mathcal{R}^\epsilon$ consisting of objects $u$ satisfying $|u|=\nu$.\end{definition}\par\begin{definition}\thlabel{U} \begin{enumerate*}[(i),font=\normalfont]\item A graded strict monoidal functor $\ast\colon\mathcal{R}^\epsilon\to(\mathcal{R}^\epsilon)^{\operatorname{rev}}$ is defined by $i\mapsto i$ on objects and by $x_i\mapsto x_i$ and $\tau_{ij}\mapsto-\tau_{ji}$ on morphisms. \label{V}\item A strict monoidal functor $-\colon\mathcal{R}^\epsilon\to(\mathcal{R}^\epsilon)^{\operatorname{op}}$ is defined by $q^di\mapsto q^{-d}i$ on objects and by $x_i\mapsto x_i$ and $\tau_{ij}\mapsto\tau_{ji}$ on morphisms. \label{W}\end{enumerate*}\end{definition}\par Assume that a weight datum $(X,\{\langle i,-\rangle\}_{i\in I})$ of type $(I,\cdot)$ is given. We define cyclotomic quiver Hecke algebras that categorify integrable highest and lowest weight modules of $\widetilde{U}$.\par\begin{definition}\thlabel{X} Given $\lambda\in X^{\epsilon}$ and $\mu\in X$, $\mathcal{R}^\lambda_\mu$ is the disjoint union of the quotient of $\mathcal{R}^{-\epsilon}(\nu)$ by the two-sided ideal generated by $x_i^{\epsilon\langle i,\lambda\rangle}\otimes1_u\in\operatorname{End}_q(i\otimes u)$ for $\nu\in\mathbb{N}[I]$ such that $\lambda-\epsilon\nu=\mu$, $i\in I$, and $u\in\mathcal{R}(\nu-i)$.\end{definition}\subsection{The 2-category $\mathscr{U}$}\label{Y}\par We use a cyclic presentation of the 2-category $\mathscr{U}$ from \cite{b}, but our diagrams are the left--right mirror images of those in loc.~cit., so that the composition of 1-morphisms is read from right to left as in \cite{V} and \cite{K}.\par\begin{definition}A \emph{choice of bubble parameters} associated to a quiver Hecke datum $Q$ and a weight datum $(X,\{\langle i,-\rangle\}_{i\in I})$ is a collection of functions $c=\{c_i\colon X\to\Bbbk^\times\}_{i\in I}$, satisfying $c_i(\lambda+j)=t_{ij}c_i(\lambda)$ for $\lambda\in X$ and $i,j\in I$ (where $t_{ii}=1$ by convention).\end{definition}\par Assume that a choice of bubble parameters $c$ is given.\par\begin{definition}\thlabel{Q} A graded $\Bbbk$-linear 2-category $\mathscr{U}$ consists of the following data. \begin{itemize}\item The set of objects $X$.\item 1-morphisms generated by $E_i\in\mathscr{U}(\lambda,\lambda+i)$ and $F_i\in\mathscr{U}(\lambda,\lambda-i)$ for $\lambda\in X,i\in I$, together with their formal degree shifts denoted by $q$.\item 2-morphisms generated by the following.\end{itemize} \begin{gather}x\coloneqq\tikz[baseline,very thick,->,scale=1.5]{\draw[postaction={ decorate, decoration={ markings, mark=at position .5 with {\fill(0,0)circle(2.5pt);} } }](0,.3)--(0,-.1)node[at end,below,text height=1.62ex,text depth=.5ex,inner ysep=2pt,scale=.8]{$i$};}\colon q^{i\cdot i}F_i\to F_i,\quad\tau\coloneqq\tikz[baseline,scale=1.5]{\draw[very thick,baseline,->](.2,.3)--(-.2,-.1)node[at end,below,text height=1.62ex,text depth=.5ex,inner ysep=2pt,scale=.8]{$i$};\draw[very thick,baseline,->](-.2,.3)--(.2,-.1)node[at end,below,text height=1.62ex,text depth=.5ex,inner ysep=2pt,scale=.8]{$j$};}\colon q^{-i\cdot j}F_jF_i\to F_iF_j,\label{Z}\displaybreak[0]\\\varepsilon\coloneqq\tikz[baseline,very thick,->,scale=0.4]{\draw[->](0,0)..controls(.4,1.2)and(1.6,1.2)..(2,0)node[at start,below,text height=1.62ex,text depth=.5ex,inner ysep=2pt,scale=0.8]{$i$};\node at(-.5,.5){$\lambda$};}\colon q_i^{\langle i,\lambda\rangle+1}F_iE_i\textnormal{\fontfamily{pbk}\selectfont1}_\lambda\to\textnormal{\fontfamily{pbk}\selectfont1}_\lambda,\quad\varepsilon'\coloneqq\tikz[baseline,very thick,->,scale=0.4]{\draw[<-](0,0)..controls(.4,1.2)and(1.6,1.2)..(2,0)node[at start,below,text height=1.62ex,text depth=.5ex,inner ysep=2pt,scale=0.8]{$i$};\node at(-.5,.5){$\lambda$};}\colon q_i^{-\langle i,\lambda\rangle+1}E_iF_i\textnormal{\fontfamily{pbk}\selectfont1}_\lambda\to\textnormal{\fontfamily{pbk}\selectfont1}_\lambda,\label{a}\displaybreak[0]\\\eta\coloneqq\tikz[baseline,very thick,->,scale=0.4,yscale=-1]{\draw[->](0,0)..controls(.4,1.2)and(1.6,1.2)..(2,0)node[at start,above,text height=1.62ex,text depth=.5ex,inner ysep=2pt,scale=0.8]{$i$};\node at(-.5,.5){$\lambda$};}\colon\textnormal{\fontfamily{pbk}\selectfont1}_\lambda\to q_i^{\langle i,\lambda\rangle-1}E_iF_i\textnormal{\fontfamily{pbk}\selectfont1}_\lambda,\quad\eta'\coloneqq\tikz[baseline,very thick,->,scale=0.4,yscale=-1]{\draw[<-](0,0)..controls(.4,1.2)and(1.6,1.2)..(2,0)node[at start,above,text height=1.62ex,text depth=.5ex,inner ysep=2pt,scale=0.8]{$i$};\node at(-.5,.5){$\lambda$};}\colon\textnormal{\fontfamily{pbk}\selectfont1}_\lambda\to q_i^{-\langle i,\lambda\rangle-1}F_iE_i\textnormal{\fontfamily{pbk}\selectfont1}_\lambda.\label{b}\end{gather} Generating 2-morphisms are subject to the following relations, presented as linear combinations of diagrams, where $F_i$ (resp. $E_i$) corresponds to a downward (resp. upward) endpoint of an oriented strand, and the identity 2-morphism of $F_i$ (resp. $E_i$) is depicted as a downward (resp. upward) oriented strand; the bottom horizontal line corresponds to the source of the 2-morphism, and the top line corresponds to the target. \begin{enumerate}\item KLR relations: \begin{equation}\label{c}\begin{tikzpicture}[scale=.5,very thick,baseline,baseline=-2pt]\draw[postaction={ decorate, decoration={ markings, mark=at position .2 with {{\arrow[scale=.8]{<}}} } }](-4,0)+(-1,-1)--+(1,1)node[below,text height=1.62ex,text depth=.5ex,inner ysep=2pt,at start]{$i$};\draw[postaction={ decorate, decoration={ markings, mark=at position .2 with {{\arrow[scale=.8]{<}}} } }][postaction={ decorate, decoration={ markings, mark=at position .75 with {\fill(0,0)circle(2.5pt);} } }](-4,0)+(1,-1)--+(-1,1)node[below,text height=1.62ex,text depth=.5ex,inner ysep=2pt,at start]{$j$};\node[text height=\the\dimexpr3\fontdimen22\textfont2\relax,text depth=\the\dimexpr\fontdimen22\textfont2\relax,]at(-2,0){$-$};\draw[postaction={ decorate, decoration={ markings, mark=at position .8 with {{\arrow[scale=.8]{<}}} } }](0,0)+(-1,-1)--+(1,1)node[below,text height=1.62ex,text depth=.5ex,inner ysep=2pt,at start]{$i$};\draw[postaction={ decorate, decoration={ markings, mark=at position .8 with {{\arrow[scale=.8]{<}}} } }][postaction={ decorate, decoration={ markings, mark=at position .25 with {\fill(0,0)circle(2.5pt);} } }](0,0)+(1,-1)--+(-1,1)node[below,text height=1.62ex,text depth=.5ex,inner ysep=2pt,at start]{$j$};\node[text height=\the\dimexpr3\fontdimen22\textfont2\relax,text depth=\the\dimexpr\fontdimen22\textfont2\relax,]at(2,0){$=$};\draw[postaction={ decorate, decoration={ markings, mark=at position .8 with {{\arrow[scale=.8]{<}}} } }][postaction={ decorate, decoration={ markings, mark=at position .25 with {\fill(0,0)circle(2.5pt);} } }](4,0)+(-1,-1)--+(1,1)node[below,text height=1.62ex,text depth=.5ex,inner ysep=2pt,at start]{$i$};\draw[postaction={ decorate, decoration={ markings, mark=at position .8 with {{\arrow[scale=.8]{<}}} } }](4,0)+(1,-1)--+(-1,1)node[below,text height=1.62ex,text depth=.5ex,inner ysep=2pt,at start]{$j$};\node[text height=\the\dimexpr3\fontdimen22\textfont2\relax,text depth=\the\dimexpr\fontdimen22\textfont2\relax,]at(6,0){$-$};\draw[postaction={ decorate, decoration={ markings, mark=at position .2 with {{\arrow[scale=.8]{<}}} } }][postaction={ decorate, decoration={ markings, mark=at position .75 with {\fill(0,0)circle(2.5pt);} } }](8,0)+(-1,-1)--+(1,1)node[below,text height=1.62ex,text depth=.5ex,inner ysep=2pt,at start]{$i$};\draw[postaction={ decorate, decoration={ markings, mark=at position .2 with {{\arrow[scale=.8]{<}}} } }](8,0)+(1,-1)--+(-1,1)node[below,text height=1.62ex,text depth=.5ex,inner ysep=2pt,at start]{$j$};\node[text height=\the\dimexpr3\fontdimen22\textfont2\relax,text depth=\the\dimexpr\fontdimen22\textfont2\relax,]at(10,0){$=$};\node[text height=\the\dimexpr3\fontdimen22\textfont2\relax,text depth=\the\dimexpr\fontdimen22\textfont2\relax,]at(11,0){$\delta_{ij}$};\draw[postaction={ decorate, decoration={ markings, mark=at position .5 with {{\arrow[scale=.8]{<}}} } }](13,0)+(-1,-1)--+(-1,1)node[below,text height=1.62ex,text depth=.5ex,inner ysep=2pt,at start]{$i$};\draw[postaction={ decorate, decoration={ markings, mark=at position .5 with {{\arrow[scale=.8]{<}}} } }](13,0)+(0,-1)--+(0,1)node[below,text height=1.62ex,text depth=.5ex,inner ysep=2pt,at start]{$j$};\end{tikzpicture}\quad,\end{equation} \begin{equation}\label{d}\begin{tikzpicture}[very thick,baseline,scale=.7,baseline=-4pt]\draw[postaction={ decorate, decoration={ markings, mark=at position .5 with {{\arrow[scale=.8]{<}}} } }](-2.8,-1)..controls(-1.2,0)..(-2.8,1)node[below,text height=1.62ex,text depth=.5ex,inner ysep=2pt,at start]{$i$};\draw[postaction={ decorate, decoration={ markings, mark=at position .5 with {{\arrow[scale=.8]{<}}} } }](-1.2,-1)..controls(-2.8,0)..(-1.2,1)node[below,text height=1.62ex,text depth=.5ex,inner ysep=2pt,at start]{$j$};\node[text height=\the\dimexpr3\fontdimen22\textfont2\relax,text depth=\the\dimexpr\fontdimen22\textfont2\relax,]at(-.5,0){$=$};\draw(1.4,0)+(0,-1)--+(0,1)node[below,text height=1.62ex,text depth=.5ex,inner ysep=2pt,at start]{$i$};\draw(2.2,0)+(0,-1)--+(0,1)node[below,text height=1.62ex,text depth=.5ex,inner ysep=2pt,at start]{$j$};\node[inner xsep=7pt,fill=white,draw,inner ysep=5pt]at(1.8,0){$Q_{ij}(y_1,y_2)$};\end{tikzpicture}\quad,\qquad\begin{tikzpicture}[very thick,baseline,scale=.7,baseline=-4pt]\draw[postaction={ decorate, decoration={ markings, mark=at position .8 with {{\arrow[scale=.8]{<}}} } }](-3,0)+(-1,-1)--+(1,1)node[below,text height=1.62ex,text depth=.5ex,inner ysep=2pt,at start]{$i$};\draw[postaction={ decorate, decoration={ markings, mark=at position .2 with {{\arrow[scale=.8]{<}}} } }](-3,0)+(1,-1)--+(-1,1)node[below,text height=1.62ex,text depth=.5ex,inner ysep=2pt,at start]{$k$};\draw[postaction={ decorate, decoration={ markings, mark=at position .5 with {{\arrow[scale=.8]{<}}} } }](-3,-1)..controls(-4,0)..(-3,1)node[below,text height=1.62ex,text depth=.5ex,inner ysep=2pt,at start]{$j$};\node[text height=\the\dimexpr3\fontdimen22\textfont2\relax,text depth=\the\dimexpr\fontdimen22\textfont2\relax,]at(-1,0){=};\draw[postaction={ decorate, decoration={ markings, mark=at position .2 with {{\arrow[scale=.8]{<}}} } }](1,0)+(-1,-1)--+(1,1)node[below,text height=1.62ex,text depth=.5ex,inner ysep=2pt,at start]{$i$};\draw[postaction={ decorate, decoration={ markings, mark=at position .8 with {{\arrow[scale=.8]{<}}} } }](1,0)+(1,-1)--+(-1,1)node[below,text height=1.62ex,text depth=.5ex,inner ysep=2pt,at start]{$k$};\draw[postaction={ decorate, decoration={ markings, mark=at position .5 with {{\arrow[scale=.8]{<}}} } }](1,-1)..controls(2,0)..(1,1)node[below,text height=1.62ex,text depth=.5ex,inner ysep=2pt,at start]{$j$};\node[text height=\the\dimexpr3\fontdimen22\textfont2\relax,text depth=\the\dimexpr\fontdimen22\textfont2\relax,]at(2.8,0){$+$};\node[text height=\the\dimexpr3\fontdimen22\textfont2\relax,text depth=\the\dimexpr\fontdimen22\textfont2\relax,]at(3.7,0){$\delta_{ik}$};\draw(6.2,0)+(1,-1)--+(1,1)node[below,text height=1.62ex,text depth=.5ex,inner ysep=2pt,at start]{$i$};\draw(6.2,0)+(-1,-1)--+(-1,1)node[below,text height=1.62ex,text depth=.5ex,inner ysep=2pt,at start]{$i$};\draw(6.2,0)+(0,-1)--+(0,1)node[below,text height=1.62ex,text depth=.5ex,inner ysep=2pt,at start]{$j$};\node[inner ysep=5pt,inner xsep=7pt,fill=white,draw]at(6.2,0){$\overline{Q}_{ij}(y_1,y_2,y_3)$};\end{tikzpicture}\quad.\end{equation}\item Cyclicity: \begin{equation}\tikz[baseline,very thick,<-,scale=1.8]{\draw(0,.3)--(0,-.1)node[at end,below,text height=1.62ex,text depth=.5ex,inner ysep=2pt,scale=.8]{$i$} }\ =\ {\tikzstyle{every picture}=[very thick,baseline,scale=.6]\begin{tikzpicture}\node[inner sep=0mm](0)at(-1,0){};\node[inner sep=0mm](1)at(0,0){};\node[inner sep=0mm](2)at(1,0){};\node[inner sep=0mm](3)at(1,1.5){};\node[inner sep=0mm](4)at(-1,-1.5){};\draw[postaction={ decorate, decoration={ markings, mark=at position .5 with {{\arrow[scale=.8]{>}}} } }](4.center)node[below]{$i$}to(0.center);\draw[postaction={ decorate, decoration={ markings, mark=at position .5 with {{\arrow[scale=.8]{>}}} } }](2.center)to(3.center);\draw[bend left=90,looseness=2.50,postaction={ decorate, decoration={ markings, mark=at position .8 with {{\arrow[scale=.8]{>}}} } }](0.center)to(1.center);\draw[bend right=90,looseness=2.50](1.center)to(2.center);\end{tikzpicture}}\ =\ {\tikzstyle{every picture}=[very thick,baseline,scale=.6,xscale=-1]\begin{tikzpicture}\node[inner sep=0mm](0)at(-1,0){};\node[inner sep=0mm](1)at(0,0){};\node[inner sep=0mm](2)at(1,0){};\node[inner sep=0mm](3)at(1,1.5){};\node[inner sep=0mm](4)at(-1,-1.5){};\draw[postaction={ decorate, decoration={ markings, mark=at position .5 with {{\arrow[scale=.8]{>}}} } }](4.center)node[below]{$i$}to(0.center);\draw[postaction={ decorate, decoration={ markings, mark=at position .5 with {{\arrow[scale=.8]{>}}} } }](2.center)to(3.center);\draw[bend left=90,looseness=2.50,postaction={ decorate, decoration={ markings, mark=at position .8 with {{\arrow[scale=.8]{>}}} } }](0.center)to(1.center);\draw[bend right=90,looseness=2.50](1.center)to(2.center);\end{tikzpicture}}\ ,\qquad\tikz[baseline,very thick,<-,scale=1.8]{\draw[postaction={ decorate, decoration={ markings, mark=at position .5 with {\fill(0,0)circle(2.5pt);} } }](0,.3)--(0,-.1)node[at end,below,text height=1.62ex,text depth=.5ex,inner ysep=2pt,scale=.8]{$i$};}\ \coloneqq\ {\tikzstyle{every picture}=[very thick,baseline,scale=.6]\begin{tikzpicture}\node[inner sep=0mm](0)at(-1,0){};\node[inner sep=0mm](1)at(0,0){};\node[inner sep=0mm](2)at(1,0){};\node[inner sep=0mm](3)at(1,1.5){};\node[inner sep=0mm](4)at(-1,-1.5){};\draw[postaction={ decorate, decoration={ markings, mark=at position .5 with {{\arrow[scale=.8]{>}}} } }](4.center)node[below]{$i$}to(0.center);\draw[postaction={ decorate, decoration={ markings, mark=at position .5 with {{\arrow[scale=.8]{>}}} } }](2.center)to(3.center);\draw[bend left=90,looseness=2.50,postaction={ decorate, decoration={ markings, mark=at position .8 with {{\arrow[scale=.8]{>}}} } }][postaction={ decorate, decoration={ markings, mark=at position .95 with {\fill(0,0)circle(2.5pt);} } }](0.center)to(1.center);\draw[bend right=90,looseness=2.50](1.center)to(2.center);\end{tikzpicture}}\ =\ {\tikzstyle{every picture}=[very thick,baseline,scale=.6,xscale=-1]\begin{tikzpicture}\node[inner sep=0mm](0)at(-1,0){};\node[inner sep=0mm](1)at(0,0){};\node[inner sep=0mm](2)at(1,0){};\node[inner sep=0mm](3)at(1,1.5){};\node[inner sep=0mm](4)at(-1,-1.5){};\draw[postaction={ decorate, decoration={ markings, mark=at position .5 with {{\arrow[scale=.8]{>}}} } }](4.center)node[below]{$i$}to(0.center);\draw[postaction={ decorate, decoration={ markings, mark=at position .5 with {{\arrow[scale=.8]{>}}} } }](2.center)to(3.center);\draw[bend left=90,looseness=2.50,postaction={ decorate, decoration={ markings, mark=at position .8 with {{\arrow[scale=.8]{>}}} } }][postaction={ decorate, decoration={ markings, mark=at position .95 with {\fill(0,0)circle(2.5pt);} } }](0.center)to(1.center);\draw[bend right=90,looseness=2.50](1.center)to(2.center);\end{tikzpicture}}\label{e}\end{equation} \begin{equation}\tikz[baseline,scale=1.5]{\draw[very thick,baseline,<-](.2,.3)--(-.2,-.1)node[at end,below,text height=1.62ex,text depth=.5ex,inner ysep=2pt,scale=.8]{$i$};\draw[very thick,baseline,<-](-.2,.3)--(.2,-.1)node[at end,below,text height=1.62ex,text depth=.5ex,inner ysep=2pt,scale=.8]{$j$};}\ \coloneqq\ {\tikzstyle{every picture}=[very thick,baseline,scale=.3]\begin{tikzpicture}\node[inner sep=0mm](0)at(-1,1){};\node[inner sep=0mm](1)at(1,1){};\node[inner sep=0mm](2)at(-1,-1){};\node[inner sep=0mm](3)at(1,-1){};\node[inner sep=0mm](4)at(-5,-1){};\node[inner sep=0mm](5)at(-3,-1){};\node[inner sep=0mm](6)at(3,1){};\node[inner sep=0mm](7)at(5,1){};\node[inner sep=0mm](8)at(-5,3){};\node[inner sep=0mm](9)at(-3,3){};\node[inner sep=0mm](10)at(3,-3){};\node[inner sep=0mm](11)at(5,-3){};\draw[postaction={ decorate, decoration={ markings, mark=at position .5 with {{\arrow[scale=.8]{<}}} } }](8.center)to(4.center);\draw[postaction={ decorate, decoration={ markings, mark=at position .5 with {{\arrow[scale=.8]{<}}} } }](9.center)to(5.center);\draw[postaction={ decorate, decoration={ markings, mark=at position .5 with {{\arrow[scale=.8]{>}}} } }](10.center)to(6.center);\draw[postaction={ decorate, decoration={ markings, mark=at position .5 with {{\arrow[scale=.8]{>}}} } }](11.center)to(7.center);\draw[bend right=90,looseness=1.50,postaction={ decorate, decoration={ markings, mark=at position .5 with {{\arrow[scale=.8]{<}}} } }](4.center)to(3.center);\draw[bend right=90,looseness=1.50,postaction={ decorate, decoration={ markings, mark=at position .5 with {{\arrow[scale=.8]{<}}} } }](5.center)to(2.center);\draw[bend left=90,looseness=1.50,postaction={ decorate, decoration={ markings, mark=at position .5 with {{\arrow[scale=.8]{<}}} } }](0.center)to(7.center);\draw[bend left=90,looseness=1.50,postaction={ decorate, decoration={ markings, mark=at position .5 with {{\arrow[scale=.8]{<}}} } }](1.center)to(6.center);\draw[in=-90,out=90,looseness=0.75,postaction={ decorate, decoration={ markings, mark=at position .8 with {{\arrow[scale=.8]{<}}} } }](2.center)to(1.center);\draw[in=-90,out=90,looseness=0.75,postaction={ decorate, decoration={ markings, mark=at position .8 with {{\arrow[scale=.8]{<}}} } }](3.center)to(0.center);\end{tikzpicture}}\ =\ {\tikzstyle{every picture}=[very thick,baseline,scale=.3,xscale=-1]\begin{tikzpicture}\node[inner sep=0mm](0)at(-1,1){};\node[inner sep=0mm](1)at(1,1){};\node[inner sep=0mm](2)at(-1,-1){};\node[inner sep=0mm](3)at(1,-1){};\node[inner sep=0mm](4)at(-5,-1){};\node[inner sep=0mm](5)at(-3,-1){};\node[inner sep=0mm](6)at(3,1){};\node[inner sep=0mm](7)at(5,1){};\node[inner sep=0mm](8)at(-5,3){};\node[inner sep=0mm](9)at(-3,3){};\node[inner sep=0mm](10)at(3,-3){};\node[inner sep=0mm](11)at(5,-3){};\draw[postaction={ decorate, decoration={ markings, mark=at position .5 with {{\arrow[scale=.8]{<}}} } }](8.center)to(4.center);\draw[postaction={ decorate, decoration={ markings, mark=at position .5 with {{\arrow[scale=.8]{<}}} } }](9.center)to(5.center);\draw[postaction={ decorate, decoration={ markings, mark=at position .5 with {{\arrow[scale=.8]{>}}} } }](10.center)to(6.center);\draw[postaction={ decorate, decoration={ markings, mark=at position .5 with {{\arrow[scale=.8]{>}}} } }](11.center)to(7.center);\draw[bend right=90,looseness=1.50,postaction={ decorate, decoration={ markings, mark=at position .5 with {{\arrow[scale=.8]{<}}} } }](4.center)to(3.center);\draw[bend right=90,looseness=1.50,postaction={ decorate, decoration={ markings, mark=at position .5 with {{\arrow[scale=.8]{<}}} } }](5.center)to(2.center);\draw[bend left=90,looseness=1.50,postaction={ decorate, decoration={ markings, mark=at position .5 with {{\arrow[scale=.8]{<}}} } }](0.center)to(7.center);\draw[bend left=90,looseness=1.50,postaction={ decorate, decoration={ markings, mark=at position .5 with {{\arrow[scale=.8]{<}}} } }](1.center)to(6.center);\draw[in=-90,out=90,looseness=0.75,postaction={ decorate, decoration={ markings, mark=at position .8 with {{\arrow[scale=.8]{<}}} } }](2.center)to(1.center);\draw[in=-90,out=90,looseness=0.75,postaction={ decorate, decoration={ markings, mark=at position .8 with {{\arrow[scale=.8]{<}}} } }](3.center)to(0.center);\end{tikzpicture}}\label{f}\end{equation}\item Infinite Grassmannian relations: \begin{equation}\label{g}\begin{tikzpicture}[very thick,baseline,baseline=-2.5pt]\draw[postaction={ decorate, decoration={ markings, mark=at position .5 with {{\arrow[scale=.8]{>}}} } }][postaction={ decorate, decoration={ markings, mark=at position .999 with {\fill(0,0)circle(2.5pt);\node[right,xshift=2pt,scale=.8]at(0,0){$k$};} } }](0,0)circle(15pt);\node at(-.7,-.3){$\lambda$};\end{tikzpicture}\ =\begin{cases}0&k<\langle i,\lambda\rangle-1\\c_i(\lambda)&k=\langle i,\lambda\rangle-1\end{cases}\ ,\quad\begin{tikzpicture}[very thick,baseline,baseline=-2.5pt]\draw[postaction={ decorate, decoration={ markings, mark=at position .5 with {{\arrow[scale=.8]{<}}} } }][postaction={ decorate, decoration={ markings, mark=at position .999 with {\fill(0,0)circle(2.5pt);\node[right,xshift=2pt,scale=.8]at(0,0){$k$};} } }](0,0)circle(15pt);\node at(-.7,-.3){$\lambda$};\end{tikzpicture}\ =\begin{cases}0&k<-\langle i,\lambda\rangle-1\\c_i(\lambda)^{-1}&k=-\langle i,\lambda\rangle-1\end{cases}\end{equation} \begin{equation}\label{h}\begin{tikzpicture}[baseline,very thick,baseline]\node at(-1,0){$\displaystyle\sum_{k=\langle i,\lambda\rangle-1}^{n+\langle i,\lambda\rangle+1}$};\draw[postaction={ decorate, decoration={ markings, mark=at position .5 with {{\arrow[scale=.8]{>}}} } }][postaction={ decorate, decoration={ markings, mark=at position .999 with {\fill(0,0)circle(2.5pt);\node[right,xshift=2pt,scale=.8]at(0,0){$k$};} } }](.5,0)circle(15pt);\node at(1.375,-.6){$\lambda$};\draw[postaction={ decorate, decoration={ markings, mark=at position .5 with {{\arrow[scale=.8]{<}}} } }][postaction={ decorate, decoration={ markings, mark=at position .001 with {\fill(0,0)circle(2.5pt);\node[right,xshift=2pt,scale=.8]at(0,0){$n-k$};} } }](2.25,0)circle(15pt);\node at(5.7,0){$=\begin{cases}1&n=-2\\0&n>-2,\end{cases}$};\end{tikzpicture}\end{equation} where all strands are colored $i\in I$, and a ``bubble'' with a negative number of dots (called a \emph{fake bubble}) is understood as a formal element of $\mathscr{U}(\lambda,\lambda)(\textnormal{\fontfamily{pbk}\selectfont1}_\lambda,q_i^d\textnormal{\fontfamily{pbk}\selectfont1}_\lambda)$. \eqref{g} and \eqref{h} allow one to express such bubbles in terms of bubbles with a nonnegative number of dots.\item Commutator relations: Define sideways crossings as follows. \begin{equation}\sigma\coloneqq\tikz[baseline,scale=1.5]{\draw[very thick,baseline,<-](.2,.3)--(-.2,-.1)node[at end,below,text height=1.62ex,text depth=.5ex,inner ysep=2pt,scale=.8]{$i$};\draw[very thick,baseline,->](-.2,.3)--(.2,-.1)node[at end,below,text height=1.62ex,text depth=.5ex,inner ysep=2pt,scale=.8]{$j$};}\ \coloneqq\ {\tikzstyle{every picture}=[very thick,baseline,scale=.3]\begin{tikzpicture}\node[inner sep=0mm](0)at(-1,1){};\node[inner sep=0mm](1)at(1,1){};\node[inner sep=0mm](2)at(-1,-1){};\node[inner sep=0mm](3)at(1,-1){};\node[inner sep=0mm](4)at(-3,1){};\node[inner sep=0mm](5)at(3,-1){};\node[inner sep=0mm](6)at(-1,-2.5){};\node[inner sep=0mm](7)at(-3,-2.5){};\node[inner sep=0mm](8)at(1,2.5){};\node[inner sep=0mm](9)at(3,2.5){};\draw[postaction={ decorate, decoration={ markings, mark=at position .6 with {{\arrow[scale=.8]{>}}} } }](7.center)to(4.center);\draw[postaction={ decorate, decoration={ markings, mark=at position .4 with {{\arrow[scale=.8]{>}}} } }](5.center)to(9.center);\draw[postaction={ decorate, decoration={ markings, mark=at position .9 with {{\arrow[scale=.8]{<}}} } }](6.center)to(2.center);\draw[postaction={ decorate, decoration={ markings, mark=at position .2 with {{\arrow[scale=.8]{<}}} } }](1.center)to(8.center);\draw[bend left=90,looseness=1.50](4.center)to(0.center);\draw[bend right=90,looseness=1.50](3.center)to(5.center);\draw[in=90,out=-90,looseness=0.75](0.center)to(3.center);\draw[in=-90,out=90,looseness=0.75](2.center)to(1.center);\end{tikzpicture}}\ ,\qquad\sigma'\coloneqq\tikz[baseline,scale=1.5]{\draw[very thick,baseline,->](.2,.3)--(-.2,-.1)node[at end,below,text height=1.62ex,text depth=.5ex,inner ysep=2pt,scale=.8]{$i$};\draw[very thick,baseline,<-](-.2,.3)--(.2,-.1)node[at end,below,text height=1.62ex,text depth=.5ex,inner ysep=2pt,scale=.8]{$j$};}\ \coloneqq\ {\tikzstyle{every picture}=[very thick,baseline,scale=.3,xscale=-1]\begin{tikzpicture}\node[inner sep=0mm](0)at(-1,1){};\node[inner sep=0mm](1)at(1,1){};\node[inner sep=0mm](2)at(-1,-1){};\node[inner sep=0mm](3)at(1,-1){};\node[inner sep=0mm](4)at(-3,1){};\node[inner sep=0mm](5)at(3,-1){};\node[inner sep=0mm](6)at(-1,-2.5){};\node[inner sep=0mm](7)at(-3,-2.5){};\node[inner sep=0mm](8)at(1,2.5){};\node[inner sep=0mm](9)at(3,2.5){};\draw[postaction={ decorate, decoration={ markings, mark=at position .6 with {{\arrow[scale=.8]{>}}} } }](7.center)to(4.center);\draw[postaction={ decorate, decoration={ markings, mark=at position .4 with {{\arrow[scale=.8]{>}}} } }](5.center)to(9.center);\draw[postaction={ decorate, decoration={ markings, mark=at position .9 with {{\arrow[scale=.8]{<}}} } }](6.center)to(2.center);\draw[postaction={ decorate, decoration={ markings, mark=at position .2 with {{\arrow[scale=.8]{<}}} } }](1.center)to(8.center);\draw[bend left=90,looseness=1.50](4.center)to(0.center);\draw[bend right=90,looseness=1.50](3.center)to(5.center);\draw[in=90,out=-90,looseness=0.75](0.center)to(3.center);\draw[in=-90,out=90,looseness=0.75](2.center)to(1.center);\end{tikzpicture}}.\label{i}\end{equation} Then the commutator relations are given by \begin{equation}\begin{tikzpicture}[very thick,baseline,baseline=-2pt]\draw[postaction={ decorate, decoration={ markings, mark=at position .5 with {{\arrow[scale=.8]{<}}} } }](0,-1)to[out=90,in=-90]node[at start,below,text height=1.62ex,text depth=.5ex,inner ysep=2pt]{$i$}(1,0)to[out=90,in=-90](0,1);\draw[postaction={ decorate, decoration={ markings, mark=at position .5 with {{\arrow[scale=.8]{>}}} } }](1,-1)to[out=90,in=-90]node[at start,below,text height=1.62ex,text depth=.5ex,inner ysep=2pt]{$j$}(0,0)to[out=90,in=-90](1,1);\end{tikzpicture}\ =\ \begin{tikzpicture}[very thick,baseline,baseline=-2pt]\draw[postaction={ decorate, decoration={ markings, mark=at position .5 with {{\arrow[scale=.8]{<}}} } }](1,-1)to[out=90,in=-90]node[at start,below,text height=1.62ex,text depth=.5ex,inner ysep=2pt]{$i$}(1,1);\draw[postaction={ decorate, decoration={ markings, mark=at position .5 with {{\arrow[scale=.8]{>}}} } }](1.7,-1)to[out=90,in=-90]node[at start,below,text height=1.62ex,text depth=.5ex,inner ysep=2pt]{$j$}(1.7,1);\end{tikzpicture}\ ,\qquad\begin{tikzpicture}[very thick,baseline,baseline=-2pt]\draw[postaction={ decorate, decoration={ markings, mark=at position .5 with {{\arrow[scale=.8]{>}}} } }](0,-1)to[out=90,in=-90]node[at start,below,text height=1.62ex,text depth=.5ex,inner ysep=2pt]{$i$}(1,0)to[out=90,in=-90](0,1);\draw[postaction={ decorate, decoration={ markings, mark=at position .5 with {{\arrow[scale=.8]{<}}} } }](1,-1)to[out=90,in=-90]node[at start,below,text height=1.62ex,text depth=.5ex,inner ysep=2pt]{$j$}(0,0)to[out=90,in=-90](1,1);\end{tikzpicture}\ =\ \begin{tikzpicture}[very thick,baseline,baseline=-2pt]\draw[postaction={ decorate, decoration={ markings, mark=at position .5 with {{\arrow[scale=.8]{>}}} } }](1,-1)to[out=90,in=-90]node[at start,below,text height=1.62ex,text depth=.5ex,inner ysep=2pt]{$i$}(1,1);\draw[postaction={ decorate, decoration={ markings, mark=at position .5 with {{\arrow[scale=.8]{<}}} } }](1.7,-1)to[out=90,in=-90]node[at start,below,text height=1.62ex,text depth=.5ex,inner ysep=2pt]{$j$}(1.7,1);\end{tikzpicture}\qquad\text{for }i\ne j,\label{j}\end{equation} and in the case of $i=j$, \begin{equation}\label{k}\begin{tikzpicture}[very thick,baseline,baseline=-2pt]\draw[postaction={ decorate, decoration={ markings, mark=at position .5 with {{\arrow[scale=.8]{<}}} } }](0,-1)node[below]{$i$}to[out=90,in=-90](1,0)to[out=90,in=-90](0,1);\draw[postaction={ decorate, decoration={ markings, mark=at position .5 with {{\arrow[scale=.8]{>}}} } }](1,-1)node[below]{$i$}to[out=90,in=-90](0,0)to[out=90,in=-90](1,1);\end{tikzpicture}=\begin{tikzpicture}[very thick,baseline,baseline=-2pt]\node[text height=\the\dimexpr3\fontdimen22\textfont2\relax,text depth=\the\dimexpr\fontdimen22\textfont2\relax,]at(.7,0){$-$};\draw[postaction={ decorate, decoration={ markings, mark=at position .5 with {{\arrow[scale=.8]{<}}} } }](1,-1)to[out=90,in=-90](1,1);\draw[postaction={ decorate, decoration={ markings, mark=at position .5 with {{\arrow[scale=.8]{>}}} } }](1.7,-1)to[out=90,in=-90](1.7,1);\end{tikzpicture}\ +\begin{tikzpicture}[very thick,baseline,baseline=-2pt]\node at(3,0){$\displaystyle\sum_{a+b+c=-2}$};\draw[postaction={ decorate, decoration={ markings, mark=at position .5 with {{\arrow[scale=.8]{<}}} } }][postaction={ decorate, decoration={ markings, mark=at position .8 with {\fill(0,0)circle(2.5pt);\node[right,xshift=2pt,scale=.8]at(0,0){$a$}; } } }](4,-1)to[out=90,in=-180](4.5,-.5)to[out=0,in=90](5,-1);\draw[postaction={ decorate, decoration={ markings, mark=at position .5 with {{\arrow[scale=.8]{>}}} } }][postaction={ decorate, decoration={ markings, mark=at position .8 with {\fill(0,0)circle(2.5pt);\node[right,xshift=2pt,scale=.8]at(0,0){$c$}; } } }](4,1)to[out=-90,in=-180](4.5,.5)to[out=0,in=-90](5,1);\draw[postaction={ decorate, decoration={ markings, mark=at position .5 with {{\arrow[scale=.8]{<}}} } }][postaction={ decorate, decoration={ markings, mark=at position .999 with {\fill(0,0)circle(2.5pt);\node[right,xshift=2pt,scale=.8]at(0,0){$b$}; } } }](5.5,0)circle(10pt);\end{tikzpicture}\end{equation} \begin{equation}\label{l}\begin{tikzpicture}[very thick,baseline,baseline=-2pt]\draw[postaction={ decorate, decoration={ markings, mark=at position .5 with {{\arrow[scale=.8]{>}}} } }](0,-1)node[below]{$i$}to[out=90,in=-90](1,0)to[out=90,in=-90](0,1);\draw[postaction={ decorate, decoration={ markings, mark=at position .5 with {{\arrow[scale=.8]{<}}} } }](1,-1)node[below]{$i$}to[out=90,in=-90](0,0)to[out=90,in=-90](1,1);\end{tikzpicture}=\begin{tikzpicture}[very thick,baseline,baseline=-2pt]\node[text height=\the\dimexpr3\fontdimen22\textfont2\relax,text depth=\the\dimexpr\fontdimen22\textfont2\relax,]at(.7,0){$-$};\draw[postaction={ decorate, decoration={ markings, mark=at position .5 with {{\arrow[scale=.8]{>}}} } }](1,-1)to[out=90,in=-90](1,1);\draw[postaction={ decorate, decoration={ markings, mark=at position .5 with {{\arrow[scale=.8]{<}}} } }](1.7,-1)to[out=90,in=-90](1.7,1);\end{tikzpicture}\ +\begin{tikzpicture}[very thick,baseline,baseline=-2pt]\node at(3,0){$\displaystyle\sum_{a+b+c=-2}$};\draw[postaction={ decorate, decoration={ markings, mark=at position .5 with {{\arrow[scale=.8]{>}}} } }][postaction={ decorate, decoration={ markings, mark=at position .8 with {\fill(0,0)circle(2.5pt);\node[right,xshift=2pt,scale=.8]at(0,0){$a$}; } } }](4,-1)to[out=90,in=-180](4.5,-.5)to[out=0,in=90](5,-1);\draw[postaction={ decorate, decoration={ markings, mark=at position .5 with {{\arrow[scale=.8]{<}}} } }][postaction={ decorate, decoration={ markings, mark=at position .8 with {\fill(0,0)circle(2.5pt);\node[right,xshift=2pt,scale=.8]at(0,0){$c$}; } } }](4,1)to[out=-90,in=-180](4.5,.5)to[out=0,in=-90](5,1);\draw[postaction={ decorate, decoration={ markings, mark=at position .5 with {{\arrow[scale=.8]{>}}} } }][postaction={ decorate, decoration={ markings, mark=at position .999 with {\fill(0,0)circle(2.5pt);\node[right,xshift=2pt,scale=.8]at(0,0){$b$}; } } }](5.5,0)circle(10pt);\end{tikzpicture}\end{equation} where the sum is over $a,b,c\in\mathbb{Z}$ making the corresponding diagram in the summand well-defined. For instance, if \eqref{k} is understood as an identity of 2-morphisms $E_iF_i\textnormal{\fontfamily{pbk}\selectfont1}_\lambda\to F_iE_i\textnormal{\fontfamily{pbk}\selectfont1}_\lambda$, the sum is over $a,c\ge0$ and $b\ge-\langle i,\lambda\rangle-1$.\end{enumerate}\end{definition}\par We write $\mathscr{U}(c)$ for the 2-category $\mathscr{U}$ associated with a choice of bubble parameters $c$, with the underlying Cartan datum, weight datum, and quiver Hecke datum understood implicitly. For any two choices of bubble parameters $c$ and $c'$ with the other data fixed, there exists a canonical isomorphism of 2-categories $\mathscr{U}(c)\to\mathscr{U}(c')$ that is the identity on objects and 1-morphisms \cite[Thm.~2.1]{b}.\par Let $(-X,\{\langle i,-\rangle\}_{i\in I})$ be the weight datum given by $-X\coloneqq\left\{\kern.07em-\lambda\,\middle|\,\lambda\in X\kern.07em\right\}$, $-\lambda+\nu\coloneqq-(\lambda-\nu)$ for $\nu\in\mathbb{Z}[I]$, and $\langle i,-\lambda\rangle\coloneqq-\langle i,\lambda\rangle$. Let $c^\ast=\{c^\ast_i\}_{i\in I}$ be the choice of bubble parameters associated to this weight datum, given by $c^\ast_i(-\lambda)=c_i(\lambda)^{-1}$ for all $\lambda\in X$.\par\begin{definition}\thlabel{m} \leavevmode \begin{enumerate}[(a)]\item The Chevalley involution $\psi\colon\mathscr{U}(c)\to\mathscr{U}(c^\ast)$ is the 2-functor defined on objects by $\lambda\mapsto-\lambda$; on generating 1-morphisms by $q\mapsto q$, $E_i\mapsto F_i$, and $F_i\mapsto E_i$; and on 2-morphisms by reversing the orientation of all strands and multiplying by $(-1)^{\#\text{crossings}}$.\label{n}\item The antiautomorphism $\ast\colon\mathscr{U}(c)\to\mathscr{U}(c^\ast)^{\operatorname{op}}$ is the 2-functor defined on objects by $\lambda\mapsto-\lambda$; on generating 1-morphisms by $q\mapsto q$, $E_i\mapsto E_i$, and $F_i\mapsto F_i$; and on 2-morphisms by reflecting the diagram across the vertical axis and multiplying by $(-1)^{\#\text{crossings}}$.\label{o}\item The bar-involution $\mathop{-}\nolimits\colon\mathscr{U}(c)\to\mathscr{U}(c)^{\operatorname{co}}$ is the 2-functor defined on objects by $\lambda\mapsto\lambda$; on generating 1-morphisms by $q\mapsto q^{-1}$, $E_i\mapsto E_i$, and $F_i\mapsto F_i$; and on 2-morphisms by reflecting the diagram across the horizontal axis and reversing the orientation of all strands.\label{p}\item The involution $\rho\colon\mathscr{U}(c)\to\mathscr{U}(c)^{\operatorname{op}}$ is the 2-functor defined on objects by $\lambda\mapsto\lambda$; on generating 1-morphisms by \[q\mapsto q,\quad E_i\textnormal{\fontfamily{pbk}\selectfont1}_\lambda\mapsto q_i^{1+\langle i,\lambda\rangle}F_i\textnormal{\fontfamily{pbk}\selectfont1}_{\lambda+i}\,,\quad F_i\textnormal{\fontfamily{pbk}\selectfont1}_\lambda\mapsto q_i^{1-\langle i,\lambda\rangle}E_i\textnormal{\fontfamily{pbk}\selectfont1}_{\lambda-i}\,;\] and on 2-morphisms by reflecting the diagram across the vertical axis and reversing the orientation of all strands.\item The involution $\tau\colon\mathscr{U}(c)\to\mathscr{U}(c)^{\operatorname{coop}}$ is defined as the composition $\tau\coloneqq\psi^{\operatorname{coop}}\circ\ast^{\operatorname{co}}\circ\mathop{-}\nolimits$. It acts on 2-morphisms by rotating the corresponding diagram by $180^\circ$.\label{q}\end{enumerate}\end{definition}\par\begin{remark}\thlabel{r} The 2-functor $\rho$ categorifies the anti-automorphism of $U$ denoted by the same symbol $\rho$ in \cite[\S 19.1.1]{a}.\end{remark}\par For $u\in\mathscr{U}(\mu,\nu)$, let $\varepsilon_u\colon\rho(\overline{u})u\to\textnormal{\fontfamily{pbk}\selectfont1}_{\mu}$ and $\eta_u\colon\textnormal{\fontfamily{pbk}\selectfont1}_{\nu}\to u\rho(\overline{u})$ be the cap and cup diagrams connecting $u$ and $\rho(\overline{u})$, that is, if $u=E_i\textnormal{\fontfamily{pbk}\selectfont1}_\mu$ (resp. $F_i\textnormal{\fontfamily{pbk}\selectfont1}_\mu$), then $(\varepsilon_u,\eta_u)=(\varepsilon,\eta)$ (resp. $(\varepsilon',\eta')$), and for $u=u_2u_1$, $\varepsilon_u$ (resp. $\eta_u$) is given by nesting $\varepsilon_{u_1}$ above $\varepsilon_{u_2}$ (resp. $\eta_{u_2}$ above $\eta_{u_1}$). Note that $\varepsilon_u$ and $\eta_u$ are the counit and unit for the adjunction $\rho(\overline{u})\dashv u$ in $\mathscr{U}$.\par For $\lambda\in X$, $\epsilon\in\{-,+\}$, and $\nu\in\mathbb{N}[I]$, there exists a graded $\Bbbk$-linear functor $\iota^{\epsilon,\lambda}_{\nu}\colon\mathcal{R}^\epsilon(\nu)\to\mathscr{U}(\lambda,\lambda+\epsilon\nu)$ given by sending an object $q^di_1\otimes\cdots\otimes i_m$ to $q^dE^\epsilon_{i_m}\cdots E^\epsilon_{i_1}\textnormal{\fontfamily{pbk}\selectfont1}_\lambda$, and sending a morphism $1\otimes\cdots\otimes x_{i_k}\otimes\cdots\otimes1$ (resp. $1\otimes\cdots\otimes\tau_{i_{k+1}i_k}\otimes\cdots\otimes1$) to a 2-morphism represented by a tangle diagram consisting of $m$ straight strands with a dot at the $k$-th strand from the left (resp. with a crossing between the $k$-th and $(k+1)$-th strands from the left). To simplify notation, for $u\in\mathcal{R}^\epsilon(\nu)$, we often write $\iota^{\epsilon,\lambda}_\nu(u)$ as $u\textnormal{\fontfamily{pbk}\selectfont1}_\lambda$, or simply $u$ when $\lambda$ is clear from the context.\par\subsection{Categorification of divided powers}\label{s}\par We recall the results of \cite{T} on the categorified divided powers. Consider a Cartan datum $(I=\{\ast\},\cdot)$ of type $\mathfrak{sl}_2$. The nilHecke algebra $\mathit{NH}_n$ can be realized as $\mathit{NH}_n\coloneqq\mathcal{R}^+|_{q=1}(\ast^{\otimes n},\ast^{\otimes n})$. Set $x_k\coloneqq1^{\otimes(n-k)}\otimes x_1\otimes1^{\otimes(k-1)}$ $(1\le k\le n)$ and $\tau_k\coloneqq1^{\otimes(n-k-1)}\otimes\tau_{11}\otimes1^{\otimes(k-1)}$ $(1\le k\le n-1)$. Then $e_n\coloneqq x_1^{n-1}\cdots x_{n-1}\tau_{\omega_n}\in\mathit{NH}_n$ and $\overline{e_n}=\tau_{\omega_n}x_1^{n-1}\cdots x_{n-1}$ (see \thref{U} \ref{W}), are idempotents, where $\tau_{\omega_n}$ is the product of $\tau_k$'s corresponding to a reduced expression of the longest element of $S_n$.\par We define 1-morphisms \begin{equation}\begin{matrix*}[l]E_i^{(n)}\textnormal{\fontfamily{pbk}\selectfont1}_\lambda&\coloneqq\Big(q_i^{-\binom{n}{2}}E_i^{n}\textnormal{\fontfamily{pbk}\selectfont1}_\lambda,e_n\Big)&\cong\Big(q_i^{\binom{n}{2}}E_i^{n}\textnormal{\fontfamily{pbk}\selectfont1}_\lambda,\overline{e_n}\Big)&\in\ \dot{\mathscr{U}}(\lambda,\lambda+ni),\\F_i^{(n)}\textnormal{\fontfamily{pbk}\selectfont1}_\lambda&\coloneqq\Big(q_i^{\binom{n}{2}}F_i^{n}\textnormal{\fontfamily{pbk}\selectfont1}_\lambda,\tau(e_n)\Big)&\cong\Big(q_i^{-\binom{n}{2}}F_i^{n}\textnormal{\fontfamily{pbk}\selectfont1}_\lambda,\tau(\overline{e_n})\Big)&\in\ \dot{\mathscr{U}}(\lambda,\lambda-ni).\end{matrix*}\label{t}\end{equation} Here, $e_n\in\mathit{NH}_n$ is regarded as an element of $\operatorname{End}_q(E_i^n\textnormal{\fontfamily{pbk}\selectfont1}_\lambda)$ via the embedding $\mathit{NH}_n\to\mathcal{R}^+(n\cdot i)(i^{\otimes n},i^{\otimes n})\xrightarrow{\iota^{+,\lambda}_{n\cdot i}}\operatorname{End}_q(E_i^n\textnormal{\fontfamily{pbk}\selectfont1}_\lambda)$, where the first map is the obvious map scaling degrees by $\frac{i\cdot i}{2}$. Then we have an isomorphism $E_i^{(a)}E_i^{(b)}\textnormal{\fontfamily{pbk}\selectfont1}_\lambda\cong\genfrac{[}{]}{0pt}{}{a+b}{a}_iE_i^{(a+b)}\textnormal{\fontfamily{pbk}\selectfont1}_\lambda$ (\cite[Thm.~5.1.1]{T}), and a categorification of the commutator relation \begin{equation}E^{(a)}_iF^{(b)}_i\textnormal{\fontfamily{pbk}\selectfont1}_\lambda\cong\bigoplus_{j=0}^{\min(a,b)}\genfrac{[}{]}{0pt}{}{\langle i,\lambda\rangle-(b-a)}{j}_iF^{(b-j)}_iE^{(a-j)}_i\textnormal{\fontfamily{pbk}\selectfont1}_\lambda\qquad\text{if }\langle i,\lambda\rangle\ge b-a.\label{u}\end{equation} where the explicit isomorphism is given by the Sto\v si\'c formula \cite[Thm.~5.2.5]{T}.\par For $(u_1,e_1),(u_2,e_2)\in\dot{\mathscr{U}}(\mu,\nu)$ and $f\in\mathscr{U}(\mu,\nu)(u_1,u_2)$ satisfying $e_2fe_1=f$, we denote the 2-morphism $(u_1,e_1)\to(u_2,e_2)$ in $\dot{\mathscr{U}}(\mu,\nu)$ induced by $f$ by the same symbol $f$, by abuse of notation.\par For a 1-morphism $(u,e)\in\dot{\mathscr{U}}(\mu,\nu)$ with $u\in\mathscr{U}(\mu,\nu)$ and an idempotent $e\colon u\to u$, we set $\varepsilon_{(u,e)}=\varepsilon_u\circ(1_{\rho(\overline{u})}\mathbin{\boldsymbol{\cdot}}e)=\varepsilon_u\circ(\rho(\overline{e})\mathbin{\boldsymbol{\cdot}}1_u)$, and $\eta_{(u,e)}=(e\mathbin{\boldsymbol{\cdot}}1_{\rho(\overline{u})})\circ\eta_u=(1_u\mathbin{\boldsymbol{\cdot}}\rho(\overline{e}))\circ\eta_u$, which give the counit and unit for the adjunction $\rho(\overline{(u,e)})\dashv(u,e)$. We have \begin{equation}\label{v}\rho(E_i^{(n)}\textnormal{\fontfamily{pbk}\selectfont1}_\lambda)=q_i^{n(n+\langle i,\lambda\rangle)}F_i^{(n)}\textnormal{\fontfamily{pbk}\selectfont1}_{\lambda+ni},\quad\rho(F_i^{(n)}\textnormal{\fontfamily{pbk}\selectfont1}_\lambda)=q_i^{n(n-\langle i,\lambda\rangle)}E_i^{(n)}\textnormal{\fontfamily{pbk}\selectfont1}_{\lambda-ni}.\end{equation}\section{$\mathscr{U}$-actions on abelian categories}\label{D}\subsection{Basic properties of categorical $\mathfrak{sl}_2$-actions}Let $\mathscr{U}(\mathfrak{sl}_2)$ denote the 2-category $\mathscr{U}$ defined over $\mathbb{Z}$, associated to the Cartan datum for $\mathfrak{sl}_2$, the weight datum induced from its root datum (with $X=\mathbb{Z}$), the trivial quiver Hecke datum $Q=(0)$, and the bubble parameters $c_\ast(n)=1$ for all $n\in\mathbb{Z}$ (where $I=\{\ast\}$).\par We fix a pseudofunctor $\mathcal{C}\colon\mathscr{U}(\mathfrak{sl}_2)\to\mathrm{Ab}$ throughout this and the next subsection. Note that the functors $E\colon\mathcal{C}_n\to\mathcal{C}_{n+2}$ and $F\colon\mathcal{C}_{n}\to\mathcal{C}_{n-2}$ induced by 1-morphisms of $\mathscr{U}(\mathfrak{sl}_2)$ are exact, since they have both left and right adjoints. For $M\in\mathcal{C}_n$, let us denote \begin{equation}\varepsilon(M)\coloneqq\sup\left\{\kern.07emm\in\mathbb{Z}_{\ge0}\,\middle|\,E^mM\ne0\kern.07em\right\},\quad\varphi(M)\coloneqq\sup\left\{\kern.07emm\in\mathbb{Z}_{\ge0}\,\middle|\,F^mM\ne0\kern.07em\right\}.\label{w}\end{equation} We also denote $d(M)\coloneqq\varepsilon(M)+\varphi(M)$.\par\begin{lemma}\thlabel{x} Let $M\in\mathcal{C}_n$ with $d(M)<\infty$. Then we have \begin{enumerate*}[(i),font=\normalfont,series=sl2_rep_basic_properties]\item $\varphi(M)-\varepsilon(M)=n$ and \label{y}\item $d(N)\le d(M)$ for any subquotient $N$ of $M$. \label{z}\end{enumerate*}Furthermore, for any 1-morphism $u$ of $\mathscr{U}(\mathfrak{sl}_2)$, \begin{enumerate}[(i),font=\normalfont,resume=sl2_rep_basic_properties]\item If $uM\ne0$, then $d(uM)=d(M)$.\label{!}\item If $M$ is simple, then $d(L)=d(M)$ holds for any subobject or quotient $L$ of $uM$. \label{?}\end{enumerate}\end{lemma}\par\begin{proof}We have $n\le\varphi(M)$. Indeed, this is immediate if $n<0$, and if $n\ge0$, $M$ is a direct summand of $E^nF^nM$ by \eqref{u}, so $F^nM\ne0$. By \eqref{u} applied to $\mathcal{C}\circ\psi$, we have $F^{\varphi(M)+\varepsilon(M)+1}E^{\varepsilon(M)}M=0$. Thus, $n+2\varepsilon(M)\le\varphi(E^{\varepsilon(M)}M)\le\varphi(M)+\varepsilon(M)$, so $n\le\varphi(M)-\varepsilon(M)$. The same argument applied to $\mathcal{C}\circ\psi$ gives $-n\le\varepsilon(M)-\varphi(M)$, proving \ref{y}. If $EM\ne0$, then $\varepsilon(EM)=\varepsilon(M)-1$. Together with \ref{y}, this implies $d(EM)=d(M)$. Similarly, $d(FM)=d(M)$ if $FM\ne0$, and an iterated application of these statements proves \ref{!}. \ref{z} follows from the exactness of $E$ and $F$. Suppose that $M$ is simple and $L$ is a subobject of $uM$. By adjunction, there exists a nonzero morphism $\rho(\overline{u})L\to M$, which is surjective since $M$ is simple. This and \ref{z} imply $d(M)\le d(\rho(\overline{u})L)$. Also, we have $d(\rho(\overline{u})L)=d(L)\le d(uM)=d(M)$ by \ref{z} and \ref{!}, so we have $d(L)=d(M)$. The case of a quotient can be proved similarly.\end{proof}\par We say that $M\in\mathcal{C}_n$ is \emph{integrable} if $x\colon q^2F\textnormal{\fontfamily{pbk}\selectfont1}_n\to F\textnormal{\fontfamily{pbk}\selectfont1}_n$ and $x\colon q^2E\textnormal{\fontfamily{pbk}\selectfont1}_n\to E\textnormal{\fontfamily{pbk}\selectfont1}_n$ act nilpotently on $M$; that is, if there exists $N\in\mathbb{N}$ such that $x^N(M)\colon q^{2N}FM\to FM$ and $x^N(M)\colon q^{2N}EM\to EM$ are zero (cf.\ \thref{&}).\par\begin{lemma}\thlabel{*} Suppose that $M\in\mathcal{C}_n$ is integrable. Then $d(M)<\infty$, and any $\gamma\in\operatorname{End}_q(\textnormal{\fontfamily{pbk}\selectfont1}_n)$ of positive degree acts nilpotently on $M$.\end{lemma} \begin{proof}Suppose that $x^N(M)\colon q^{2N}EM\to EM$ and $x^N(M)\colon q^{2N}FM\to FM$ are zero. Then $\iota^{+,n}_{(N+1)\cdot\ast}\colon\mathcal{R}^+((N+1)\cdot\ast)\to\operatorname{End}_q(E^{N+1}\textnormal{\fontfamily{pbk}\selectfont1}_n)$ factors through $\mathcal{R}^{-N}_{N+2}$, which vanishes since $\mathcal{R}^{-N}_m$ categorifies the weight $m$ space of the integrable $\mathfrak{sl}_2$-representation of lowest weight $-N$. Hence, $E^{N+1}M=0$ and $\varepsilon(M)<\infty$. Similarly, $\varphi(M)<\infty$.\par Let $\Lambda$ be the ring of symmetric functions, and let $h_r$ (resp. $e_r$) be the complete (resp. elementary) symmetric function of degree $r$. \bytes
$\Delta_{m,\varepsilon}(p)=\sum p_{(1)}\otimes p_{(2)}$, where $\Delta_{m,\varepsilon}$ is the coproduct separating the first $m$ variables and the last $\varepsilon$ variables. Under the isomorphism $\Psi(M)$, the action of $p$ corresponds to an action of $\sum p_{(1)}\otimes p_{(2)}$ on the bottom right of \eqref{7}, where $p_{(1)}$ acts on $Z_{m,\varphi}$ via multiplication, and $p_{(2)}$ acts on $E^{(\varepsilon)}M$ via $\iota^{+,n}_{\varepsilon\cdot\ast}(p_{(2)}e_\varepsilon)$ by \cite[(2.61)]{T}. Since $E^{(\varepsilon)}M$ is simple, the terms with $\deg p_{(2)}>0$ act as zero, and this reduces to the multiplication by $p(x_1,\ldots,x_m,0,\ldots,0)$ on $Z_{m,\varphi}$. It follows that $V$ is stable under the multiplication by $Z_{m,\varphi}$, i.e., $V$ is a right ideal of $Z_{m,\varphi}\otimes_\mathbb{Z}D$.\par Then the symmetric algebra structure of $Z_{m,\varphi}$ (see \cite[\S 3.3.2]{_}) forces $V\supset\omega D$. Indeed, let $v=\sum_{\alpha\subset(\varphi-m)^m}s_\alpha\otimes d_\alpha\in V$ $(d_\alpha\in D)$ be a nonzero element. Pick a partition $\alpha_0$ with a minimal $|\alpha_0|$ such that $d_{\alpha_0}\ne0$. Then $v\cdot(s_{\alpha_0^c}\otimes d_{\alpha_0}^{-1})=\omega\otimes1\in V$, where $\alpha_0^c\coloneqq(\varphi-m-(\alpha_0)_m,\ldots,\varphi-m-(\alpha_0)_1)$. Therefore, $E^{(\varepsilon+m)}L$ contains $q^{-m(\varphi-m)}\omega\otimes E^{(\varepsilon)}M$, which is precisely the image of the map $\Psi(\omega\otimes-)(M)$ (cf.\ \eqref{0}). Hence, $\Psi(\omega\otimes-)(M)$ factors through $E^{(\varepsilon+m)}(f)$. By adjunction, $\Phi$ factors through $f$.\end{proof}\par Now we can prove the following generalization of \cite[Prop.~5.20]{_} and \cite[Thm.~4.31]{Z} to arbitrary abelian categories.\par\begin{corollary}\thlabel{8} $F^{(m)}M$ has a simple essential socle and a simple essential head, and the following hold. \begin{enumerate}[(a),font=\normalfont]\item The map $\gamma_m\colon\mathbb{Z}[x_1,\ldots,x_m]^{S_m}\to\operatorname{End}_q(F^{(m)}M)$, $\gamma_m(p)\coloneqq\tau\big(\iota^{+,n-2m}_{m\cdot\ast}(pe_m)\big)(M)$, factors through $Z_{m,\varphi}$, and extends to an isomorphism of graded rings $Z_{m,\varphi}\otimes_\mathbb{Z}\operatorname{End}_q(M)\to\operatorname{End}_q(F^{(m)}M)$ by sending $p\otimes f$ to $\gamma_m(p)\circ F^{(m)}(f)$. \label{9}\item $q^{2m(\varphi-m)}\operatorname{hd}_e(F^{(m)}M)\cong\operatorname{soc}_e(F^{(m)}M)\cong\operatorname{im}\gamma_m(\omega)$. \label{:}\item $F^{(m)}M$ has a filtration $F^{(m)}M=A_0\supset B_0\supset A_1\supset\cdots\supset B_{m(\varphi-m)}=0$ such that \begin{enumerate}[(1),font=\normalfont]\item $d(B_k/A_{k+1})<d(M)$, and\item $\operatorname{hd}_eA_k\cong A_k/B_k\cong q^{2k}(\operatorname{hd}_e(F^{(m)}M))^{\oplus c_k}$, where $c_k$ is the coefficient of $q^{2k}$ in $q^{m(\varphi-m)}\genfrac{[}{]}{0pt}{}{\varphi}{m}$.\end{enumerate} \label{;}\end{enumerate}\end{corollary} \begin{proof}By \thref{6} applied to $\mathcal{C}$ and $\operatorname{op}\circ\mathcal{C}^{\operatorname{co}}\circ-$, $F^{(m)}M$ has a simple essential socle and a simple essential head. Let us show \ref{9} when $m=\varphi$. The map $F^{(\varphi)}\colon\operatorname{End}_q(M)\to\operatorname{End}_q(F^{(\varphi)}M)$ is injective, since any nonzero homogeneous element of the domain is invertible. Note that the adjunction gives a nonzero map $\pi\colon q^{\varphi\varepsilon}E^{(\varphi)}F^{(\varphi)}M\to M$, which realizes the essential head of the domain. Hence, for any $g\in\operatorname{End}_q(F^{(\varphi)}M)$ of degree $c$, the map $q^{c+\varphi\varepsilon}E^{(\varphi)}F^{(\varphi)}M\to M$ induced by adjunction from $g$, factors as $h\circ\pi$ for some $h\colon q^cM\to M$. Then $F^{(\varphi)}(h)=g$, so $F^{(\varphi)}$ is surjective.\par We have a commutative diagram below: \begin{equation}\label{<}\begin{tikzpicture}[xscale=2,yscale=1,baseline=0]\node at(0,0)(A){$\mathbb{Z}[x_1,\ldots,x_m]^{S_m}\otimes_\mathbb{Z}\operatorname{End}_q(M)$};\node at(5,0)(B){$\operatorname{End}_q(F^{(m)}M)$};\node at(0,-1.5)(C){$\mathbb{Z}[x_1,\ldots,x_m]^{S_m}\otimes_\mathbb{Z}\operatorname{End}_q(E^{(\varepsilon)}M)$};\node at(2.7,-1.5)(D){$\operatorname{End}_q(Z_{m,\varphi}\otimes_\mathbb{Z}E^{(\varepsilon)}M)$};\node at(5,-1.5)(E){$\operatorname{End}_q(E^{(\varepsilon+m)}F^{(m)}M)$};\draw[->](A)--(B)node[midway,above,scale=.9]{$p\otimes f\mapsto\gamma_m(p)\circ F^{(m)}(f)$};\draw[->](A)--(C)node[midway,left,scale=.9]{$\sim$}node[midway,right,scale=.9]{$1\otimes E^{(\varepsilon)}(-)$};\draw[{Hooks[left]}->](B)--(E)node[midway,right,scale=.9]{$E^{(\varepsilon+m)}(-)$};\draw[->](C)--(D);\draw[decorate,decoration=double deco,line width=0.8\pgflinewidth,alt double distance=1.8pt](D)--(E)node[midway,below,scale=.9]{$\sim$}node[midway,above,scale=.9]{$\Psi(M)$};\end{tikzpicture}\end{equation} At the bottom left arrow, $\mathbb{Z}[x_1,\ldots,x_m]^{S_m}$ acts on $Z_{m,\varphi}$ via multiplication, and $\operatorname{End}_q(E^{(\varepsilon)}M)$ acts on the second factor. Indeed, the two images of $p\otimes1$ at the bottom right coincide by $E^{(\varepsilon+m)}\big(\gamma_m(p)\big)\circ\Psi(p'\otimes-)(M)=\Psi(pp'\otimes-)(M)$ (cf.\ \eqref{+}), and the commutativity for the general case follows from the commutativity of horizontal compositions.\par By the $m=\varphi$ case for $\mathcal{C}\circ\psi$, the left vertical map is an isomorphism, and by the exactness of $E^{(\varepsilon+m)}$ and \thref{x} \ref{?}, the right vertical map is injective. The bottom left arrow kills $s_\lambda\otimes f$ for $\lambda\not\subset(\varphi-m)^m$. It follows that the top horizontal arrow factors through $Z_{m,\varphi}\otimes_\mathbb{Z}\operatorname{End}_q(M)$, which proves the first part of \ref{9}. Moreover, the map $Z_{m,\varphi}\otimes_\mathbb{Z}\operatorname{End}_q(E^{(\varepsilon)}M)\to\operatorname{End}_q(Z_{m,\varphi}\otimes_\mathbb{Z}E^{(\varepsilon)}M)$ is clearly injective, so the map $Z_{m,\varphi}\otimes_\mathbb{Z}\operatorname{End}_q(M)\to\operatorname{End}_q(F^{(m)}M)$ is injective.\par To see \ref{:}, it suffices to argue that $X\coloneqq\operatorname{im}\gamma_m(\omega)$ is simple by the uniqueness of simple essential socle and head. Under the isomorphism $\Psi(M)$, $E^{(\varepsilon+m)}X\subset E^{(\varepsilon+m)}F^{(m)}M$ corresponds to $\mathbb{Z}\omega\otimes E^{(\varepsilon)}M$, so $E^{(\varepsilon+m)}X$ is simple. Being a subobject of $F^{(m)}M$, $X$ has a simple essential socle. By \thref{x} \ref{?}, $E^{(\varepsilon+m)}\operatorname{soc}_e(X)$ is a nonzero subobject of $E^{(\varepsilon+m)}X$, so $E^{(\varepsilon+m)}\operatorname{soc}_e(X)=E^{(\varepsilon+m)}X$, and dually $E^{(\varepsilon+m)}X=E^{(\varepsilon+m)}\operatorname{hd}_e(X)$. If $X$ is not simple, $\operatorname{soc}_e(X)\to X\to\operatorname{hd}_e(X)$ vanishes, but applying $E^{(\varepsilon+m)}$ gives a nonzero map, a contradiction.\par Let $A_k\coloneqq\sum_{|\lambda|=k}\operatorname{im}\gamma_m(s_\lambda)$. We have well-defined maps\vspace{-.3em} \begin{equation}\label{=}\begin{tikzpicture}\node(A)at(0,0){$q^{2k}\bigoplus_{|\lambda|=k}F^{(m)}M$};\node(B)at(3.6,0){$A_k$};\node(C)at(8.2,0){$q^{2k-2m(\varphi-m)}\bigoplus_{|\lambda|=k}\operatorname{im}\gamma_m(\omega)$\:,};\draw[->>](A)--(B)node[midway,above,scale=.9]{$(\gamma_m(s_\lambda))_{\lambda}$};\draw[->](B)--(C)node[midway,above,scale=.9]{$(\gamma_m(s_{\lambda^c}))^{\mathsf{T}}_{\lambda}$};\end{tikzpicture}\vspace{-.3em}\end{equation} whose composition realizes the essential head of the domain, since the composition gives a diagonal matrix with entries $\gamma_m(\omega)$. This shows $\operatorname{hd}_eA_k\cong q^{2k}(\operatorname{hd}_e(F^{(m)}M))^{\oplus c_k}$. Let $B_k\coloneqq\ker(\gamma_m(s_{\lambda^c}))^{\mathsf{T}}_{\lambda}=A_k\cap\bigcap_{|\mu|=m(\varphi-m)-k}\ker\gamma_m(s_\mu)$. Then clearly $B_k\supset A_{k+1}$. By applying $E^{(\varepsilon+m)}$ to \eqref{=}, we get $(Z_{m,\varphi})_{\ge k}\otimes_\mathbb{Z}E^{(\varepsilon)}M\to\mathbb{Z}\omega\otimes_\mathbb{Z}E^{(\varepsilon)}M$. Thus, $E^{(\varepsilon+m)}B_k\cong(Z_{m,\varphi})_{\ge k+1}\otimes_\mathbb{Z}E^{(\varepsilon)}M\cong E^{(\varepsilon+m)}A_{k+1}$. Therefore, we have $E^{(\varepsilon+m)}(B_k/A_{k+1})=0$, implying $d(B_k/A_{k+1})<d(M)$. This proves \ref{;}.\par It remains to show the surjectivity of the map in \ref{9}. Since $F^{(\varphi-m)}X$ is simple and admits a surjection from $F^{(\varphi-m)}F^{(m)}M$, $F^{(\varphi-m)}X$ is graded isomorphic to $F^{(\varphi)}M$. By the $m=\varphi$ case of \ref{9}, the maps $\operatorname{End}_q(M)\to\operatorname{End}_q(F^{(\varphi)}M)\leftarrow\operatorname{End}_q(X)$ are isomorphisms, so the canonical map $\operatorname{End}_q(M)\to\operatorname{End}_q(\operatorname{hd}_eF^{(m)}M)$ is an isomorphism. By \ref{;} and \thref{x}, the dimension of the degree $d$ component of the target as a right $\operatorname{End}(M)$-vector space is at most $\sum_{d-2k\in S}c_k$ where $S\coloneqq\left\{\kern.07eme\in\mathbb{Z}\,\middle|\,q^{e}M\cong M\kern.07em\right\}$. This coincides with the dimension of the degree $d$ component of the source. Since we already know that the map is injective, it must be surjective.\end{proof}\par\begin{remark}The main obstacle to extending the proof of \cite[Prop.~5.20]{_} is that $F^{(m)}M$ is not guaranteed to have a simple subobject at all; its existence is needed \textit{a priori} to bootstrap the rest of the argument. The noetherian and artinian assumptions of \cite{_} and the locally Schurian setup of \cite{Z} both ensure its existence via ambient finiteness conditions. Realizing this simple subobject intrinsically (\thref{6}) allows us to dispense with assumptions on the abelian categories altogether.\par We also point out an analogy with the argument of \cite[Thm.~3.2]{f}, which is further developed in \cite{g}: the role of the renormalized $R$-matrix is, albeit by different means, played by the map $\Phi$.\par\end{remark}\subsection{Categorical crystals}Recall that $\Bbbk$ is a commutative ring with unit.\par\begin{definition}\thlabel{&} Let $\mathcal{C}\colon\mathscr{U}\to\mathrm{Ab}_\Bbbk$ be a $\Bbbk$-linear pseudofunctor. We say that $M\in\mathcal{C}_\mu$ is \emph{integrable (with respect to $\mathcal{C}$)} if $x\colon q_i^2F_i\textnormal{\fontfamily{pbk}\selectfont1}_\mu\to F_i\textnormal{\fontfamily{pbk}\selectfont1}_\mu$ and $x\colon q_i^2E_i\textnormal{\fontfamily{pbk}\selectfont1}_\mu\to E_i\textnormal{\fontfamily{pbk}\selectfont1}_\mu$ act nilpotently on $M$ for all $i\in I$. We denote the full subcategory of $\mathcal{C}_\mu$ consisting of integrable objects by $\mathcal{C}_\mu^{\operatorname{int}}$.\end{definition}\par\begin{lemma}\thlabel{>} The category $\mathcal{C}_\mu^{\operatorname{int}}$ is a Serre subcategory of $\mathcal{C}_\mu$, and the pseudofunctor $\mathcal{C}$ restricts to a pseudofunctor $\mathcal{C}^{\operatorname{int}}$.\end{lemma} \begin{proof}Clearly, $\mathcal{C}_\mu^{\operatorname{int}}$ is closed under subobjects, quotients, and extensions. Fix $M\in\mathcal{C}_\mu^{\operatorname{int}}$, and let $i,j\in I$. Since the map $\iota^{-,\mu}_{i+j}(M)|_{q=1}\colon\mathcal{R}^-|_{q=1}(i\otimes j,i\otimes j)\to\operatorname{End}_q(F_jF_iM)$ factors through a cyclotomic quotient, $x\mathbin{\boldsymbol{\cdot}}1_{F_iM}$ acts nilpotently on $F_jF_iM$ by \cite[Lem.~4.3 (a)]{L}. Suppose that $x^N$ acts as zero on $E_jM$ and $F_jM$. By the commutator relations and \eqref{c}, $x^{2N}\mathbin{\boldsymbol{\cdot}}1_{F_i}\in\operatorname{End}_q(E_jF_i\textnormal{\fontfamily{pbk}\selectfont1}_\mu)$ is a linear combination of a diagram with two sideways crossings $\sigma\circ\sigma'$ and diagrams in which the top two endpoints are connected by a cup, each having at least $N$ dots on a strand adjacent to the leftmost area. Hence $x\mathbin{\boldsymbol{\cdot}}1_{F_iM}$ acts nilpotently on $E_jF_iM$, implying $F_iM\in\mathcal{C}_{\mu-i}^{\operatorname{int}}$. Similarly we have $E_iM\in\mathcal{C}_{\mu+i}^{\operatorname{int}}$.\end{proof}\par For an abelian category $\mathcal{A}$, we denote by $\mathbf{B}(\mathcal{A})$ the set of isomorphism classes of simple objects in $\mathcal{A}$. Whenever we write $\mathbf{B}(\mathcal{A})$, we implicitly assume that the full subcategory of simple objects of $\mathcal{A}$ is essentially small. For a pseudofunctor $\mathcal{C}\colon\mathscr{U}\to\mathrm{Ab}_\Bbbk$, we denote $\mathbf{B}(\mathcal{C})\coloneqq\bigsqcup_{\mu\in X}\mathbf{B}(\mathcal{C}_\mu)$.\par The definition of a crystal (see e.g.\ \cite[\S1.5]{O}) applies verbatim to a weight datum.\par\begin{theorem}\thlabel{@} Let $\mathcal{C}\colon\mathscr{U}\to\mathrm{Ab}_\Bbbk$ be a $\Bbbk$-linear pseudofunctor. Let $M\in\mathbf{B}(\mathcal{C}^{\operatorname{int}}_\mu)$ and $i\in I$. The following endows $\mathbf{B}(\mathcal{C}^{\operatorname{int}})$ with the structure of a seminormal crystal.\begin{align}\varepsilon_i(M)\coloneqq\sup\left\{\kern.07emm\in\mathbb{Z}_{\ge0}\,\middle|\,E_i^mM\ne0\kern.07em\right\},&\kern1.5em\varphi_i(M)\coloneqq\sup\left\{\kern.07emm\in\mathbb{Z}_{\ge0}\,\middle|\,F_i^mM\ne0\kern.07em\right\},\\\widetilde{e}_i(M)\coloneqq q_i^{-\varepsilon_i(M)+1}\operatorname{soc}_eE_iM,&\kern1.5em\widetilde{f}_i(M)\coloneqq q_i^{\varphi_i(M)-1}\operatorname{hd}_eF_iM.\label{^}\end{align} \end{theorem} \begin{proof}Let $\mathscr{U}_i$ be the $\Bbbk$-linear 2-category $\mathscr{U}$ associated to the Cartan datum of $\mathfrak{sl}_2$, a weight datum $(\mu+\mathbb{Z}i,\langle\ast,-\rangle)$ with $\langle\ast,\mu'\rangle\coloneqq\langle i,\mu'\rangle$ for $\mu'\in\mu+\mathbb{Z}i$, the trivial quiver Hecke datum, and the bubble parameters given by restricting $c_i$ to $\mu+\mathbb{Z}i$. We have an obvious 2-functor $\mathscr{U}_i\to\mathscr{U}$. Let $\mathscr{U}_i'$ be the $\Bbbk$-linear 2-category constructed from the same data as $\mathscr{U}_i$ except with bubble parameters $c'_i(\mu')=1$ for all $\mu'$. There exists an isomorphism of $\Bbbk$-linear 2-categories $\mathscr{U}'_i\to\mathscr{U}_i$ that is the identity on objects and 1-morphisms (cf.\ \cite[Thm.~2.1]{b}), and $\mathscr{U}'_i$ is isomorphic to the full sub-2-category of $\mathscr{U}(\mathfrak{sl}_2)$ consisting of objects in $\langle i,\mu\rangle+2\mathbb{Z}$. Let $\mathcal{C}_i\colon\mathscr{U}(\mathfrak{sl}_2)\to\mathrm{Ab}_\Bbbk$ be the pseudofunctor that extends $\mathscr{U}_i'\to\mathscr{U}_i\to\mathscr{U}\to\mathrm{Ab}_\Bbbk$ by sending objects not in $\langle i,\mu\rangle+2\mathbb{Z}$ to the zero category. By applying \thref{8} to $\mathcal{C}_i$, we obtain $\varepsilon_i(M)=\max\left\{\kern.07emm\in\mathbb{Z}_{\ge0}\,\middle|\,E_i^mM\ne0\kern.07em\right\}$ and $\widetilde{e}_i(M)\in\mathbf{B}(\mathcal{C}^{\operatorname{int}}_{\mu+i})\sqcup\{0\}$, and similarly for $\varphi_i(M)$ and $\widetilde{f}_i(M)$. We have $\varphi_i(M)=\varepsilon_i(M)+\langle i,\mu\rangle$ by \thref{x} \ref{y}. It remains to check that $\widetilde{f}_i(M)\cong N$ if and only if $M\cong\widetilde{e}_i(N)$ for $M\in\mathbf{B}(\mathcal{C}^{\operatorname{int}}_\mu)$ and $N\in\mathbf{B}(\mathcal{C}^{\operatorname{int}}_{\mu-i})$. Either isomorphism implies $\varphi_i(M)-1=\varphi_i(N)$ (\thref{x} \ref{?}). Then both conditions are equivalent to the non-vanishing of $\mathcal{C}_{\mu-i}(\widetilde{f}_i(M),N)\cong\mathcal{C}_{\mu-i}(q_i^{\varphi_i(M)-1}F_iM,N)\cong\mathcal{C}_{\mu}(M,q_i^{-\varepsilon_i(N)+1}E_iN)\cong\mathcal{C}_{\mu}(M,\widetilde{e}_i(N))$. \end{proof}\par\section{Crystals of orthodox bases}\label{E}\subsection{Humorous categories}\label{`}\par Let $\mathcal{K}$ be a category with an endofunctor $q$. \begin{definition}\thlabel{|} We say that $\mathcal{K}$ is \emph{locally left-bounded} with respect to $q$ if $\mathcal{K}(x,q^dy)=0$ for $d\gg0$ for all $x,y\in\mathcal{K}$. If, moreover, $\mathcal{K}$ is $\Bbbk$-linear and $\mathcal{K}(x,y)$ is a finitely generated $\Bbbk$-module for all $x,y\in\mathcal{K}$, we say that $\mathcal{K}$ is \emph{locally Laurentian} with respect to $q$. \end{definition}\par We call a graded $\Bbbk$-linear category locally left-bounded (resp.\ locally Laurentian) if it is so with respect to the grading shift functor $q$. When $\Bbbk$ is a field, locally Laurentian graded categories are many-object analogues of the \emph{Laurentian graded algebras} in \cite{h}.\par Suppose that $\Bbbk$ is a commutative Henselian local ring. For a $\Bbbk$-linear category $\mathcal{K}$ whose Hom-spaces are finitely generated over $\Bbbk$, $\operatorname{End}(x)$ is a semiperfect ring for every $x\in\dot{\mathcal{K}}$ \cite[A.30]{i}, and hence $\dot{\mathcal{K}}$ is a Krull--Schmidt category \cite[Cor.~4.4]{j}.\par\begin{definition}[{cf.\ \cite[Def.~1.2]{V}}]\thlabel{_A} A \emph{humorous category} over $\Bbbk$ is a small locally Laurentian Krull--Schmidt graded $\Bbbk$-linear category $\mathcal{K}$, equipped with a duality functor $\circledast\colon\mathcal{K}\to\mathcal{K}^{\operatorname{op}}$ (that is, a functor satisfying $\circledast^{\operatorname{op}}\circ\circledast\cong\textnormal{\fontfamily{pbk}\selectfont1}_{\mathcal{K}}$), such that \begin{enumerate}[(i)]\item $\circledast\circ q\cong q^{-1}\circ\circledast$, and \label{AA}\item every indecomposable object of $\mathcal{K}$ is absolutely indecomposable, and is graded isomorphic to a self-dual object.\label{BA}\end{enumerate} Here, an object $x$ is \emph{self-dual} if $x\cong x^\circledast$, and is \emph{absolutely indecomposable} if $\operatorname{End}(x)\otimes_\Bbbk\Bbbk'$ is a local ring for any local homomorphism $\Bbbk\to\Bbbk'$ of commutative Henselian local rings.\end{definition}\par If $\mathcal{K}$ is a humorous category, then its split Grothendieck group $K_0(\mathcal{K})$ is a free $\mathbb{Z}[q^\pm]$-module, and the isomorphism classes of self-dual indecomposable objects form a basis, called the \emph{orthodox basis} in \cite{V}.\par\begin{remark}\thlabel{CA} In \cite{V}, humorous categories are defined over a field, and are required to have only finitely many graded isomorphism classes of indecomposable objects. We drop the latter condition, since $\dot{\mathscr{U}}(\lambda,\mu)$ does not satisfy it; see \thref{DA}. Note that \cite[\S3--5]{V} works over a commutative complete local ring, and that the arguments readily extend to a commutative Henselian local ring. Theorem~5.17 of \cite{V}, which states that each $\dot{\mathscr{U}}(\lambda,\mu)$ is a humorous category, may thus be understood in the sense of \thref{_A} (see also \thref{EA}).\end{remark}\par\subsection{$\mathscr{U}$-actions on Krull--Schmidt categories}\label{FA}\par For a small $\Bbbk$-linear Krull--Schmidt category $\mathcal{K}$, let $\mathbf{G}(\mathcal{K})$ be the set of isomorphism classes of indecomposable objects in $\mathcal{K}$. For $P\in\mathbf{G}(\mathcal{K})$, $\mathcal{K}(-,P)\in\operatorname{mod}\!\operatorname{-}\hspace{-.8mm}\mathcal{K}$ has the simple head $\mathcal{K}(-,P)/\operatorname{rad}\mathcal{K}(-,P)$, which induces a bijection $\mathbf{G}(\mathcal{K})\to\mathbf{B}(\operatorname{mod}\!\operatorname{-}\hspace{-.8mm}\mathcal{K})$.\par Let $\wedge\colon\mathrm{Cat}_\Bbbk\to\mathrm{Ab}_\Bbbk^{\operatorname{coop}}$ denote the 2-functor sending a category $\mathcal{K}$ to $\operatorname{mod}\!\operatorname{-}\hspace{-.8mm}\mathcal{K}\coloneqq\mathrm{Cat}_\Bbbk(\mathcal{K}^{\operatorname{op}},\Bbbk\operatorname{\!-\!}\operatorname{mod})$. For a pseudofunctor $\mathcal{K}\colon\mathscr{U}\to\mathrm{Cat}_\Bbbk$, define\begin{equation}\label{GA}\begin{tikzcd}\operatorname{mod}\!\operatorname{-}\hspace{-.8mm}\mathcal{K}\coloneqq\wedge^{\operatorname{coop}}\circ\mathcal{K}^{\operatorname{coop}}\circ\tau\colon\ \mathscr{U}\arrow[r,"\tau"]&\mathscr{U}^{\operatorname{coop}}\arrow[r,"{\mathcal{K}^{\operatorname{coop}}}"]&\mathrm{Cat}_\Bbbk^{\operatorname{coop}}\arrow[r,"{\wedge^{\operatorname{coop}}}"]&\mathrm{Ab}_\Bbbk.\end{tikzcd}\end{equation}\par Let $\mathcal{K}\colon\mathscr{U}\to\mathrm{Cat}_\Bbbk$ be a pseudofunctor. Suppose that \begin{enumerate}[(a),font=\normalfont]\item $\mathcal{K}_\mu$ is a small Krull--Schmidt category, and \label{HA}\item all simple objects of $\operatorname{mod}\!\operatorname{-}\hspace{-.8mm}\mathcal{K}_\mu$ are integrable with respect to $\operatorname{mod}\!\operatorname{-}\hspace{-.8mm}\mathcal{K}$ in the sense of \thref{&} \label{IA}\end{enumerate} for all $\mu\in X$. By transferring the crystal structure on $\bigsqcup_{\mu\in X}\mathbf{B}(\operatorname{mod}\!\operatorname{-}\hspace{-.8mm}\mathcal{K}_\mu)$ given by \thref{@} via bijections $\mathbf{G}(\mathcal{K}_\mu)\to\mathbf{B}(\operatorname{mod}\!\operatorname{-}\hspace{-.8mm}\mathcal{K}_\mu)$, we obtain a natural crystal structure on $\mathbf{G}(\mathcal{K})\coloneqq\bigsqcup_{\mu\in X}\mathbf{G}(\mathcal{K}_\mu)$, characterized up to degree shifts by the following proposition.\par\begin{proposition}\thlabel{JA} Let $P\in\mathbf{G}(\mathcal{K}_\mu)$ and $i\in I$. \begin{enumerate}[(i),font=\normalfont]\item $\varepsilon_i(P)=\max\left\{\kern.07emn\in\mathbb{Z}_{\ge0}\,\middle|\,P\text{ is a direct summand of }F_i^nQ\text{ for some }Q\in\mathcal{K}_{\mu+ni}\kern.07em\right\}$. \label{KA}\item $E_i^{(m)}P\cong\genfrac{[}{]}{0pt}{}{\varphi_i(P)+m}{m}_i\widetilde{e}_i^m(P)\oplus\bigoplus_{k=1}^nR_k$ for some $R_k\in\mathbf{G}(\mathcal{K}_{\mu+mi})$ with $d_i(R_k)>d_i(P)$ $(1\le k\le n)$. \label{LA}\end{enumerate} \end{proposition} \begin{proof}Let $L\coloneqq\operatorname{hd}\mathcal{K}(-,P)$. The object $E_i^nL$ does not vanish if and only if $E_i^nL(Q)=L(F_i^nQ)\ne0$ for some $Q\in\mathcal{K}_{\mu+ni}$, which is equivalent to $P$ being a summand of $F_i^nQ$. This proves \ref{KA}.\par Let $Q\in\mathbf{G}(\mathcal{K}_{\mu+mi})$ and $L'\coloneqq\operatorname{hd}\mathcal{K}(-,Q)$. The object $Q$ occurs as a summand of $E_i^{(m)}P$ if and only if $0\ne L'(E_i^{(m)}P)\cong F_i^{(m)}L'(P)$. This is equivalent to $L$ being a subquotient of $F_i^{(m)}L'$. Under this condition, we have $d_i(L)\le d_i(L')$, and the equality holds if and only if $L$ is graded isomorphic to $\widetilde{f}_i^mL'$ (\thref{8} \ref{;}). Hence, $E_i^{(m)}P\cong f(q)\widetilde{e}_i^m(P)\oplus R$ for some $f(q)\in\mathbb{N}[q^\pm]$ and $R\in\mathcal{K}_{\mu+mi}$ with $d_i(R)>d_i(P)$. If $\widetilde{e}_i^m(P)=0$, then we are done. Otherwise, we compare dimensions of $L'(E_i^{(m)}P)\cong F_i^{(m)}L'(P)$ as right $\operatorname{End}(L')$-vector spaces for $Q=q^d\widetilde{e}_i^m(P)$. By \thref{8} \ref{;}, we obtain $\sum_{e-d\in S}[q^e]f(q)=\sum_{e+d\in S}[q^e]\genfrac{[}{]}{0pt}{}{\varphi_i(L')}{m}_i$, where $[q^x]f(q)$ denotes the coefficient of $q^x$ in $f(q)$, and $S\coloneqq\left\{\kern.07eme\in\mathbb{Z}\,\middle|\,q^eP\cong P\kern.07em\right\}$. Thus, $f(q)\widetilde{e}_i^m(P)\cong\genfrac{[}{]}{0pt}{}{\varphi_i(P)+m}{m}_i\widetilde{e}_i^m(P)$, which proves \ref{LA}.\end{proof}\subsection{$\mathscr{U}$-actions on locally left-bounded categories}\label{MA}\par Let $\mathcal{K}\colon\mathscr{U}\to\mathrm{Cat}_\Bbbk$ be a pseudofunctor. For $\mu\in X$, $i\in I$ and $N\in\mathbb{Z}_{\ge0}$, let $\mathcal{I}^{\pm i,N}_\mu$ be a two-sided ideal of $\mathcal{K}_\mu$ generated by morphisms $q_i^{2N}E^\pm_ix\to E^\pm_ix$ for $x\in\mathcal{K}_{\mu\mp i}$ that are induced by the 2-morphism $x_i^N\colon q_i^{2N}E^\pm_i\textnormal{\fontfamily{pbk}\selectfont1}_{\mu\mp i}\to E^\pm_i\textnormal{\fontfamily{pbk}\selectfont1}_{\mu\mp i}$.\par Recall that the radical $\operatorname{rad}\mathcal{A}$ of a $\mathbb{Z}$-linear category $\mathcal{A}$ is a two-sided ideal of $\mathcal{A}$ consisting of all morphisms $f\colon x\to y$ such that $1_x-gf$ is invertible for all $g\colon y\to x$.\par\begin{lemma}\thlabel{NA} Let $x\in\mathcal{K}_\mu$ be given. Suppose that $\mathcal{K}_{\mu\mp i}$ is locally left-bounded with respect to $q\textnormal{\fontfamily{pbk}\selectfont1}_{\mu\mp i}$. Then the following hold for all $N\gg0$\kern.05em\textup{:}\kern .3em \begin{enumerate*}[(a),font=\normalfont]\item $\mathcal{I}^{\pm i,N}_\mu(x,x)=0$.\label{OA}\item $\mathcal{I}^{\pm i,N}_\mu(y,x)\subset\operatorname{rad}\mathcal{K}_\mu(y,x)$ for all $y\in\mathcal{K}_\mu$.\label{PA}\end{enumerate*}\end{lemma} \begin{proof}Consider a map \[\psi^{+i,N}_\mu\colon\mathcal{K}_{\mu-i}\left(q_i^{\langle i,\mu\rangle-1}F_ix,q_i^{-\langle i,\mu\rangle+1+2N}F_ix\right)\to\mathcal{K}_\mu(x,x)\] that sends $f$ in the domain to a composition \begin{equation}\begin{tikzcd}x\arrow[r,"\eta(x)"]&q_i^{\langle i,\mu\rangle-1}E_iF_ix\arrow[r,"E_i(f)"]&q_i^{-\langle i,\mu\rangle+1+2N}E_iF_ix\arrow[r,"x^N(F_ix)"]&q_i^{-\langle i,\mu\rangle+1}E_iF_ix\arrow[r,"\n'(x)"]&x.\end{tikzcd}\label{QA}\end{equation} Clearly, $\psi^{+i,N}_\mu(f)$ lies in the ideal $\mathcal{I}^{+i,N}_\mu(x,x)$. On the other hand, $\mathcal{I}^{+i,N}_\mu(x,x)$ is spanned by maps of the form $f'\coloneqq x\xrightarrow{g}q_i^{2N}E_iz\xrightarrow{x^N(z)}E_iz\xrightarrow{h}x$ for some $z\in\mathcal{K}_{\mu-i}$ and $g,h$. Using the cyclicity relations \eqref{e}, one can check that the map $F_i(h)\circ\eta'(z)\circ\varepsilon(z)\circ F_i(g)$ lies in the inverse image of $f'$ under $\psi^{+i,N}_\mu$. Indeed, the map $\psi^{+i,N}_\mu$ and the choice of an inverse image may be depicted as below:\vspace{-.7em} \begingroup 
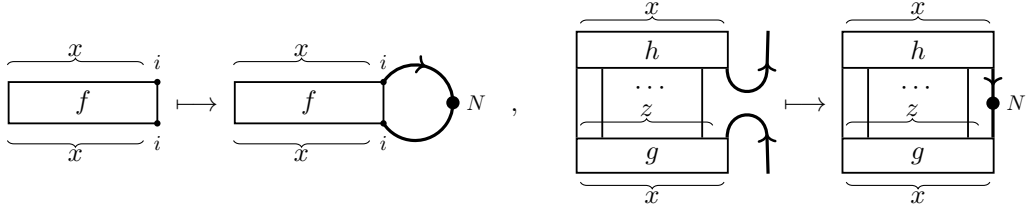
\[\begin{tikzpicture}[baseline=-3pt]\node[fill=white,draw=black,inner xsep=25pt,line width=0.7pt](A)at(0,0){$f$};\fill(A.north east)++(-.35pt,-.35pt)coordinate(topdot)circle(1.2pt)node[scale=.8,above=2pt]{$i$};\fill(A.south east)++(-.35pt,.35pt)coordinate(bottomdot)circle(1.2pt)node[scale=.8,below=2pt]{$i$};\draw[decorate,decoration=brace,-](.8,-.5)--(-1,-.5)node[midway,below]{$x$};\draw[decorate,decoration=brace,-](-1,.5)--(.8,.5)node[midway,above]{$x$};\end{tikzpicture}\longmapsto\;\begin{tikzpicture}[baseline=-3pt]\node[fill=white,draw=black,inner xsep=25pt,line width=0.7pt](A)at(0,0){$f$};\fill(A.north east)++(-.35pt,-.35pt)coordinate(topdot)circle(1.2pt)node[scale=.8,above=2pt]{$i$};\fill(A.south east)++(-.35pt,.35pt)coordinate(bottomdot)circle(1.2pt)node[scale=.8,below=2pt]{$i$};\draw[decorate,decoration=brace,-](.8,-.5)--(-1,-.5)node[midway,below]{$x$};\draw[decorate,decoration=brace,-](-1,.5)--(.8,.5)node[midway,above]{$x$};\draw[very thick,baseline,-,postaction={ decorate, decoration={ markings, mark=at position .5 with {\fill(0,0)circle(2.5pt);\node[right,xshift=2pt,scale=.8]at(0,0){$N$};} } },postaction={ decorate, decoration={ markings, mark=at position .25 with {{\arrow[scale=.8]{>}}} } }](topdot)arcto[radius=.5,counter clockwise](bottomdot);\end{tikzpicture}\ \ ,\qquad\begin{tikzpicture}[baseline=-3pt]\node[fill=white,draw=black,minimum width=2cm,minimum height=.6,align=center,line width=0.7pt](A)at(0,.7){$h$};\node[fill=white,draw=black,minimum width=2cm,minimum height=.6,align=center,line width=0.7pt](B)at(0,-.7){$g$};\draw[decorate,decoration=brace,-](.95,-1.05)--(-.95,-1.05)node[midway,below]{$x$};\draw[decorate,decoration=brace,-](-.95,1.05)--(.95,1.05)node[midway,above]{$x$};\draw[decorate,decoration=brace,-](-.95,-.3)--(.8,-.3)node[midway,above]{$z$};\draw[line width=0.7pt]([xshift=.35pt]A.south west)--([xshift=.35pt]B.north west);\draw[line width=0.7pt]([xshift=+10pt]A.south west)--([xshift=+10pt]B.north west);\node at(0,.2){$\cdots$};\draw[line width=0.7pt]([xshift=-10pt]A.south east)--([xshift=-10pt]B.north east);\draw[very thick,baseline][postaction={ decorate, decoration={ markings, mark=at position .7 with {{\arrow[scale=.8]{>}}} } }]([xshift=-.6pt]A.south east)..controls+(0,-.4)and+(0,-.4)..+(15pt,0)--([xshift=14.8pt]A.north east);\draw[very thick,baseline][postaction={ decorate, decoration={ markings, mark=at position .7 with {{\arrow[scale=.8]{<}}} } }]([xshift=-.6pt]B.north east)..controls+(0,.4)and+(0,.4)..+(15pt,0)--([xshift=14.8pt]B.south east);\end{tikzpicture}\longmapsto\;\begin{tikzpicture}[baseline=-3pt]\node[fill=white,draw=black,minimum width=2cm,minimum height=.6,align=center,line width=0.7pt](A)at(0,.7){$h$};\node[fill=white,draw=black,minimum width=2cm,minimum height=.6,align=center,line width=0.7pt](B)at(0,-.7){$g$};\draw[decorate,decoration=brace,-](.95,-1.05)--(-.95,-1.05)node[midway,below]{$x$};\draw[decorate,decoration=brace,-](-.95,1.05)--(.95,1.05)node[midway,above]{$x$};\draw[decorate,decoration=brace,-](-.95,-.3)--(.8,-.3)node[midway,above]{$z$};\draw[line width=0.7pt]([xshift=.35pt]A.south west)--([xshift=.35pt]B.north west);\draw[line width=0.7pt]([xshift=+10pt]A.south west)--([xshift=+10pt]B.north west);\node at(0,.2){$\cdots$};\draw[line width=0.7pt]([xshift=-10pt]A.south east)--([xshift=-10pt]B.north east);\draw[very thick,baseline][postaction={ decorate, decoration={ markings, mark=at position .5 with {\fill(0,0)circle(2.5pt);\node[right,xshift=2pt,scale=.8]at(0,0){$N$};} } }][postaction={ decorate, decoration={ markings, mark=at position .25 with {{\arrow[scale=.8]{>}}} } }]([xshift=-.6pt]A.south east)--([xshift=-.6pt]B.north east);\end{tikzpicture}\vspace{-.4em}\] \endgroup It follows that $\mathcal{I}^{+i,N}_\mu(x,x)=\operatorname{im}\psi^{+i,N}_\mu$. Since $\mathcal{K}_{\mu-i}$ is locally left-bounded with respect to $q\textnormal{\fontfamily{pbk}\selectfont1}_{\mu-i}$, the domain of $\psi^{+i,N}_\mu$ vanishes for all sufficiently large $N$, and hence so does its image. Therefore $\mathcal{I}^{+i,N}_\mu(x,x)=0$ for all sufficiently large $N$. For $f\in\mathcal{I}^{\pm i,N}_\mu(y,x)$ and $g\in\mathcal{K}_\mu(x,y)$, $fg\in\mathcal{I}^{\pm i,N}_\mu(x,x)$ vanishes, so $1_x-fg=1_x$ is invertible. Hence, $f\in\operatorname{rad}\mathcal{K}_\mu(y,x)$. This proves the case of the $+$ sign, and the case of the $-$ sign is similar.\end{proof}\par\begin{proposition}\thlabel{RA} Suppose that $\mathcal{K}\colon\mathscr{U}\to\mathrm{Cat}_\Bbbk$ is a pseudofunctor such that $\mathcal{K}_\mu$ is locally left-bounded with respect to $q\textnormal{\fontfamily{pbk}\selectfont1}_\mu$ for all $\mu\in X$. Then all simple modules of $\mathcal{K}_\mu$ are integrable with respect to $\operatorname{mod}\!\operatorname{-}\hspace{-.8mm}\mathcal{K}$.\end{proposition} \begin{proof}It suffices to show the statement for $\mathscr{U}(\mathfrak{sl}_2)$. Let $L\in\mathbf{B}(\operatorname{mod}\!\operatorname{-}\hspace{-.8mm}\mathcal{K}_\mu)$. There exists $P\in\mathcal{K}_\mu$ such that $L(P)\ne0$, and a nonzero element of $L(P)$ induces a surjective map $\mathcal{K}_\mu(-,P)\twoheadrightarrow L$, which factors through $H\coloneqq\mathcal{K}_\mu(-,P)/\operatorname{rad}\mathcal{K}_\mu(-,P)$, since elements of the radical act trivially on simple modules \cite[Prop.~9]{k}.\par From \eqref{GA}, the action of the 2-morphism $x^N\colon q^{2N}F\textnormal{\fontfamily{pbk}\selectfont1}_\mu\to F\textnormal{\fontfamily{pbk}\selectfont1}_\mu$ on $\mathcal{K}_\mu(-,P)\in\operatorname{mod}\!\operatorname{-}\hspace{-.8mm}\mathcal{K}_\mu$ gives a map whose component at $Q\in\mathcal{K}_{\mu-2}$ is given by precomposition with $x^N(Q)\colon EQ\to q^{-2N}EQ$. It follows that $x^N(\mathcal{K}_\mu(-,P))(Q)\in\mathcal{I}^{+,N}_\mu(EQ,P)$. Since $\mathcal{I}^{+,N}_\mu(EQ,P)\subset\operatorname{rad}\mathcal{K}(EQ,P)$ holds for all $Q$ for all sufficiently large $N$ by \thref{NA}, we have $x^N(H)=0$ for $N\gg0$, and thus $x\colon q^2F\textnormal{\fontfamily{pbk}\selectfont1}_\mu\to F\textnormal{\fontfamily{pbk}\selectfont1}_\mu$ acts nilpotently on $L$. Applying this to $\mathcal{K}\circ\psi$, $x\colon q^{2}E1_\mu\to E1_\mu$ acts nilpotently on $L$.\end{proof}\par\subsection{Duality of a $\mathscr{U}$-action}\label{SA}\par Let $\operatorname{op}\colon\mathrm{Cat}_\Bbbk\to\mathrm{Cat}_\Bbbk^{\operatorname{co}}$ be the 2-functor sending a category to its opposite category.\par\begin{definition}\thlabel{TA} A \emph{duality} of a pseudofunctor $\mathcal{V}\colon\mathscr{U}\to\mathrm{Cat}_\Bbbk$ is a strong transformation \[\circledast\colon\mathcal{V}^{\operatorname{co}}\circ-\longrightarrow\operatorname{op}\circ\mathcal{V},\] such that the composition of strong transformations \begin{equation}\label{UA}(\operatorname{op}^{\operatorname{co}}\mathbin{\smallstar}\circledast)\circ(\circledast^{\operatorname{co}}\mathbin{\smallstar}\mathop{-}\nolimits)\colon\mathcal{V}=\mathcal{V}\circ\mathop{-}\nolimits^{\operatorname{co}}\circ\mathop{-}\nolimits\xrightarrow{\circledast^{\operatorname{co}}\mathbin{\smallstar}\mathop{-}\nolimits}\operatorname{op}^{\operatorname{co}}\circ\mathcal{V}^{\operatorname{co}}\circ\mathop{-}\nolimits\xrightarrow{\operatorname{op}^{\operatorname{co}}\mathbin{\smallstar}\kern.14em\circledast}\operatorname{op}^{\operatorname{co}}\circ\operatorname{op}\circ\mathcal{V}=\mathcal{V}\end{equation} is isomorphic to $\textnormal{\fontfamily{pbk}\selectfont1}_{\mathcal{V}}$ in the category $\mathrm{Bicat}_\Bbbk(\mathscr{U},\mathrm{Cat}_\Bbbk)(\mathcal{V},\mathcal{V})$.\end{definition}\par\begin{remark}A duality decategorifies to a bar-involution (in the sense of \cite[27.1.2~(c)]{a}) of the $\widetilde{U}_\mathbb{Z}$-module $V=\bigoplus_{\mu\in X}K_0(\dot{\mathcal{V}}_\mu)$. Indeed, we may regard $V$ as a functor $\widetilde{U}_\mathbb{Z}\to\mathbb{Z}[q^{\pm}]\operatorname{\!-\!}\operatorname{mod}$ (see \S\ref{S}). Then the data of a bar-involution of $V$ amounts to a natural isomorphism $\phi\colon V\circ\mathop{-}\nolimits\to V$ satisfying $\phi\circ(\phi\mathbin{\boldsymbol{\cdot}}1_{\mathop{-}\nolimits})=1_V$.\end{remark}\par Unwinding the definition, the data of a duality $\circledast$ consists of $\Bbbk$-linear functors $\circledast_\mu\colon\mathcal{V}_\mu\to\mathcal{V}_\mu^{\operatorname{op}}$ for each $\mu\in X$, natural isomorphisms $\circledast_f$ for each 1-morphism $f\in\mathscr{U}(\mu,\nu)$ fitting into the diagram \eqref{VA} below (these together form the strong transformation $\circledast$), and natural isomorphisms $\gamma_\mu\colon\textnormal{\fontfamily{pbk}\selectfont1}_{\mathcal{V}_\mu}\to\circledast_\mu^{\operatorname{op}}\circ\circledast_\mu$ for each $\mu\in X$ that witness the invertible modification required by the last condition. \begin{equation}\label{VA}\begin{tikzpicture}[yscale=.8,baseline=-10pt]\node(A)at(-1,1){$\mathcal{V}_\mu$};\node(B)at(1,1){$\mathcal{V}_\nu$};\node(C)at(-1,-1){$\mathcal{V}_\mu^{\operatorname{op}}$};\node(D)at(1,-1){$\mathcal{V}_\nu^{\operatorname{op}}$};\draw[->](A)--(B)node[midway,above,scale=.9]{$\mathcal{V}(\overline{f})$};\draw[->](A)--(C)node[midway,left]{$\circledast_\mu$};\draw[->](B)--(D)node[midway,right]{$\circledast_\nu$};\draw[->](C)--(D)node[midway,below,scale=.9]{$\mathcal{V}(f)^{\operatorname{op}}$};\draw[double distance=1.5pt,-{Implies}](.25,.25)--(-.25,-.25)node[midway,auto=right,inner sep=1pt,scale=.9]{$\circledast_f$};\end{tikzpicture}\end{equation}\par\begin{lemma}The strong transformation $\circledast^{-1}\coloneqq\operatorname{op}\mathbin{\smallstar}\circledast^{\operatorname{co}}\mathbin{\smallstar}\mathop{-}\nolimits$ is a quasi-inverse of $\circledast$ in the 2-category $\mathrm{Bicat}_\Bbbk(\mathscr{U},\mathrm{Cat}_\Bbbk^{\operatorname{co}})$. In particular, $\circledast_\mu^{\operatorname{op}}$ is a quasi-inverse of $\circledast_\mu$ for each $\mu\in X$.\end{lemma} \begin{proof}Applying the 2-functor $(-^{\operatorname{co}})^\ast$ (see \S \ref{J}) to the invertible modification from \eqref{UA} to $\textnormal{\fontfamily{pbk}\selectfont1}_{\mathcal{V}}$ gives $(\operatorname{op}^{\operatorname{co}}\mathbin{\smallstar}\circledast\mathbin{\smallstar}\mathop{-}\nolimits^{\operatorname{co}})\circ\circledast^{\operatorname{co}}\cong1_{\mathcal{V}^{\operatorname{co}}}$. Applying the 2-functor $\operatorname{op}_{\ast}$ to \eqref{UA} gives $\circledast\circ(\operatorname{op}\mathbin{\smallstar}\circledast^{\operatorname{co}}\mathbin{\smallstar}\mathop{-}\nolimits)\cong1_{\mathcal{V}}$. The last assertion follows from $(\operatorname{op}\mathbin{\smallstar}\circledast^{\operatorname{co}}\mathbin{\smallstar}\mathop{-}\nolimits)_\mu=\circledast_\mu^{\operatorname{op}}$.\end{proof}\par For an abelian category $\mathcal{A}$ (resp.\ a small Krull--Schmidt category $\mathcal{K}$) with a duality functor, we denote by $\mathbf{B}^{\textup{o}}(\mathcal{A})\subseteq\mathbf{B}(\mathcal{A})$ (resp. $\mathbf{G}^{\textup{o}}(\mathcal{K})\subseteq\mathbf{G}(\mathcal{K})$) the subset of self-dual classes, with the duality functor understood implicitly from the context.\par\begin{proposition}\thlabel{WA} Let $\mathcal{C}\colon\mathscr{U}\to\mathrm{Ab}_\Bbbk$ be a $\Bbbk$-linear pseudofunctor equipped with a duality $\circledast$. Then the crystal structure on $\mathbf{B}(\mathcal{C}_\mu^{\operatorname{int}})$ of \thref{@} restricts to $\mathbf{B}^{\textup{o}}(\mathcal{C}_\mu^{\operatorname{int}})$.\end{proposition}\par\begin{proof}Let $M\in\mathbf{B}^{\textup{o}}(\mathcal{C}_\mu^{\operatorname{int}})$. By \thref{N} (i)\ensuremath{^{\prime}}\ applied to $\circledast$, there exists a strong transformation $\dot{\circledast}\colon\dot{\mathcal{C}}^{\operatorname{co}}\circ-\to\operatorname{op}\circ\dot{\mathcal{C}}$ satisfying $\circledast=\dot{\circledast}\mathbin{\smallstar}(\iota_{\mathcal{C}})^{\operatorname{co}}$. Then we have an isomorphism $F_i^{(m)}M\cong F_i^{(m)}M^{\circledast}\cong(F_i^{(m)}M)^{\circledast}$, where the first isomorphism is from the self-duality of $M$, and the second isomorphism is from the natural isomorphism $\dot{\circledast}_{F_i^{(m)}}$. Hence, $F_i^{(m)}M$ is self-dual, and $(\operatorname{hd}F_i^{(m)}M)^{\circledast}\cong\operatorname{soc}F_i^{(m)}M\cong q_i^{2m(\varphi_i(M)-m)}\operatorname{hd}F_i^{(m)}M$ by \thref{@}. By putting $m=1$, $\widetilde{f}_iM$ is self-dual, and similarly $\widetilde{e}_iM$ is self-dual.\end{proof}\par For the rest of this section, we assume that $\Bbbk$ is a commutative Henselian local ring with maximal ideal $\mathfrak{m}$. Let $\Bbbk\operatorname{\!-\!}\operatorname{mod}_{\operatorname{fl}}$ be the category of finite length $\Bbbk$-modules. Let $\ast\colon\Bbbk\operatorname{\!-\!}\operatorname{mod}_{\operatorname{fl}}\allowbreak\to(\Bbbk\operatorname{\!-\!}\operatorname{mod}_{\operatorname{fl}})^{\operatorname{op}}$ ($V\mapsto V^\ast\coloneqq\operatorname{Hom}_\Bbbk(V,E)$) be the duality on $\Bbbk\operatorname{\!-\!}\operatorname{mod}_{\operatorname{fl}}$, where $E\in\Bbbk\operatorname{\!-\!}\operatorname{mod}$ is the injective envelope of $\Bbbk/\mathfrak{m}$.\par Let $\wedge_f\colon\mathrm{Cat}_\Bbbk\to\mathrm{Cat}_\Bbbk^{\operatorname{coop}}$ denote the 2-functor sending a category $\mathcal{K}$ to $\operatorname{mod}\hspace{-.35mm}\operatorname{_{lfl}-\hspace{-.8mm}}\mathcal{K}\coloneqq\mathrm{Cat}_\Bbbk(\mathcal{K}^{\operatorname{op}},\Bbbk\operatorname{\!-\!}\operatorname{mod}_{\operatorname{fl}})$. The duality functor $\ast$ induces a 2-natural transformation $\mathop{\lozenge}\nolimits\colon\operatorname{op}^{\operatorname{coop}}\circ\allowbreak\wedge_f\;\longrightarrow\;\wedge_f^{\operatorname{co}}\circ\operatorname{op}$ in $\mathrm{Bicat}_\Bbbk(\mathrm{Cat}_\Bbbk,\mathrm{Cat}_\Bbbk^{\operatorname{op}})$, whose component at $\mathcal{K}\in\mathrm{Cat}_\Bbbk$ is given by a functor $\mathop{\lozenge}\nolimits_\mathcal{K}\colon\mathrm{Cat}_\Bbbk(\mathcal{K},\Bbbk\operatorname{\!-\!}\operatorname{mod}_{\operatorname{fl}})\to\mathrm{Cat}_\Bbbk(\mathcal{K},\Bbbk\operatorname{\!-\!}\operatorname{mod}_{\operatorname{fl}}^{\operatorname{op}})\cong\mathrm{Cat}_\Bbbk(\mathcal{K}^{\operatorname{op}},\allowbreak\Bbbk\operatorname{\!-\!}\operatorname{mod}_{\operatorname{fl}})^{\operatorname{op}}$, where the first arrow is a postcomposition with $\ast$.\par The canonical isomorphisms $V\to V^{\ast\ast}$ for $V\in\Bbbk\operatorname{\!-\!}\operatorname{mod}_{\operatorname{fl}}$ induce an invertible modification \begin{equation}\label{XA}\textnormal{\fontfamily{pbk}\selectfont1}_{\wedge_f}\;\longrightarrow\;(\mathop{\lozenge}\nolimits^{\operatorname{co}}\mathbin{\smallstar}\operatorname{op})\circ(\operatorname{op}^{\operatorname{op}}\mathbin{\smallstar}\mathop{\mathop{\lozenge}\nolimits}),\end{equation} in $\mathrm{Bicat}_\Bbbk(\mathrm{Cat}_\Bbbk,\mathrm{Cat}_\Bbbk^{\operatorname{coop}})$: The component of this modification at $\mathcal{K}\in\mathrm{Cat}_\Bbbk$ is given by a natural isomorphism of functors $\textnormal{\fontfamily{pbk}\selectfont1}_{\operatorname{mod}\hspace{-.35mm}\operatorname{_{lfl}-\hspace{-.8mm}}\mathcal{K}}\to\mathop{\lozenge}\nolimits_\mathcal{K}^{\operatorname{op}}\circ\mathop{\lozenge}\nolimits_\mathcal{K}$, defined by $M(P)\to M(P)^{\ast\ast}$ for each $M\in\operatorname{mod}\hspace{-.35mm}\operatorname{_{lfl}-\hspace{-.8mm}}\mathcal{K}$ and $P\in\mathcal{K}$.\par Let $\mathcal{K}:\mathscr{U}\to\mathrm{Cat}_\Bbbk$ be a pseudofunctor. As with $\operatorname{mod}\!\operatorname{-}\hspace{-.8mm}\mathcal{K}$ \eqref{GA}, we define a pseudofunctor $\operatorname{mod}\hspace{-.35mm}\operatorname{_{lfl}-\hspace{-.8mm}}\mathcal{K}\colon\mathscr{U}\to\mathrm{Cat}_\Bbbk$ as $\operatorname{mod}\hspace{-.35mm}\operatorname{_{lfl}-\hspace{-.8mm}}\mathcal{K}\coloneqq\wedge_f^{\operatorname{coop}}\circ\mathcal{K}^{\operatorname{coop}}\circ\tau$.\par\begin{lemma}\thlabel{YA} If $\circledast$ is a duality of a pseudofunctor $\mathcal{K}\colon\mathscr{U}\to\mathrm{Cat}_\Bbbk$, then \begin{equation}\label{ZA}\big((\mathop{\lozenge}\nolimits^{\operatorname{coop}}\mathbin{\smallstar}\mathcal{K})\circ(\wedge_f^{\operatorname{op}}\mathbin{\smallstar}(\circledast^{-1})^{\operatorname{coop}})\big)\mathbin{\smallstar}\tau\ \colon\ (\operatorname{mod}\hspace{-.35mm}\operatorname{_{lfl}-\hspace{-.8mm}}\mathcal{K})^{\operatorname{co}}\circ-\ \longrightarrow\ \operatorname{op}\circ\:(\operatorname{mod}\hspace{-.35mm}\operatorname{_{lfl}-\hspace{-.8mm}}\mathcal{K})\end{equation} is a duality of $\operatorname{mod}\hspace{-.35mm}\operatorname{_{lfl}-\hspace{-.8mm}}\mathcal{K}$.\end{lemma}\par When a pseudofunctor $\mathcal{K}$ is equipped with a duality $\circledast$, we implicitly regard $\operatorname{mod}\hspace{-.35mm}\operatorname{_{lfl}-\hspace{-.8mm}}\mathcal{K}$ as equipped with the duality \eqref{ZA}. Note that the component of the strong transformation \eqref{ZA} at $\mu\in X$ is given by \[\operatorname{mod}\hspace{-.35mm}\operatorname{_{lfl}-\hspace{-.8mm}}\mathcal{K}_\mu\xrightarrow{(\circledast_\mu)^\ast}\operatorname{mod}\hspace{-.35mm}\operatorname{_{lfl}-\hspace{-.8mm}}(\mathcal{K}^{\operatorname{op}}_\mu)\xrightarrow{\mathop{\lozenge}\nolimits_{\mathcal{K}_\mu}}(\operatorname{mod}\hspace{-.35mm}\operatorname{_{lfl}-\hspace{-.8mm}}\mathcal{K}_\mu)^{\operatorname{op}}.\]\par\begin{proof}The following shows that the composition \eqref{UA} for the strong transformation in \eqref{ZA} is isomorphic to the identity. \begin{align*}&\big(\operatorname{op}^{\operatorname{co}}\mathbin{\smallstar}((\mathop{\lozenge}\nolimits^{\operatorname{coop}}\mathbin{\smallstar}\mathcal{K})\circ(\wedge_f^{\operatorname{op}}\mathbin{\smallstar}(\circledast^{-1})^{\operatorname{coop}}))\mathbin{\smallstar}\tau\big)\circ\big(((\mathop{\lozenge}\nolimits^{\operatorname{op}}\mathbin{\smallstar}\mathcal{K}^{\operatorname{op}})\circ(\wedge_f^{\operatorname{coop}}\mathbin{\smallstar}(\circledast^{-1})^{\operatorname{op}}))\mathbin{\smallstar}\tau^{\operatorname{co}}\mathbin{\smallstar}\mathop{-}\nolimits\big)\\=\hspace{.2em}&\big(((\operatorname{op}^{\operatorname{co}}\mathbin{\smallstar}\mathop{\lozenge}\nolimits^{\operatorname{coop}})\circ(\mathop{\lozenge}\nolimits^{\operatorname{op}}\mathbin{\smallstar}\operatorname{op}^{\operatorname{coop}}))\mathbin{\smallstar}\mathcal{K}\mathbin{\smallstar}\tau\big)\circ\big(\wedge_f^{\operatorname{coop}}\mathbin{\smallstar}((\operatorname{op}^{\operatorname{op}}\mathbin{\smallstar}(\circledast^{-1})^{\operatorname{coop}})\circ((\circledast^{-1})^{\operatorname{op}}\mathbin{\smallstar}\mathop{-}\nolimits^{\operatorname{coop}}))\mathbin{\smallstar}\tau\big)\\\cong\hspace{.2em}&\textnormal{\fontfamily{pbk}\selectfont1}_{\operatorname{mod}\hspace{-.35mm}\operatorname{_{lfl}-\hspace{-.8mm}}\mathcal{K}}\circ\textnormal{\fontfamily{pbk}\selectfont1}_{\operatorname{mod}\hspace{-.35mm}\operatorname{_{lfl}-\hspace{-.8mm}}\mathcal{K}}=\textnormal{\fontfamily{pbk}\selectfont1}_{\operatorname{mod}\hspace{-.35mm}\operatorname{_{lfl}-\hspace{-.8mm}}\mathcal{K}}\end{align*} For the first equality, we have used \eqref{L} to interchange $(\circledast^{-1})^{\operatorname{coop}}$ and $\mathop{\lozenge}\nolimits^{\operatorname{op}}$ after expanding the parentheses. For the second isomorphism, we have used \eqref{XA} and the isomorphism of \eqref{UA} with the identity.\end{proof}\par\begin{theorem}\thlabel{aA} Let $\mathcal{K}\colon\mathscr{U}\to\mathrm{Cat}_\Bbbk$ be a pseudofunctor with a duality such that $\mathcal{K}_\mu$ is small and locally Laurentian with respect to $q\textnormal{\fontfamily{pbk}\selectfont1}_\mu$ for all $\mu\in X$. Then $\mathbf{G}^{\textup{o}}(\dot{\mathcal{K}})$ carries a canonical crystal structure, uniquely determined by \thref{JA}. The map $\mathbf{G}^{\textup{o}}(\dot{\mathcal{K}})\to\mathbf{B}^{\textup{o}}(\operatorname{mod}\hspace{-.35mm}\operatorname{_{lfl}-\hspace{-.8mm}}\mathcal{K})$, $P\mapsto\operatorname{hd}\mathcal{K}_\mu(-,P)$ is an isomorphism of crystals.\end{theorem} In particular, if each $\dot{\mathcal{K}}_\mu$ is a humorous category, the underlying set of this crystal is the orthodox basis. \begin{proof}The pseudofunctor $\dot{\mathcal{K}}$ given by extending $\mathcal{K}$ to additive Karoubi envelopes satisfies the conditions \ref{HA} (see \S\ref{`}) and \ref{IA} (by \thref{RA}) of \S\ref{FA}. \thref{JA} \ref{LA} uniquely determines $\widetilde{e}_i(P)$ for $P\in\mathbf{G}(\dot{\mathcal{K}})$, since if $Q\coloneqq\widetilde{e}_i(P)$ is nonzero, then $q^dQ\not\cong Q$ for $d\ne0$ by the locally left-bounded condition.\par Let $L\in\mathbf{B}(\operatorname{mod}\!\operatorname{-}\hspace{-.8mm}\mathcal{K}_\mu)$. Since $L$ is cyclic, $L(x)$ is a finitely generated $\Bbbk$-module for all $x\in\mathcal{K}_\mu$. By Nakayama's lemma, $\mathfrak{m}L=0$, and each $L(x)$ is a finite-dimensional vector space over $\Bbbk/\mathfrak{m}$. Hence $\mathbf{B}(\operatorname{mod}\hspace{-.35mm}\operatorname{_{lfl}-\hspace{-.8mm}}\mathcal{K}_\mu)=\mathbf{B}(\operatorname{mod}\!\operatorname{-}\hspace{-.8mm}\mathcal{K}_\mu)$. By \thref{WA}, $\mathbf{B}^{\textup{o}}(\operatorname{mod}\hspace{-.35mm}\operatorname{_{lfl}-\hspace{-.8mm}}\mathcal{K})$ is a subcrystal of $\mathbf{B}(\operatorname{mod}\hspace{-.35mm}\operatorname{_{lfl}-\hspace{-.8mm}}\mathcal{K})$.\par It remains to show that it is mapped to $\mathbf{G}^{\textup{o}}(\dot{\mathcal{K}})$ under the bijection $\mathbf{G}(\dot{\mathcal{K}})\to\mathbf{B}(\operatorname{mod}\!\operatorname{-}\hspace{-.8mm}\mathcal{K})$. Let $P\in\mathbf{G}^{\textup{o}}(\dot{\B}_\mu)$ and $L\coloneqq\operatorname{hd}\mathcal{K}_\mu(-,P)\in\operatorname{mod}\hspace{-.35mm}\operatorname{_{lfl}-\hspace{-.8mm}}\mathcal{K}_\mu$. Then for $Q\in\mathbf{G}(\dot{\B}_\mu)$, $L^\circledast(Q)=L(Q^\circledast)^\ast$ is nonzero only if $Q\cong P^\circledast\cong P$. Hence, $\mathcal{K}_\mu(-,P)$ is the projective cover of $L^\circledast$, and thus $L^\circledast\cong L\in\mathbf{B}^{\textup{o}}(\operatorname{mod}\hspace{-.35mm}\operatorname{_{lfl}-\hspace{-.8mm}}\mathcal{K}_\mu)$.\end{proof}\par We now construct two 2-representations with dualities, which are our main object of study in \S\ref{G}.\par\begin{example}\thlabel{EA} The bar-involution of $\mathscr{U}$ induces functors $\mathscr{U}(\lambda,\mu)\to\mathscr{U}(\lambda,\mu)^{\operatorname{op}}$, and for each $u\in\mathscr{U}(\mu_1,\mu_2)$, we have a commutative diagram \vspace{-.7em} \begin{equation}\label{bA}\begin{tikzpicture}[xscale=2,yscale=.7,baseline=-10pt]\node(A)at(-1,1){$\mathscr{U}(\lambda,\mu_1)$};\node(B)at(1,1){$\mathscr{U}(\lambda,\mu_2)$};\node(C)at(-1,-1){$\mathscr{U}(\lambda,\mu_1)^{\operatorname{op}}$};\node(D)at(1,-1){$\mathscr{U}(\lambda,\mu_2)^{\operatorname{op}}$};\draw[->](A)--(B)node[midway,above]{$\bar{u}\circ-$};\draw[->](A)--(C)node[midway,left]{$\mathop{-}\nolimits$};\draw[->](B)--(D)node[midway,right]{$\mathop{-}\nolimits$};\draw[->](C)--(D)node[midway,below]{$u^{\operatorname{op}}\circ-$};\end{tikzpicture}\vspace{-.4em}\end{equation} in $\mathrm{Cat}_\Bbbk$. Clearly, these data form a 2-natural transformation $\mathscr{U}(\lambda,-)^{\operatorname{co}}\circ-\to\operatorname{op}\circ\mathscr{U}(\lambda,-)$. Since $\mathop{-}\nolimits$ is an involution, this is a duality in the sense of \thref{TA}. The category $\mathscr{U}(\lambda,\mu)$ is humorous (cf.\ \cite[Thm.~5.17]{V} and \thref{CA}) with respect to a duality given by $\mathop{-}\nolimits$. Therefore, \thref{aA} applies to the 2-functor $\mathscr{U}(\lambda,-)$, providing a crystal structure on the set $\bigsqcup_{\mu\in X}\mathbf{G}^{\textup{o}}(\dot{\mathscr{U}}(\lambda,\mu))$.\end{example}\par\begin{example}\thlabel{cA} Similarly, for each $u\in\mathscr{U}(\lambda_1,\lambda_2)$, we have a commutative diagram \vspace{-.7em} \[\begin{tikzpicture}[xscale=2,yscale=.7,baseline=-10pt]\node(A)at(-1,1){$\mathscr{U}(-\lambda_1,\mu)$};\node(B)at(1,1){$\mathscr{U}(-\lambda_2,\mu)$};\node(C)at(-1,-1){$\mathscr{U}(-\lambda_1,\mu)^{\operatorname{op}}$};\node(D)at(1,-1){$\mathscr{U}(-\lambda_2,\mu)^{\operatorname{op}}$};\draw[->](A)--(B)node[midway,above]{$-\circ(\bar{u})^\ast$};\draw[->](A)--(C)node[midway,left]{$\mathop{-}\nolimits$};\draw[->](B)--(D)node[midway,right]{$\mathop{-}\nolimits$};\draw[->](C)--(D)node[midway,below]{$-\circ(u^\ast)^{\operatorname{op}}$};\end{tikzpicture}\vspace{-.4em}\] where the two paths of the diagram send $v\in\mathscr{U}(-\lambda_1,\mu)$ to $\overline{v(\bar{u})^\ast}=\bar{v}u^\ast$. These data provide a duality of the 2-representation $\mathscr{U}(-,\mu)\circ\ast\colon\mathscr{U}(c^\ast)\to\mathscr{U}(c)^{\operatorname{op}}\to\mathrm{Cat}_\Bbbk$. Therefore, \thref{aA} applies to the 2-functor $\mathscr{U}(-,\mu)\circ\ast$ to provide a crystal structure on $\bigsqcup_{\lambda\in X}\mathbf{G}^{\textup{o}}(\dot{\mathscr{U}}(\lambda,\mu))$, which we refer to as the \emph{star crystal} structure. We denote by $E_i^\ast$ and $F_i^\ast$ the functors $E_i$ and $F_i$ with respect to $\mathscr{U}(-,\mu)\circ\ast$ and $\operatorname{mod}\!\operatorname{-}\hspace{-.8mm}(\mathscr{U}(-,\mu)\circ\ast)$, and write $\widetilde{e}_i^\ast$ and $\widetilde{f}_i^\ast$ for the corresponding crystal operators. Note that the object-wise component of the duality functor on $\mathscr{U}(\lambda,\mu)$ is the same as the one in \thref{EA}, so the notation $\mathbf{G}^{\textup{o}}(\dot{\mathscr{U}}(\lambda,\mu))$ is unambiguous.\end{example}\par\section{Generalized tensor product categorifications}\label{F}\par For the rest of the paper, we assume that $X$ is an abelian group, $X\times\mathbb{Z}[I]\to X$ and $\langle i,-\rangle$ ($i\in I$) are group homomorphisms, and the action of $\mathbb{Z}[I]$ on $X$ is free. We identify $\mathbb{Z}[I]$ as a subset of $X$.\par\subsection{Tensor product of crystals}\label{dA}\par Our convention for the tensor product of crystals is opposite to the standard one in \cite[(1.5.13)--(1.5.16)]{O}. Under this convention, the tensor product of crystals $B_1$ and $B_2$ is the set $B_1\otimes B_2=\left\{\kern.07emb_1\otimes b_2\,\middle|\,b_1\in B_1,b_2\in B_2\kern.07em\right\}$ endowed with a crystal structure as follows: \begin{align}\operatorname{wt}(b_1\otimes b_2)&\coloneqq\operatorname{wt}(b_1)+\operatorname{wt}(b_2),\\\varepsilon_i(b_1\otimes b_2)&\coloneqq\max(\varepsilon_i(b_2),\varepsilon_i(b_1)-\langle i,\operatorname{wt}(b_2)\rangle),\\\varphi_i(b_1\otimes b_2)&\coloneqq\max(\varphi_i(b_1),\varphi_i(b_2)+\langle i,\operatorname{wt}(b_1)\rangle),\label{eA}\\\widetilde{e}_i(b_1\otimes b_2)&\coloneqq\begin{cases}\widetilde{e}_ib_1\otimes b_2,&\text{if }\varepsilon_i(b_1)>\varphi_i(b_2)\\b_1\otimes\widetilde{e}_ib_2,&\text{if }\varepsilon_i(b_1)\le\varphi_i(b_2),\end{cases}\label{fA}\\\widetilde{f}_i(b_1\otimes b_2)&\coloneqq\begin{cases}\widetilde{f}_ib_1\otimes b_2,&\text{if }\varepsilon_i(b_1)\ge\varphi_i(b_2)\\b_1\otimes\widetilde{f}_ib_2,&\text{if }\varepsilon_i(b_1)<\varphi_i(b_2).\end{cases}\end{align}\par Recall that a \emph{seminormal} crystal is a crystal in which $\varepsilon_i(b)=\max\{k\ge0\mid\widetilde{e}_i^kb\ne0\}$ and $\varphi_i(b)=\max\{k\ge0\mid\widetilde{f}_i^kb\ne0\}$ hold for all $b$ and $i$. Let $b_1,b_2,b_1'$ be elements of seminormal $\mathfrak{sl}_2$-crystals. Then \begin{equation}\label{gA}\widetilde{e}(b_1\otimes b_2)=b_1\otimes\widetilde{e}b_2\quad\Longrightarrow\quad\operatorname{wt}b_1'-\operatorname{wt}b_1\le\varphi(b_1'\otimes b_2)-\varphi(b_1\otimes b_2).\end{equation} To see this, assume the hypothesis of \eqref{gA}. Then the upper branch of \eqref{fA} cannot be taken, so $\varepsilon(b_1)\le\varphi(b_2)$. Applying \eqref{eA}, we obtain $\varphi(b_1\otimes b_2)=\varphi(b_2)+\operatorname{wt}b_1$ and $\varphi(b_1'\otimes b_2)\ge\varphi(b_2)+\operatorname{wt}b_1'$, yielding the desired inequality. Note that if $\widetilde{e}(b_1\otimes b_2)=0$, then $b_1\otimes\widetilde{e}b_2=0$, so the hypothesis of \eqref{gA} is fulfilled.\subsection{Upper finite fully stratified categories}\label{hA}\par Let $\mathcal{C}$ be a category equivalent to $\operatorname{mod}\!\operatorname{-}\hspace{-.8mm}\mathcal{K}$ for a small Krull--Schmidt category $\mathcal{K}$, or equivalently, a Grothendieck category with a generating set of projective objects $\left\{\kern.07emP(b)\,\middle|\,b\in\mathbf{B}\kern.07em\right\}$ indexed by a set $\mathbf{B}$ such that each $P(b)$ has a simple essential head $L(b)$. Following \cite{l}, we refer to such a category as a \emph{semiperfect Grothendieck category}.\par Let $(\Xi,\ge)$ be an upper finite poset, i.e., a poset with the set $\left\{\kern.07em\xi'\in\Xi\,\middle|\,\xi'\ge\xi\kern.07em\right\}$ finite for all $\xi\in\Xi$. Suppose that a \emph{stratification function} $\varrho\colon\mathbf{B}\to\Xi$ is given. For $\xi\in\Xi$, let $\mathcal{C}_{\le\xi}$ (resp. $\mathcal{C}_{<\xi}$) be the localizing subcategory of $\mathcal{C}$ consisting of objects $M$ such that if $L(b)$ is a simple subquotient of $M$, then $\varrho(b)\le\xi$ (resp. $\varrho(b)<\xi$). Let $\pi_\xi\colon\mathcal{C}_{\le\xi}\to\mathcal{C}_{\le\xi}/\mathcal{C}_{<\xi}\eqqcolon\mathcal{C}_\xi$ be the canonical quotient functor. Note that $\mathcal{C}_\xi$ is again a semiperfect Grothendieck category.\par If $\pi_\xi$ has an exact left adjoint functor $\Delta_\xi\colon\mathcal{C}_\xi\to\mathcal{C}_{\le\xi}$, it is called the \emph{standardization} functor. If standardization functors exist, we denote $\bar{\Delta}(b)\coloneqq\Delta_{\varrho(b)}(\bar{L}(b))$ and $\Delta(b)\coloneqq\Delta_{\varrho(b)}(\bar{P}(b))$, where $\bar{L}(b)\coloneqq\pi_{\varrho(b)}(L(b))$ and $\bar{P}(b)$ denotes the projective cover of $\bar{L}(b)$ in $\mathcal{C}_{\varrho(b)}$.\par\begin{definition}\thlabel{iA} $\mathcal{C}$ is called an \emph{upper finite fully stratified category} if the standardization functors $\Delta_\xi$ exist for all $\xi\in\Xi$, and for each $b\in\mathbf{B}$, there exists a projective object $P_b$ admitting a finite filtration with subquotients $\Delta(a)$ where $\varrho(a)\ge\varrho(b)$, and with top subquotient $\Delta(b)$. \end{definition}\par\begin{remark}The notion of an upper finite fully stratified category is due to Brundan and Stroppel \cite{W}, where it is an instance of their framework of $\varepsilon$-stratified categories. As in \cite{M}, we formulate the definition in terms of the exactness of the left adjoint of $\pi_\xi$. When each $\mathcal{C}_\xi$ is locally artinian (in the sense of \cite{m}), it is equivalent to the condition $(\widehat{P\Delta}_\varepsilon)$ in \cite{W} being satisfied for all sign functions $\varepsilon$ (cf. \cite[Thm.~3.33]{W}), which is the definition used in \cite[Def.~3.34]{W} in the Schurian setting. We do not assume that $\mathcal{C}$ is Schurian, nor that $\varrho$ has finite fibers, so that the ground ring $\Bbbk$ is an arbitrary commutative ring and the definition captures the graded lifts of the (generalized) tensor product categorifications (cf.\ \cite[\S1]{n}).\par\end{remark}\par The next lemma is a direct analogue of \cite[(2.1), (2.2)]{M}, and its proof is essentially contained in \cite[\S3.4]{W}. We include a proof for the reader's convenience.\par\begin{lemma}\thlabel{jA} Let $\xi,\xi'\in\Xi$. For $M\in\mathcal{C}_{\xi}$, $M'\in\mathcal{C}_{\xi'}$, and $i\ge0$, we have the following\kern.07em\kern.05em\textup{:}\begin{enumerate}[(i),font=\normalfont]\item $\operatorname{Ext}^i(\Delta_{\xi}(M),\Delta_{\xi'}(M'))=0$ if $\xi\not\le\xi'$. \label{kA}\item $\operatorname{Ext}^i(\Delta_\xi(M),\Delta_\xi(M'))\cong\operatorname{Ext}^i(M,M')$ if $\xi=\xi'$. \label{lA}\end{enumerate}\end{lemma}\par\begin{proof}Let $i_\xi\colon\mathcal{C}_{\le\xi}\to\mathcal{C}$ be the inclusion, $i^\ast_\xi\colon\mathcal{C}\to\mathcal{C}_{\le\xi}$ its left adjoint, and $\mathbb{L}i_\xi^\ast\colon D^-(\mathcal{C})\to D^-(\mathcal{C}_{\le\xi})$ the left derived functor. If $\xi\not\le\xi'$, then $i_{\xi'}^\ast P(b)=0$ for $b\in\varrho^{-1}(\xi)$. Since $P_b$ is a direct sum of $P(a)$ with $\varrho(a)\ge\xi$, we have $i_{\xi'}^\ast P_b=0$. There exists a resolution of $\Delta(b)$ by direct sums of $P_a$ with $\varrho(a)\ge\xi$, which implies $\mathbb{L}i_{\xi'}^\ast\Delta(b)=0$. Let $P_\bullet\to M$ be a projective resolution of $M$. Then $\Delta_\xi(P_\bullet)\to\Delta_\xi(M)$ is a resolution, which gives $\mathbb{L}i_{\xi'}^\ast\Delta_\xi(M)=0$. This proves \ref{kA}.\par Now suppose $\xi=\xi'$. Let $Q_b$ be the kernel of $P_b\twoheadrightarrow\Delta(b)$. Since it contains the kernel of $P_b\twoheadrightarrow i^\ast_\xi P_b$, the map $i^\ast_\xi Q_b\to i^\ast_\xi P_b$ is a monomorphism, which gives $\mathbb{L}_1i^\ast_\xi\Delta(b)=0$. For $i\ge2$, $\mathbb{L}_ii^\ast_\xi\Delta(b)\cong\mathbb{L}_{i-1}i^\ast_\xi Q_b\cong0$ by \ref{kA} and induction on $i$. Since $\Delta_\xi(P_\bullet)\to\Delta_\xi(M)$ is a resolution by $i^\ast_\xi$-acyclic objects, we have $\mathbb{L}i^\ast_\xi\Delta_\xi(M)\cong\Delta_\xi(M)$. Hence $\operatorname{Ext}^i(\Delta_\xi(M),\Delta_\xi(M'))$ can be computed in $\mathcal{C}_{\le\xi}$. The exactness of $\Delta_\xi$ and $\pi_\xi\Delta_\xi\cong1_{\mathcal{C}_\xi}$ yield \ref{lA}. \end{proof}\par\subsection{Structure of highest weight $\mathscr{U}$-representations}Fix $\lambda\in X$ and denote $R\coloneqq\mathscr{U}(\lambda,\lambda)|_{q=1}(\textnormal{\fontfamily{pbk}\selectfont1}_\lambda,\textnormal{\fontfamily{pbk}\selectfont1}_\lambda)$. For a graded ring map $R\to S$, we apply base change to obtain a 2-functor $S\otimes_R\mathscr{U}(\lambda,-)|_{q=1}$, which sends $\mu\in X$ to a category that is enriched over graded $\Bbbk$-modules. By taking their degree-zero component, we obtain a 2-functor $\big(S\otimes_R\mathscr{U}(\lambda,-)|_{q=1}\big)_0$.\par Let $\mathcal{K}\colon\mathscr{U}\to\mathrm{Cat}_\Bbbk$ be a pseudofunctor. For an object $P\in\mathcal{K}_\lambda$ we have a strong transformation $\widetilde{\phi}\colon\big(\kern-.15em\operatorname{End}_q(P)\otimes_R\mathscr{U}(\lambda,-)|_{q=1}\big)_0\to\mathcal{K}$, whose component at $\mu\in X$ is the functor $\phi_\mu$ given by $\phi_\mu(u)=uP$ for $u\in\mathscr{U}(\lambda,\mu)$ and $\phi_\mu(f\otimes g)=gP\circ u(f)$ for $f\in\operatorname{End}_q(P)$ and $g\in\mathscr{U}(\lambda,\mu)$; whose component at $u\in\mathscr{U}(\mu,\nu)$ is the natural transformation, whose components are given by $\mathcal{K}^2_{\lambda,\mu,\nu}(P)$ (see \eqref{K}).\par\newcommand{\E}{\mathcal{V}} \newcommand{\F}{\p{\E}}\par For $\mu\in X$, let $\widetilde{\mathcal{V}}^\lambda_\mu$ be the quotient of $\mathscr{U}(\lambda,\mu)$ by the ideal generated by the identity morphisms of $uE_i\textnormal{\fontfamily{pbk}\selectfont1}_\lambda$ for all $i\in I$ and $u\in\mathscr{U}(\lambda+i,\mu)$. Since this ideal is stable under left multiplication by 1-morphisms, we have a 2-functor $\widetilde{\mathcal{V}}^\lambda\colon\mathscr{U}\to\mathrm{Cat}_\Bbbk$ and the 2-natural transformation $\mathscr{U}(\lambda,-)\to\widetilde{\mathcal{V}}^\lambda$.\par\begin{lemma}\thlabel{mA} If $E_iP=0$ for all $i\in I$, then $\widetilde{\phi}$ factors through $\phi\colon\big(\kern-.15em\operatorname{End}_q(P)\otimes_R\widetilde{\mathcal{V}}^\lambda|_{q=1}\big)_0\to\mathcal{K}$, and each $\phi_\mu$ is fully faithful.\end{lemma}\par\begin{proof}It is a straightforward graded analogue of \cite[Prop.~5.5]{F} or \cite[Cor.~3.27]{K}.\end{proof}\par For a pseudofunctor $\mathcal{K}\colon\mathscr{U}\to\mathrm{Cat}_\Bbbk$, define a pseudofunctor $\mathcal{K}^\wedge$ by \begin{equation}\label{nA}\begin{tikzcd}\mathcal{K}^{\wedge}\coloneqq\wedge^{\operatorname{coop}}\circ\mathcal{K}^{\operatorname{coop}}\circ\rho^{\operatorname{co}}\circ\mathop{-}\nolimits\colon\mathscr{U}\arrow[r,"\rho^{\operatorname{co}}\circ\y"]&\mathscr{U}^{\operatorname{coop}}\arrow[r,"{\mathcal{K}^{\operatorname{coop}}}"]&\mathrm{Cat}_\Bbbk^{\operatorname{coop}}\arrow[r,"{\wedge^{\operatorname{coop}}}"]&\mathrm{Ab}_\Bbbk.\end{tikzcd}\end{equation}\par\begin{remark}\thlabel{oA} Let $n\colon X\to\mathbb{Z}$ be a function satisfying $n(\mu+i)-n(\mu)=\frac{i\cdot i}{2}(1+\langle i,\mu\rangle)$ for all $\mu\in X$ and $i\in I$. Such $n$ always exists and is determined by its value at a single point for each $\mathbb{Z}[I]$-orbit of $X$. There exists a 2-natural isomorphism $\psi\colon\operatorname{mod}\!\operatorname{-}\hspace{-.8mm}\mathcal{K}\to\mathcal{K}^\wedge$ given by $\psi_\mu(M)=M\circ q^{n(\mu)}\textnormal{\fontfamily{pbk}\selectfont1}_\mu$ for $\mu\in X$ and $M\in\operatorname{mod}\!\operatorname{-}\hspace{-.8mm}\mathcal{K}_\mu$. \end{remark}\par Let $\mathcal{C}\colon\mathscr{U}\to\mathrm{Ab}_\Bbbk$ be a pseudofunctor, such that each $\mathcal{C}_\mu$ is a semiperfect Grothendieck category. Let $\mathcal{C}_{\mu}\operatorname{\!-\!}\operatorname{fgproj}$ denote the full subcategory of finitely generated projective objects in $\mathcal{C}_{\mu}$.\par\begin{lemma}\thlabel{pA} The action of $\mathscr{U}$ restricts to $\{\mathcal{C}_{\mu}\operatorname{\!-\!}\operatorname{fgproj}\}_{\mu\in X}$. Let $\mathcal{C}\operatorname{\!-\!}\operatorname{fgproj}$ be the resulting pseudofunctor. Then we have an equivalence of 2-representations $\mathcal{C}\simeq(\mathcal{C}\operatorname{\!-\!}\operatorname{fgproj})^{\wedge}$.\end{lemma} \begin{proof}The functors $E_i$ and $F_i$ have right adjoints which themselves have right adjoints. It formally follows that they preserve the projective (resp. finitely generated) objects of $\mathcal{C}_\mu$.\par We construct a strong transformation $\mathcal{Y}\colon\mathcal{C}\to(\mathcal{C}\operatorname{\!-\!}\operatorname{fgproj})^{\wedge}$. For an object $\mu\in\mathscr{U}$, $\mathcal{Y}_\mu\colon\mathcal{C}_\mu\to\operatorname{mod}\!\operatorname{-}\hspace{-.8mm}(\mathcal{C}_\mu\operatorname{\!-\!}\operatorname{fgproj})$ is the Yoneda embedding; For a 1-morphism $u\in\mathscr{U}(\mu,\nu)$ and an object $M\in\mathcal{C}_\mu$, let $\mathcal{Y}_u(M)\colon u\mathcal{C}_\mu(-,M)\to\mathcal{C}_\nu(-,uM)$ be the natural transformation whose component at $P\in\mathcal{C}_\nu\operatorname{\!-\!}\operatorname{fgproj}$ is the map $(\rho(\overline{u})P,M)\to(P,uM)$, $f\mapsto u(f)\circ\eta_u(P)$ (see \S\ref{Y}). This defines a natural transformation $\mathcal{Y}_u$. We omit the verification of the coherence conditions for $\mathcal{Y}_\mu$ and $\mathcal{Y}_u$.\par Each $\mathcal{Y}_\mu$ is an equivalence of categories, since $\mathcal{C}_\mu$ is a semiperfect Grothendieck category \cite[Thm.~A]{l}. This implies that $\mathcal{Y}$ is an equivalence, since the invertibility of a strong transformation can be checked object-wise \cite[Prop.~6.2.16]{S}.\end{proof}\par Let $P\in\mathcal{C}_\lambda\operatorname{\!-\!}\operatorname{fgproj}$ be such that $E_iP=0$ for all $i\in I$ and $\left\{\kern.07emuP\,\middle|\,u\in\mathscr{U}(\lambda,\mu)\kern.07em\right\}$ generates $\mathcal{C}_\mu$ for all $\mu\in X$.\par\begin{corollary}\thlabel{qA} We have an equivalence of 2-representations $\mathcal{C}\cong\mathcal{L}^\wedge$, where $\mathcal{L}\coloneqq\big(\kern-.15em\operatorname{End}_q(P)\otimes_R\widetilde{\mathcal{V}}^\lambda|_{q=1}\big)_0$.\end{corollary} \begin{proof}We have a strong transformation $\phi\colon\mathcal{L}\to\mathcal{C}\operatorname{\!-\!}\operatorname{fgproj}$ by \thref{mA}. By the assumption on $P$, each $\operatorname{mod}\!\operatorname{-}\hspace{-.8mm}(\mathcal{C}_\mu\operatorname{\!-\!}\operatorname{fgproj})\to\operatorname{mod}\!\operatorname{-}\hspace{-.8mm}\mathcal{L}_\mu$ is essentially surjective, and hence is an equivalence of categories. Therefore, $\wedge^{\operatorname{coop}}\mathbin{\smallstar}\phi^{\operatorname{coop}}\mathbin{\smallstar}\rho^{\operatorname{co}}\mathbin{\smallstar}\mathop{-}\nolimits\colon(\mathcal{C}\operatorname{\!-\!}\operatorname{fgproj})^\wedge\to\mathcal{L}^\wedge$ is an equivalence of 2-representations. The desired equivalence follows from \thref{pA}.\end{proof}\par Let ${\vphantom{\mathcal{R}}\smash{\widetilde{\mathcal{R}}}}^\lambda_\mu$ be the full subcategory of $\widetilde{\mathcal{V}}^\lambda_\mu$ consisting of 1-morphisms $q^dF_{i_l}\cdots F_{i_1}\textnormal{\fontfamily{pbk}\selectfont1}_\lambda$ for $i_1,\ldots,i_l\in I$. This is the `universal categorification' of \cite{K}: ${\vphantom{\mathcal{R}}\smash{\widetilde{\mathcal{R}}}}^\lambda_\mu\hookrightarrow\widetilde{\mathcal{V}}^\lambda_\mu$ induces an equivalence of categories after passing to additive Karoubi envelopes \cite[Prop.~3.25]{K}, ${\vphantom{\mathcal{R}}\smash{\widetilde{\mathcal{R}}}}^\lambda_\mu|_{q=1}$ is a free $\operatorname{End}_q(\textnormal{\fontfamily{pbk}\selectfont1}_\lambda)$-module, and satisfies ${\vphantom{\mathcal{R}}\smash{\widetilde{\mathcal{R}}}}^\lambda_\mu|_{q=1}\otimes_R\Bbbk\cong\mathcal{R}^\lambda_\mu|_{q=1}$ \cite[Cor.~3.26]{K}.\par\begin{lemma}\thlabel{rA} Suppose that $M\coloneqq\operatorname{hd}P$ is simple and integrable with respect to $\mathcal{C}$. Then we have a strict embedding $L\colon B(\lambda)\hookrightarrow\mathbf{B}(\mathcal{C})$ of crystals with $L(u_\lambda)=M$, and every object of $\mathbf{B}(\mathcal{C}_\mu)$ is isomorphic to $q^nL(b)$ for a unique $b\in B(\lambda)_\mu$ and some $n\in\mathbb{Z}$.\end{lemma}\par\begin{proof}Let $N$ be a simple right $\mathcal{L}_\mu$-module. The two-sided ideal $I\coloneqq\big((\operatorname{rad}\operatorname{End}_q(P))\otimes_R{\vphantom{\mathcal{R}}\smash{\widetilde{\mathcal{R}}}}^\lambda_\mu|_{q=1}\big)_0$ of $\mathcal{L}_\mu$ generates a $\mathcal{L}_\mu$-submodule $IN$ of $N$. Since the hom spaces of ${\vphantom{\mathcal{R}}\smash{\widetilde{\mathcal{R}}}}^\lambda_\mu|_{q=1}$ are finitely generated $R$-modules, $\bigoplus_{d\in\mathbb{Z}}q^dN(x)$ is a finitely generated $\operatorname{End}_q(P)$-module for each $x\in{\vphantom{\mathcal{R}}\smash{\widetilde{\mathcal{R}}}}^\lambda_\mu$. By Nakayama's lemma, $IN=0$. Note that $\operatorname{End}_q(P)/\operatorname{rad}\operatorname{End}_q(P)\cong\operatorname{End}_q(M)$ and the map $R\to\operatorname{End}_q(M)$ factors through $\Bbbk$ by \thref{sA}\ref{tA} and \thref{*}. Hence, the action of $\mathcal{L}_\mu$ on $N$ factors through $\big(\kern-.15em\operatorname{End}_q(M)\otimes_\Bbbk\mathcal{R}^\lambda_\mu|_{q=1}\big)_0$. Let $D\coloneqq\operatorname{End}(M)$ be a division ring, and $d$ be the minimal positive integer such that $q^dM\cong M$, with $d=0$ if no such integer exists. Then $\big(\kern-.15em\operatorname{End}_q(M)\otimes_\Bbbk\mathcal{R}^\lambda_\mu|_{q=1}\big)_0\cong D\otimes_\Bbbk\mathcal{R}^\lambda_\mu|_{q^d=1}$.\par Let $\Bbbk'$ be the field inside $D$ generated by the image of $\Bbbk$. Then $\bigoplus(\Bbbk'\otimes_\Bbbk\mathcal{R}^\lambda_\mu|_{q=1})(F_{i_l}\cdots\allowbreak F_{i_1}\textnormal{\fontfamily{pbk}\selectfont1}_\lambda,F_{j_l}\cdots F_{j_1}\textnormal{\fontfamily{pbk}\selectfont1}_\lambda)\allowbreak\eqqcolon R'_\mu$, where the summation runs over all $i_1,\ldots,i_l,j_1,\ldots,j_l\in I$ with $\mu=\lambda-\sum_{k=1}^li_k=\lambda-\sum_{k=1}^lj_k$, is a graded finite-dimensional $\Bbbk'$-algebra. As its simple modules are absolutely simple and have graded lifts, $R'_\mu/\operatorname{rad}R'_\mu$ is isomorphic to a product of graded matrix algebras over $\Bbbk'$. Thus, base change induces a bijection $\mathbf{B}(\operatorname{mod}\!\operatorname{-}\hspace{-.8mm}\Bbbk'\otimes_\Bbbk\mathcal{R}^\lambda_\mu|_{q^d=1})\to\mathbf{B}(\operatorname{mod}\!\operatorname{-}\hspace{-.8mm}D\otimes_\Bbbk\mathcal{R}^\lambda_\mu|_{q^d=1})$, and forgetting the grading induces a surjective map $\mathbf{B}\big(\kern-.15em\operatorname{mod}\!\operatorname{-}\hspace{-.8mm}\Bbbk'\otimes_\Bbbk\mathcal{R}^\lambda_\mu\big)\to\mathbf{B}(\operatorname{mod}\!\operatorname{-}\hspace{-.8mm}\Bbbk'\otimes_\Bbbk\mathcal{R}^\lambda_\mu|_{q^d=1})$.\par Let $\mathscr{U}^-$ be the 2-category with the set of objects $X$, 1-morphisms generated by $F_i\textnormal{\fontfamily{pbk}\selectfont1}_\lambda$'s and their degree shifts, and 2-morphisms generated by $x$ and $\tau$ of \thref{Q}, subject only to the KLR relations. Define $\mathscr{U}^+$ analogously, and let $\jmath^\pm\colon\mathscr{U}^\pm\to\mathscr{U}$ be the canonical 2-functor. For a pseudofunctor $\mathcal{K}\colon\mathscr{U}^-\to\mathrm{Cat}_\Bbbk$, we define $\operatorname{mod}\!\operatorname{-}\hspace{-.8mm}\mathcal{K}$ and $\mathcal{K}^\wedge\colon\mathscr{U}^+\to\mathrm{Ab}_\Bbbk$ by \eqref{GA} and \eqref{nA} restricted to $\mathscr{U}^+$. Clearly, we have 2-functors $\mathcal{R}^\lambda,{\vphantom{\mathcal{R}}\smash{\widetilde{\mathcal{R}}}}^\lambda\colon\mathscr{U}^-\to\mathrm{Cat}_\Bbbk$, and inclusions ${\vphantom{\mathcal{R}}\smash{\widetilde{\mathcal{R}}}}^\lambda_\mu\hookrightarrow\widetilde{\mathcal{V}}^\lambda_\mu$ assemble into a 2-natural transformation ${\vphantom{\mathcal{R}}\smash{\widetilde{\mathcal{R}}}}^\lambda\to\widetilde{\mathcal{V}}^\lambda\circ\jmath^-$. Hence, we obtain strong transformations $\mathcal{C}\circ\jmath^+\xrightarrow{\sim}\big(\kern-.15em\operatorname{End}_q(P)\otimes_\Bbbk\widetilde{\mathcal{V}}^\lambda|_{q^d=1}\big)_0^\wedge\circ\jmath^+\leftarrow\big(D\otimes_\Bbbk\mathcal{R}^\lambda|_{q^d=1}\big)_0^\wedge\leftarrow\big(\Bbbk'\otimes_\Bbbk\mathcal{R}^\lambda|_{q^d=1}\big)_0^\wedge$, whose components induce bijections $\mathbf{B}(\mathcal{C}_\mu)\leftarrow\mathbf{B}(\operatorname{mod}\!\operatorname{-}\hspace{-.8mm}\Bbbk'\otimes_\Bbbk\mathcal{R}^\lambda_\mu|_{q^d=1})$ that commute with the operators $\widetilde{e}_i$ defined by \eqref{^} using $\mathscr{U}^+$-actions.\par The bar-involution induces a duality on $\widetilde{\mathcal{V}}^\lambda$, and hence on $\operatorname{mod}\!\operatorname{-}\hspace{-.8mm}\widetilde{\mathcal{V}}^\lambda$ by \thref{YA}. The induced object-wise duality functors $\operatorname{mod}\!\operatorname{-}\hspace{-.8mm}\widetilde{\mathcal{V}}^\lambda_\mu\to(\operatorname{mod}\!\operatorname{-}\hspace{-.8mm}\widetilde{\mathcal{V}}^\lambda_\mu)^{\operatorname{op}}$ restrict to $\operatorname{mod}\!\operatorname{-}\hspace{-.8mm}\mathcal{R}^\lambda_\mu\to(\operatorname{mod}\!\operatorname{-}\hspace{-.8mm}\mathcal{R}^\lambda_\mu)^{\operatorname{op}}$. It follows from \thref{aA} that the self-dual objects of $\mathbf{B}(\operatorname{mod}\!\operatorname{-}\hspace{-.8mm}\Bbbk'\otimes_\Bbbk\mathcal{R}^\lambda)$ are stable under the crystal operators $\widetilde{e}_i$. The map $\bigsqcup\mathbf{B}(\operatorname{mod}\!\operatorname{-}\hspace{-.8mm}\Bbbk'\otimes_\Bbbk\mathcal{R}^\lambda_\mu)\to\bigsqcup_{\mu\in X}\mathbf{B}(\operatorname{mod}\!\operatorname{-}\hspace{-.8mm}\Bbbk'\otimes_\Bbbk R'_\mu)$ is surjective, each fiber has a unique self-dual object, and the crystal operator $\widetilde{e}_i$ on the source commutes with $\tilde{e}_i^\lambda$ on the domain given by \cite[(5.9)]{E}. Thus, we have a bijection $\mathbf{B}(\operatorname{mod}\!\operatorname{-}\hspace{-.8mm}\Bbbk'\otimes_\Bbbk\mathcal{R}^\lambda)\to B(\lambda)\times\mathbb{Z}$ that commutes with $\widetilde{e}_i$'s. From the existence of a 2-natural isomorphism $\operatorname{mod}\!\operatorname{-}\hspace{-.8mm}\Bbbk'\otimes_\Bbbk\mathcal{R}^\lambda\cong\big(\Bbbk'\otimes_\Bbbk\mathcal{R}^\lambda\big)^\wedge$ (cf.\ \thref{oA}), we have a bijection $\mathbf{B}\big((\Bbbk'\otimes_\Bbbk\mathcal{R}^\lambda)^\wedge\big)\to B(\lambda)\times\mathbb{Z}$, and hence $\mathbf{B}\big((\Bbbk'\otimes_\Bbbk\mathcal{R}^\lambda|_{q^d=1})^\wedge\big)\to B(\lambda)\times\mathbb{Z}/d$. Since seminormal crystals are completely determined by $\widetilde{e}_i$'s, we conclude $\mathbf{B}(\mathcal{C})\cong B(\lambda)\times\mathbb{Z}/d$ as crystals.\end{proof}\subsection{Axioms}\label{uA}\par Fix ${\underline{{\lambda}}}=(\lambda_1,\ldots,\lambda_m)\in(X^+)^m$ and ${\underline{{\varepsilon}}}=(\varepsilon_1,\ldots,\varepsilon_m)\in\{-,+\}^m$, and set ${{\underline{{\varepsilon}}}\underline{{\lambda}}}\coloneqq(\varepsilon_1\lambda_1,\ldots,\allowbreak\varepsilon_m\lambda_m)\allowbreak\in X^m$. We define a notion of a categorification of the tensor product of $\widetilde{U}$-modules $V(\varepsilon_i\lambda_i)$ that extends the tensor product categorification of \cite{M}.\par Define a partial order on $X$ by $\mu\ge\mu'\Leftrightarrow\mu-\mu'\in\mathbb{N}[I]$, and a partial order on $X^m$ by ${\underline{{\mu}}}=(\mu_1,\ldots,\mu_m)\ge{\underline{{\mu}}}'=(\mu_1',\ldots,\mu_m')$ if and only if \begin{equation}\label{vA}\sum_{i=k}^m\varepsilon_i\mu_i\ge\sum_{i=k}^m\varepsilon_i\mu_i'\quad\text{for all }k=1,\ldots,m.\end{equation} For $\mu\in X$, define an upper finite poset $\Xi_\mu$ by \begin{equation}\label{wA}\Xi_\mu\coloneqq\left\{\kern.07em{\underline{{\mu}}}=(\mu_1,\ldots,\mu_m)\in X^m\,\middle|\,|{\underline{{\mu}}}|=\mu,\ \varepsilon_i\mu_i\le\lambda_i\ \text{ for all }i\kern.07em\right\},\end{equation} where $|{\underline{{\mu}}}|\coloneqq\sum_{i=1}^m\mu_i$, with the partial order induced by that of $X^m$.\par Consider the Cartan datum $(\coprod_{j=1}^mI,\cdot)$ (${}_ji\cdot{}_{j'}i'=\delta_{jj'}i\cdot i'$), the weight datum $(X^m,\{\langle{}_ji,\cdot\rangle\})$ ($\langle{}_ji,{\underline{{\lambda}}}\rangle=\langle i,\lambda_j\rangle$), the quiver Hecke datum $Q_{{}_ji,{}_{j'}i'}=\delta_{jj'}Q_{i,i'}+(1-\delta_{jj'})$, and the bubble parameters $c_{{}_ji}({\underline{{\lambda}}})=c_i(\lambda_j)$. Here ${}_ji$ for $i\in I$ and $j=1,\ldots,m$ denotes the element of $\coprod_{j=1}^mI$ corresponding to $i$ in the $j$-th component. Set $\mathscr{U}(\mathfrak{g})\coloneqq\mathscr{U}$, and let $\mathscr{U}(\mathfrak{g}^{\oplus m})$ be the 2-category $\mathscr{U}$ associated with the above data, with 1-morphisms $E_{{}_ji}$ and $F_{{}_ji}$ denoted by ${}_jE_i$ and ${}_jF_i$. We also write $\widetilde{U}(\mathfrak{g})$ and $\widetilde{U}(\mathfrak{g}^{\oplus m})$ for the algebra $\widetilde{U}$ with the respective data, and similarly for $\mathfrak{g}$- and $\mathfrak{g}^{\oplus m}$-crystals.\par\begin{definition}\thlabel{sA} A \emph{generalized tensor product categorification} of type ${{\underline{{\varepsilon}}}\underline{{\lambda}}}$ is a family $\{\mathcal{C}_\mu\}_{\mu\in X}$ of upper finite fully stratified categories with stratification functions $\varrho_\mu\colon\mathbf{B}_\mu\to\Xi_\mu$, equipped with \begin{enumerate}[(a)]\item an action of $\mathscr{U}(\mathfrak{g})$ on $\{\mathcal{C}_\mu\}_{\mu\in X}$, and \label{xA}\item an action of $\mathscr{U}(\mathfrak{g}^{\oplus m})$ on $\{\mathcal{C}_{\underline{{\mu}}}\}_{{\underline{{\mu}}}\in X^m}$, where $\mathcal{C}_{\underline{{\mu}}}\coloneqq(\mathcal{C}_{|{\underline{{\mu}}}|})_{\underline{{\mu}}}=(\mathcal{C}_{|{\underline{{\mu}}}|})_{\leq{\underline{{\mu}}}}/(\mathcal{C}_{|{\underline{{\mu}}}|})_{<{\underline{{\mu}}}}$,\label{yA}\end{enumerate} satisfying the following conditions. Let $\mathbf{B}_{\underline{{\mu}}}\coloneqq\varrho_{|{\underline{{\mu}}}|}^{-1}({\underline{{\mu}}})$. \begin{enumerate}[(i)]\item All simple objects of $\mathcal{C}_\mu$ (resp. $\mathcal{C}_{\underline{{\mu}}}$) are integrable with respect to the action \ref{xA} (resp. \ref{yA}). \label{tA}\item There exists $b_0\in\mathbf{B}_{{{\underline{{\varepsilon}}}\underline{{\lambda}}}}$ such that $\left\{\kern.07emu\bar{P}(b_0)\,\middle|\,u\in\mathscr{U}(\mathfrak{g}^{\oplus m})({{\underline{{\varepsilon}}}\underline{{\lambda}}},{\underline{{\mu}}})\kern.07em\right\}$ generates $\mathcal{C}_{\underline{{\mu}}}$ for all ${\underline{{\mu}}}\in X^m$. \label{zA}\item For each $M\in\mathcal{C}_{\underline{{\mu}}}$, the object $E_i\Delta_{\underline{{\mu}}}(M)$ admits a filtration with subquotients $q_i^{c_j}\Delta_{{\underline{{\mu}}}+{}_ji}({}_jE_iM)$ for $j=1,\ldots,m$, where $c_j=\sum_{k,\varepsilon_kk>\varepsilon_jj}\langle i,\mu_k\rangle$. Similarly, $F_i\Delta_{\underline{{\mu}}}(M)$ admits a filtration with subquotients $q_i^{d_j}\Delta_{{\underline{{\mu}}}-{}_ji}({}_jF_iM)$ for $j=1,\ldots,m$, where $d_j=-\sum_{k,\varepsilon_kk<\varepsilon_jj}\langle i,\mu_k\rangle$. \label{!A}\end{enumerate}\end{definition}\par Let us denote the pseudofunctor from the action \ref{yA} by $\mathcal{C}^{\mathop{\textup{gr}}}$.\par\begin{remark}Let $\psi^{\underline{{\varepsilon}}}\colon\mathscr{U}(\mathfrak{g}^{\oplus m})\to\mathscr{U}(\mathfrak{g}^{\oplus m})$ be the 2-functor defined on objects by ${\underline{{\mu}}}\mapsto{\underline{{\varepsilon}}}{\underline{{\mu}}}\coloneqq(\varepsilon_1\mu_1,\ldots,\allowbreak\varepsilon_m\mu_m)$, on 1-morphisms by $q\mapsto q$ and ${}_jE^\pm_i\mapsto{}_jE^{\pm\varepsilon_j}_i$, and on 2-morphisms by reversing the orientation of all strands labeled by ${}_ji$ with $\varepsilon_j=-1$, and then multiplying by $(-1)^{\text{\# of ${}_ji$-${}_{j'}i'$ crossings with $\varepsilon_j=-1$ and $\varepsilon_{j'}=-1$}}$. \bytes
$\lambda\in X$ and $\mu\in X^+$. Its 2-morphisms are generated by those of $\mathscr{U}$, and by the following additional 2-morphisms: \begin{gather*}\tikz[baseline,scale=1.5]{\draw[preaction={draw, black, line width=2.2pt},draw=gray,line width=1.5pt,-](.2,.3)--(-.2,-.1)node[at end,below,text height=1.62ex,text depth=.5ex,inner ysep=2pt,scale=.8]{$\mu$};\draw[very thick,baseline,->](-.2,.3)--(.2,-.1)node[at end,below,text height=1.62ex,text depth=.5ex,inner ysep=2pt,scale=.8]{$i$};}\colon q_i^{\langle i,\mu\rangle}F_i\mathfrak{I}_{\mu}\to\mathfrak{I}_{\mu}F_i,\quad\tikz[baseline,scale=1.5]{\draw[very thick,baseline,->](.2,.3)--(-.2,-.1)node[at end,below,text height=1.62ex,text depth=.5ex,inner ysep=2pt,scale=.8]{$i$};\draw[preaction={draw, black, line width=2.2pt},draw=gray,line width=1.5pt,-](-.2,.3)--(.2,-.1)node[at end,below,text height=1.62ex,text depth=.5ex,inner ysep=2pt,scale=.8]{$\mu$};}\colon q_i^{\langle i,\mu\rangle}\mathfrak{I}_{\mu}F_i\to F_i\mathfrak{I}_{\mu}\label{8A}\\\tikz[baseline,scale=1.5]{\draw[preaction={draw, black, line width=2.2pt},draw=gray,line width=1.5pt,decorate,decoration={snake, amplitude=1pt, segment length=10pt},-](.2,.3)--(-.2,-.1)node[at end,below,text height=1.62ex,text depth=.5ex,inner ysep=2pt,scale=.8]{$-\mu$};\draw[very thick,baseline,<-](-.2,.3)--(.2,-.1)node[at end,below,text height=1.62ex,text depth=.5ex,inner ysep=2pt,scale=.8]{$i$};}\colon q_i^{\langle i,\mu\rangle}E_i\mathfrak{I}_{-\mu}\to\mathfrak{I}_{-\mu}E_i,\quad\tikz[baseline,scale=1.5]{\draw[very thick,baseline,<-](.2,.3)--(-.2,-.1)node[at end,below,text height=1.62ex,text depth=.5ex,inner ysep=2pt,scale=.8]{$i$};\draw[preaction={draw, black, line width=2.2pt},draw=gray,line width=1.5pt,decorate,decoration={snake, amplitude=1pt, segment length=10pt},-](-.2,.3)--(.2,-.1)node[at end,below,text height=1.62ex,text depth=.5ex,inner ysep=2pt,scale=.8]{$-\mu$};}\colon q_i^{\langle i,\mu\rangle}\mathfrak{I}_{-\mu}E_i\to E_i\mathfrak{I}_{-\mu}.\label{9A}\end{gather*} The identity 2-morphism of $\mathfrak{I}_{\mu}$ (resp.\ $\mathfrak{I}_{-\mu}$) is depicted as an unoriented thick (resp.\ wavy) strand labelled $\mu$ (resp.\ $-\mu$). 2-morphisms satisfy all the relations \eqref{c}--\eqref{l} of $\mathscr{U}$, plus the relations involving $\mathfrak{I}_{\mu}$ below. \begin{enumerate}\item KLR relations: \begin{equation}\label{:A}\begin{tikzpicture}[scale=.5,very thick,baseline,baseline=-2pt]\draw[preaction={draw, black, line width=2.2pt},draw=gray,line width=1.5pt](-4,0)+(-1,-1)--+(1,1)node[below,text height=1.62ex,text depth=.5ex,inner ysep=2pt,at start]{$\mu$};\draw[postaction={ decorate, decoration={ markings, mark=at position .2 with {{\arrow[scale=.8]{<}}} } }][postaction={ decorate, decoration={ markings, mark=at position .75 with {\fill(0,0)circle(2.5pt);} } }](-4,0)+(1,-1)--+(-1,1)node[below,text height=1.62ex,text depth=.5ex,inner ysep=2pt,at start]{$i$};\node at(-2,0){$=$};\draw[preaction={draw, black, line width=2.2pt},draw=gray,line width=1.5pt](0,0)+(-1,-1)--+(1,1)node[below,text height=1.62ex,text depth=.5ex,inner ysep=2pt,at start]{$\mu$};\draw[postaction={ decorate, decoration={ markings, mark=at position .8 with {{\arrow[scale=.8]{<}}} } }][postaction={ decorate, decoration={ markings, mark=at position .25 with {\fill(0,0)circle(2.5pt);} } }](0,0)+(1,-1)--+(-1,1)node[below,text height=1.62ex,text depth=.5ex,inner ysep=2pt,at start]{$i$};\end{tikzpicture}\ ,\quad\begin{tikzpicture}[very thick,baseline,scale=.7,baseline=-1.5pt]\draw[postaction={ decorate, decoration={ markings, mark=at position .5 with {{\arrow[scale=.8]{<}}} } }](-2.8,-1)..controls(-1.2,0)..(-2.8,1)node[below,text height=1.62ex,text depth=.5ex,inner ysep=2pt,at start]{$i$};\draw[preaction={draw, black, line width=2.2pt},draw=gray,line width=1.5pt](-1.2,-1)..controls(-2.8,0)..(-1.2,1)node[below,text height=1.62ex,text depth=.5ex,inner ysep=2pt,at start]{$\mu$};\node at(-.5,0){$=$};\draw[postaction={ decorate, decoration={ markings, mark=at position .75 with {{\arrow[scale=.8]{<}}} } }][postaction={ decorate, decoration={ markings, mark=at position .5 with {\fill(0,0)circle(2.5pt);\node[left,xshift=-5pt,scale=.8]at(0,0){$\langle i,\mu\rangle$};} } }](1.4,0)+(0,-1)--+(0,1)node[below,text height=1.62ex,text depth=.5ex,inner ysep=2pt,at start]{$i$};\draw[preaction={draw, black, line width=2.2pt},draw=gray,line width=1.5pt](2.2,0)+(0,-1)--+(0,1)node[below,text height=1.62ex,text depth=.5ex,inner ysep=2pt,at start]{$\mu$};\end{tikzpicture}\ ,\quad\begin{tikzpicture}[very thick,baseline,scale=.7,baseline=-4pt]\draw[postaction={ decorate, decoration={ markings, mark=at position .8 with {{\arrow[scale=.8]{<}}} } }](0,0)+(-1,-1)--+(1,1)node[below,text height=1.62ex,text depth=.5ex,inner ysep=2pt,at start]{$i$};\draw[preaction={draw, black, line width=2.2pt},draw=gray,line width=1.5pt](0,0)+(1,-1)--+(-1,1)node[below,text height=1.62ex,text depth=.5ex,inner ysep=2pt,at start]{$\mu$};\draw[postaction={ decorate, decoration={ markings, mark=at position .5 with {{\arrow[scale=.8]{<}}} } }](0,-1)..controls(-1,0)..(0,1)node[below,text height=1.62ex,text depth=.5ex,inner ysep=2pt,at start]{$j$};\end{tikzpicture}\ =\ \begin{tikzpicture}[very thick,baseline,scale=.7,baseline=-4pt]\draw[postaction={ decorate, decoration={ markings, mark=at position .2 with {{\arrow[scale=.8]{<}}} } }](0,0)+(-1,-1)--+(1,1)node[below,text height=1.62ex,text depth=.5ex,inner ysep=2pt,at start]{$i$};\draw[preaction={draw, black, line width=2.2pt},draw=gray,line width=1.5pt](0,0)+(1,-1)--+(-1,1)node[below,text height=1.62ex,text depth=.5ex,inner ysep=2pt,at start]{$\mu$};\draw[postaction={ decorate, decoration={ markings, mark=at position .5 with {{\arrow[scale=.8]{<}}} } }](0,-1)..controls(1,0)..(0,1)node[below,text height=1.62ex,text depth=.5ex,inner ysep=2pt,at start]{$j$};\end{tikzpicture}\end{equation} together with their mirror images obtained by flipping each diagram left and right, and \begin{equation}\label{;A}\begin{tikzpicture}[very thick,baseline,scale=.7,baseline=-4pt]\draw[postaction={ decorate, decoration={ markings, mark=at position .8 with {{\arrow[scale=.8]{<}}} } }](0,0)+(-1,-1)--+(1,1)node[below,text height=1.62ex,text depth=.5ex,inner ysep=2pt,at start]{$i$};\draw[postaction={ decorate, decoration={ markings, mark=at position .2 with {{\arrow[scale=.8]{<}}} } }](0,0)+(1,-1)--+(-1,1)node[below,text height=1.62ex,text depth=.5ex,inner ysep=2pt,at start]{$j$};\draw[preaction={draw, black, line width=2.2pt},draw=gray,line width=1.5pt](0,-1)..controls(-1,0)..(0,1)node[below,text height=1.62ex,text depth=.5ex,inner ysep=2pt,at start]{$\mu$};\end{tikzpicture}\ -\ \begin{tikzpicture}[very thick,baseline,scale=.7,baseline=-4pt]\draw[postaction={ decorate, decoration={ markings, mark=at position .2 with {{\arrow[scale=.8]{<}}} } }](0,0)+(-1,-1)--+(1,1)node[below,text height=1.62ex,text depth=.5ex,inner ysep=2pt,at start]{$i$};\draw[postaction={ decorate, decoration={ markings, mark=at position .8 with {{\arrow[scale=.8]{<}}} } }](0,0)+(1,-1)--+(-1,1)node[below,text height=1.62ex,text depth=.5ex,inner ysep=2pt,at start]{$j$};\draw[preaction={draw, black, line width=2.2pt},draw=gray,line width=1.5pt](0,-1)..controls(1,0)..(0,1)node[below,text height=1.62ex,text depth=.5ex,inner ysep=2pt,at start]{$\mu$};\end{tikzpicture}\ =\ \delta_{ij}\sum_{a+b=\langle i,\mu\rangle-1}\begin{tikzpicture}[very thick,baseline,scale=.7,baseline=-4pt]\draw[postaction={ decorate, decoration={ markings, mark=at position .75 with {{\arrow[scale=.8]{<}}} } }][postaction={ decorate, decoration={ markings, mark=at position .5 with {\fill(0,0)circle(2.5pt);\node[left,xshift=-5pt,scale=.8]at(0,0){$b$};} } }](0,0)+(1,-1)--+(1,1)node[below,text height=1.62ex,text depth=.5ex,inner ysep=2pt,at start]{$j$};\draw[postaction={ decorate, decoration={ markings, mark=at position .75 with {{\arrow[scale=.8]{<}}} } }][postaction={ decorate, decoration={ markings, mark=at position .5 with {\fill(0,0)circle(2.5pt);\node[left,xshift=-5pt,scale=.8]at(0,0){$a$};} } }](0,0)+(-1,-1)--+(-1,1)node[below,text height=1.62ex,text depth=.5ex,inner ysep=2pt,at start]{$i$};\draw[preaction={draw, black, line width=2.2pt},draw=gray,line width=1.5pt](0,0)+(0,-1)--+(0,1)node[below,text height=1.62ex,text depth=.5ex,inner ysep=2pt,at start]{$\mu$};\end{tikzpicture}\ .\end{equation}\item Interaction with upward strands: Define upward crossings as follows. \begin{equation}\tikz[baseline,scale=1.5]{\draw[very thick,baseline,<-](.2,.3)--(-.2,-.1)node[at end,below,text height=1.62ex,text depth=.5ex,inner ysep=2pt,scale=.8]{$i$};\draw[preaction={draw, black, line width=2.2pt},draw=gray,line width=1.5pt](-.2,.3)--(.2,-.1)node[at end,below,text height=1.62ex,text depth=.5ex,inner ysep=2pt,scale=.8]{$\mu$};}\ \coloneqq\ {\tikzstyle{every picture}=[very thick,baseline,scale=.3]\begin{tikzpicture}\node[inner sep=0mm](0)at(-1,1){};\node[inner sep=0mm](1)at(1,1){};\node[inner sep=0mm](2)at(-1,-1){};\node[inner sep=0mm](3)at(1,-1){};\node[inner sep=0mm](4)at(-3,1){};\node[inner sep=0mm](5)at(3,-1){};\node[inner sep=0mm](6)at(-1,-2.5){};\node[inner sep=0mm](7)at(-3,-2.5){};\node[inner sep=0mm](8)at(1,2.5){};\node[inner sep=0mm](9)at(3,2.5){};\draw[postaction={ decorate, decoration={ markings, mark=at position .6 with {{\arrow[scale=.8]{>}}} } }](7.center)to(4.center);\draw[postaction={ decorate, decoration={ markings, mark=at position .4 with {{\arrow[scale=.8]{>}}} } }](5.center)to(9.center);\draw[preaction={draw, black, line width=2.2pt},draw=gray,line width=1.5pt](6.center)to(2.center);\draw[preaction={draw, black, line width=2.2pt},draw=gray,line width=1.5pt](1.center)to(8.center);\draw[bend left=90,looseness=1.50](4.center)to(0.center);\draw[bend right=90,looseness=1.50](3.center)to(5.center);\draw[in=90,out=-90,looseness=0.75](0.center)to(3.center);\draw[in=-90,out=90,looseness=0.75,preaction={draw, black, line width=2.2pt},draw=gray,line width=1.5pt](2.center)to(1.center);\end{tikzpicture}}\ ,\qquad\tikz[baseline,scale=1.5]{\draw[preaction={draw, black, line width=2.2pt},draw=gray,line width=1.5pt](.2,.3)--(-.2,-.1)node[at end,below,text height=1.62ex,text depth=.5ex,inner ysep=2pt,scale=.8]{$\mu$};\draw[very thick,baseline,<-](-.2,.3)--(.2,-.1)node[at end,below,text height=1.62ex,text depth=.5ex,inner ysep=2pt,scale=.8]{$i$};}\ \coloneqq\ {\tikzstyle{every picture}=[very thick,baseline,scale=.3,xscale=-1]\begin{tikzpicture}\node[inner sep=0mm](0)at(-1,1){};\node[inner sep=0mm](1)at(1,1){};\node[inner sep=0mm](2)at(-1,-1){};\node[inner sep=0mm](3)at(1,-1){};\node[inner sep=0mm](4)at(-3,1){};\node[inner sep=0mm](5)at(3,-1){};\node[inner sep=0mm](6)at(-1,-2.5){};\node[inner sep=0mm](7)at(-3,-2.5){};\node[inner sep=0mm](8)at(1,2.5){};\node[inner sep=0mm](9)at(3,2.5){};\draw[postaction={ decorate, decoration={ markings, mark=at position .6 with {{\arrow[scale=.8]{>}}} } }](7.center)to(4.center);\draw[postaction={ decorate, decoration={ markings, mark=at position .4 with {{\arrow[scale=.8]{>}}} } }](5.center)to(9.center);\draw[preaction={draw, black, line width=2.2pt},draw=gray,line width=1.5pt](6.center)to(2.center);\draw[preaction={draw, black, line width=2.2pt},draw=gray,line width=1.5pt](1.center)to(8.center);\draw[bend left=90,looseness=1.50](4.center)to(0.center);\draw[bend right=90,looseness=1.50](3.center)to(5.center);\draw[in=90,out=-90,looseness=0.75](0.center)to(3.center);\draw[in=-90,out=90,looseness=0.75,preaction={draw, black, line width=2.2pt},draw=gray,line width=1.5pt](2.center)to(1.center);\end{tikzpicture}}.\label{<A}\end{equation} Then we have \begin{equation}\label{=A}\begin{tikzpicture}[very thick,baseline,scale=.7,baseline=-3pt]\draw[postaction={ decorate, decoration={ markings, mark=at position .5 with {{\arrow[scale=.8]{>}}} } }](-2.8,-1)..controls(-1.2,0)..(-2.8,1)node[below,text height=1.62ex,text depth=.5ex,inner ysep=2pt,at start]{$i$};\draw[preaction={draw, black, line width=2.2pt},draw=gray,line width=1.5pt](-1.2,-1)..controls(-2.8,0)..(-1.2,1)node[below,text height=1.62ex,text depth=.5ex,inner ysep=2pt,at start]{$\mu$};\end{tikzpicture}\ =\ \ \begin{tikzpicture}[very thick,baseline,scale=.7,baseline=-3pt]\draw[postaction={ decorate, decoration={ markings, mark=at position .5 with {{\arrow[scale=.8]{>}}} } }](0.8,0)+(0,-1)--+(0,1)node[below,text height=1.62ex,text depth=.5ex,inner ysep=2pt,at start]{$i$};\draw[preaction={draw, black, line width=2.2pt},draw=gray,line width=1.5pt](1.6,0)+(0,-1)--+(0,1)node[below,text height=1.62ex,text depth=.5ex,inner ysep=2pt,at start]{$\mu$};\end{tikzpicture}\ ,\quad\begin{tikzpicture}[very thick,baseline,scale=.7,baseline=-3pt]\draw[postaction={ decorate, decoration={ markings, mark=at position .8 with {{\arrow[scale=.8]{>}}} } }](0,0)+(-1,-1)--+(1,1)node[below,text height=1.62ex,text depth=.5ex,inner ysep=2pt,at start]{$i$};\draw[postaction={ decorate, decoration={ markings, mark=at position .25 with {{\arrow[scale=.8]{>}}} } }](0,0)+(1,-1)--+(-1,1)node[below,text height=1.62ex,text depth=.5ex,inner ysep=2pt,at start]{$j$};\draw[preaction={draw, black, line width=2.2pt},draw=gray,line width=1.5pt](0,-1)..controls(-1,0)..(0,1)node[below,text height=1.62ex,text depth=.5ex,inner ysep=2pt,at start]{$\mu$};\end{tikzpicture}\ =\ \begin{tikzpicture}[very thick,baseline,scale=.7,baseline=-3pt]\draw[postaction={ decorate, decoration={ markings, mark=at position .25 with {{\arrow[scale=.8]{>}}} } }](0,0)+(-1,-1)--+(1,1)node[below,text height=1.62ex,text depth=.5ex,inner ysep=2pt,at start]{$i$};\draw[postaction={ decorate, decoration={ markings, mark=at position .8 with {{\arrow[scale=.8]{>}}} } }](0,0)+(1,-1)--+(-1,1)node[below,text height=1.62ex,text depth=.5ex,inner ysep=2pt,at start]{$j$};\draw[preaction={draw, black, line width=2.2pt},draw=gray,line width=1.5pt](0,-1)..controls(1,0)..(0,1)node[below,text height=1.62ex,text depth=.5ex,inner ysep=2pt,at start]{$\mu$};\end{tikzpicture}\end{equation}\end{enumerate} 2-morphisms involving $\mathfrak{I}_{-\mu}$ satisfy wavy counterparts of \eqref{:A}\ensuremath{^{\prime}}--\eqref{=A}\ensuremath{^{\prime}}, obtained by rotating the corresponding diagrams by $180^\circ$ and replacing thick strands labelled $\mu$ with wavy strands labelled $-\mu$.\end{definition}\par As with $\mathscr{U}$, we write $\mathscr{T}(c)$ to emphasize the dependence of $\mathscr{T}$ on the choice of bubble parameters $c$.\par\begin{definition}\thlabel{>A} \leavevmode \begin{enumerate}[(a)]\item The Chevalley involution $\psi\colon\mathscr{T}(c)\to\mathscr{T}(c^\ast)$ is the 2-functor defined on objects by $\lambda\mapsto-\lambda$; on generating 1-morphisms by $q\mapsto q$, $E_i\mapsto F_i$, $F_i\mapsto E_i$, and $\mathfrak{I}_{\pm\mu}\mapsto\mathfrak{I}_{\mp\mu}$; and on 2-morphisms by reversing the orientation of all strands, changing thick strands labelled $\mu$ to wavy strands labelled $-\mu$ and vice versa, and multiplying by $(-1)^{\#\text{crossings}}$.\label{@A}\item The bar-involution $\mathop{-}\nolimits\colon\mathscr{T}(c)\to\mathscr{T}(c)^{\operatorname{co}}$ is the 2-functor defined on objects by $\lambda\mapsto\lambda$; on generating 1-morphisms by $q\mapsto q^{-1}$, $E_i\mapsto E_i$, $F_i\mapsto F_i$, and $\mathfrak{I}_{\pm\mu}\mapsto\mathfrak{I}_{\pm\mu}$; and on 2-morphisms by reflecting the diagram across the horizontal axis and reversing the orientation of all thin strands.\label{^A}\end{enumerate}\end{definition}\par The relations allow one to perform Reidemeister moves with an arbitrary combination of strands at the cost of introducing terms with strictly smaller number of crossings, so $\mathscr{T}(\mu,\nu)$ is locally Laurentian. Note that no additional terms appear in any type III move involving a thick strand and two upward strands.\par\begin{lemma}Let $\lambda,\mu\in X$ and $\nu\in\mathbb{N}[I]$. There exists a natural isomorphism $\alpha\colon(\mathfrak{I}_{\pm\mu}\textnormal{\fontfamily{pbk}\selectfont1}_{\lambda\pm\nu})_\ast\circ\iota^{\pm,\lambda}_\nu\to(\mathfrak{I}_{\pm\mu}\textnormal{\fontfamily{pbk}\selectfont1}_{\lambda})^\ast\circ\iota^{\pm,\lambda\pm\mu}_\nu$ of functors $\mathcal{R}^\pm(\nu)\to\mathscr{T}(\lambda,\lambda\pm\mu\pm\nu)$.\end{lemma} \begin{proof}We only prove the case of $+$. Define the component of $\alpha$ at $i_1\otimes\cdots\otimes i_m\in\mathcal{R}^+(\nu)$ as the 2-morphism $\mathfrak{I}_{\mu}E_{i_1}\cdots E_{i_m}\textnormal{\fontfamily{pbk}\selectfont1}_{\lambda}\to E_{i_1}\cdots E_{i_m}\mathfrak{I}_{\mu}\textnormal{\fontfamily{pbk}\selectfont1}_\lambda$ represented by the following diagram: \begin{equation}\label{`A}\begin{tikzpicture}[very thick,baseline,scale=.6,baseline=-4pt]\draw[preaction={draw, black, line width=2.2pt},draw=gray,line width=1.5pt,in=-105,out=75,looseness=0.75](-2,-1)to node[below,text height=1.62ex,text depth=.5ex,inner ysep=2pt,at start]{$\mu$}(2,1);\draw[postaction={ decorate, decoration={ markings, mark=at position .5 with {{\arrow[scale=.8]{<}}} } }](-2,1)--(-1,-1)node[below,text height=1.62ex,text depth=.5ex,inner ysep=2pt,at end]{$i_1$};\draw[postaction={ decorate, decoration={ markings, mark=at position .5 with {{\arrow[scale=.8]{<}}} } }](-1,1)--(0,-1)node[below,text height=1.62ex,text depth=.5ex,inner ysep=2pt,at end]{$i_2$};\draw[postaction={ decorate, decoration={ markings, mark=at position .5 with {{\arrow[scale=.8]{<}}} } }](1,1)--(2,-1)node[below,text height=1.62ex,text depth=.5ex,inner ysep=2pt,at end]{$i_m$};\node at(.7,-.4){$\cdots$};\end{tikzpicture}\vspace{-.7em}\end{equation} Unpacking the coherence conditions for $\alpha$, we are reduced to checking that dots and crossings attached at the bottom right of the diagram \eqref{`A} can be slid past the thick strand. For dots, this follows from the first relation of \eqref{:A}. For crossings, this follows from the third relation of \eqref{:A} together with the cyclicity relation. The inverse of $\alpha$ is given by a diagram obtained by reflecting \eqref{`A} across the horizontal axis, by the first relation of \eqref{=A}.\end{proof} By the universal property of additive Karoubi envelopes (\S\ref{M}), $\alpha$ extends to a natural isomorphism $\dot{\alpha}\colon(\mathfrak{I}_{\pm\mu}\textnormal{\fontfamily{pbk}\selectfont1}_{\lambda\pm\nu})_\ast\circ\dot{\iota}^{\pm,\lambda}_\nu\to(\mathfrak{I}_{\pm\mu}\textnormal{\fontfamily{pbk}\selectfont1}_{\lambda})^\ast\circ\dot{\iota}^{\pm,\lambda\pm\mu}_\nu$ of functors $\dot{\mathcal{R}}^\pm(\nu)\to\dot{\mathscr{T}}(\lambda,\lambda\pm\mu\pm\nu)$. In particular, we have \begin{corollary}\thlabel{|A} For $u\in\dot{\mathcal{R}}^\pm(\nu)$, we have $\mathfrak{I}_{\pm\mu}u\textnormal{\fontfamily{pbk}\selectfont1}_\lambda\cong u\mathfrak{I}_{\pm\mu}\textnormal{\fontfamily{pbk}\selectfont1}_\lambda$ in $\dot{\mathscr{T}}(\lambda,\lambda\pm\mu\pm\nu)$.\end{corollary}\par\begin{remark}\thlabel{_B} The non-degeneracy of $\mathscr{T}$, i.e. that the Hom-spaces of $\mathscr{T}$ have the expected size, is due to Webster \cite[Thm.~5.11]{J}. Although loc.\ cit.\ imposes conditions on the base field, it can be applied to the algebraic closure of the fraction field of the ring of indeterminates (denoted $\mathbf{k}_C$ in \cite[\S3.2.3]{F}) to deduce the non-degeneracy over an arbitrary commutative ring $\Bbbk$. \end{remark}\subsection{Construction}\label{+A}\par By definition, there exists an obvious 2-functor $\mathscr{U}\to\mathscr{T}$. By 2-categorical Yoneda embedding, we have a 2-functor $\mathscr{T}(0,-)\colon\mathscr{T}\to\mathrm{Cat}_\Bbbk$ and thus a functor $\mathscr{T}(0,-)|_\mathscr{U}\colon\mathscr{U}\to\mathrm{Cat}_\Bbbk$.\par Recall ${\underline{{\lambda}}}=(\lambda_1,\ldots,\lambda_m)\in(X^+)^m$ and ${\underline{{\varepsilon}}}=(\varepsilon_1,\ldots,\varepsilon_m)\in\{-,+\}^m$ from \S\ref{uA}. For $\mu\in X$, let $\mathscr{T}(0,\mu)^{{\underline{{\varepsilon}}}\underline{{\lambda}}}$ be the full subcategory of $\mathscr{T}(0,\mu)$ consisting of 1-morphisms of the form \begin{equation}\label{AB}q^du_m\mathfrak{I}_{\varepsilon_m\lambda_m}u_{m-1}\mathfrak{I}_{\varepsilon_{m-1}\lambda_{m-1}}\cdots u_1\mathfrak{I}_{\varepsilon_1\lambda_1}u_0\end{equation} for some 1-morphisms $u_i$ in $\mathscr{U}$ for $0\le i\le m$. For a 1-morphism $u\in\mathscr{T}(\mu,\nu)$, the functor $\mathscr{T}(0,u)\colon\mathscr{T}(0,\mu)\to\mathscr{T}(0,\nu)$ is given by left multiplication by $u$, and hence if $u\in\mathscr{U}(\mu,\nu)$, it restricts to a functor $\mathscr{T}(0,u)^{{\underline{{\varepsilon}}}\underline{{\lambda}}}\colon\allowbreak\mathscr{T}(0,\mu)^{{\underline{{\varepsilon}}}\underline{{\lambda}}}\allowbreak\to\mathscr{T}(0,\nu)^{{\underline{{\varepsilon}}}\underline{{\lambda}}}$, and we have a 2-functor \[\mathscr{T}(0,-)^{{\underline{{\varepsilon}}}{\underline{{\lambda}}}}|_\mathscr{U}\colon\mathscr{U}\longrightarrow\mathrm{Cat}_\Bbbk.\] Let $I(\mu)$ be the two-sided ideal of a category $\mathscr{T}(0,\mu)^{{\underline{{\varepsilon}}}\underline{{\lambda}}}$ generated by identity morphisms of objects of the form \eqref{AB} with $u_0\ne q^d\textnormal{\fontfamily{pbk}\selectfont1}_0$.\footnote{Such a 2-morphism is called \emph{violating} in \cite{K}.} The ideal $I(\mu)$ is preserved by $\mathscr{T}(0,u)$ for any $u\in\mathscr{U}(\mu,\nu)$, and thus we have a 2-functor \[\mathcal{X}^{{\underline{{\varepsilon}}}{\underline{{\lambda}}}}\coloneqq\left.\frac{\mathscr{T}(0,-)}{I(-)}^{{\underline{{\varepsilon}}}{\underline{{\lambda}}}}\right|_\mathscr{U}\colon\mathscr{U}\longrightarrow\mathrm{Cat}_\Bbbk.\]\par Let ${\underline{{\nu}}}=(\nu_1,\ldots,\nu_m)\in\mathbb{N}[I]^m$. Define a graded $\Bbbk$-linear functor \[\tilde{\iota}^{{\underline{{\varepsilon}}}\underline{{\lambda}}}_{\underline{{\nu}}}\colon\mathcal{R}^{-\varepsilon_m}(\nu_m)\mathbin{\BeginAccSupp{method=hex,unicode,ActualText=22A0}\mathpalette{\m}{\otimes}\EndAccSupp{}}\cdots\mathbin{\BeginAccSupp{method=hex,unicode,ActualText=22A0}\mathpalette{\m}{\otimes}\EndAccSupp{}}\mathcal{R}^{-\varepsilon_1}(\nu_1)\longrightarrow\mathcal{X}^{{{\underline{{\varepsilon}}}\underline{{\lambda}}}}_\mu\] by concatenating diagrams horizontally with $\mathfrak{I}_{\varepsilon_j\lambda_j}$'s interleaved: that is, by $\tilde{\iota}^{{\underline{{\varepsilon}}}\underline{{\lambda}}}_{\underline{{\nu}}}\big((u_m,\ldots,u_1)\big)=\iota^{-\varepsilon_m,\xi_m}_{\nu_m}(u_m)\allowbreak\mathfrak{I}_{\varepsilon_m\lambda_m}\cdots\iota^{-\varepsilon_1,\xi_1}_{\nu_1}(u_1)\mathfrak{I}_{\varepsilon_1\lambda_1}$ on objects, and $\tilde{\iota}^{{\underline{{\varepsilon}}}\underline{{\lambda}}}_{\underline{{\nu}}}(\gamma_m\otimes\cdots\otimes\gamma_1)=\iota^{-\varepsilon_m,\xi_m}_{\nu_m}(\gamma_m)\mathbin{\boldsymbol{\cdot}}1_{\mathfrak{I}_{\varepsilon_m\lambda_m}}\mathbin{\boldsymbol{\cdot}}\cdots\mathbin{\boldsymbol{\cdot}}\iota^{-\varepsilon_1,\xi_1}_{\nu_1}(\gamma_1)\mathbin{\boldsymbol{\cdot}}1_{\mathfrak{I}_{\varepsilon_1\lambda_1}}$ for morphisms $\gamma_j\in\mathcal{R}^{-\varepsilon_j}(\nu_j)(u_j,u_j')$, where ${\underline{{\mu}}}\coloneqq{\underline{{\varepsilon}}}({\underline{{\lambda}}}-{\underline{{\nu}}})$, $\mu\coloneqq|{\underline{{\mu}}}|$, and $\xi_j\coloneqq\varepsilon_j\lambda_j+\sum_{k=1}^{j-1}\mu_k$.\par Let $\mathcal{X}^{{\underline{{\varepsilon}}}\underline{{\lambda}}}_{\le{\underline{{\mu}}}}$ (resp. $\mathcal{X}^{{\underline{{\varepsilon}}}\underline{{\lambda}}}_{<{\underline{{\mu}}}}$) be the quotient of $\mathcal{X}^{{\underline{{\varepsilon}}}\underline{{\lambda}}}_\mu$ by the two-sided ideal generated by identity morphisms of objects lying in the image of $\tilde{\iota}^{{\underline{{\varepsilon}}}\underline{{\lambda}}}_{{\underline{{\nu}}}'}$ for ${\underline{{\nu}}}'\in\mathbb{N}[I]^m$ with ${\underline{{\varepsilon}}}({\underline{{\lambda}}}-{\underline{{\nu}}}')\not\le{\underline{{\mu}}}$ (resp. $\not<{\underline{{\mu}}}$). Let $\mathcal{X}^{{\underline{{\varepsilon}}}\underline{{\lambda}}}_{\underline{{\mu}}}$ be the full subcategory of $\mathcal{X}^{{\underline{{\varepsilon}}}\underline{{\lambda}}}_{\le{\underline{{\mu}}}}$ consisting of objects lying in the image of $\tilde{\iota}^{{\underline{{\varepsilon}}}\underline{{\lambda}}}_{{\underline{{\nu}}}}$.\par\begin{lemma}The functor $\tilde{\iota}^{{\underline{{\varepsilon}}}\underline{{\lambda}}}_{\underline{{\nu}}}$ post-composed with the projection to $\mathcal{X}^{{\underline{{\varepsilon}}}\underline{{\lambda}}}_{\le{\underline{{\mu}}}}$ factors through \[\iota^{{\underline{{\varepsilon}}}\underline{{\lambda}}}_{\underline{{\mu}}}\colon\mathcal{R}^{\varepsilon_m\lambda_m}_{\mu_m}\mathbin{\BeginAccSupp{method=hex,unicode,ActualText=22A0}\mathpalette{\m}{\otimes}\EndAccSupp{}}\cdots\mathbin{\BeginAccSupp{method=hex,unicode,ActualText=22A0}\mathpalette{\m}{\otimes}\EndAccSupp{}}\mathcal{R}^{\varepsilon_1\lambda_1}_{\mu_1}\longrightarrow\mathcal{X}^{{\underline{{\varepsilon}}}\underline{{\lambda}}}_{\le{\underline{{\mu}}}}\,.\]\end{lemma} \begin{proof}For simplicity, let us omit $\iota$'s. Let $1\le k\le m$ and suppose $\varepsilon_k=+$. We show that $\gamma\coloneqq1_v\mathbin{\boldsymbol{\cdot}}x^{\langle i,\lambda_k\rangle}\mathbin{\boldsymbol{\cdot}}1_{\mathfrak{I}_{\lambda_k}u}\in\operatorname{End}_q(\allowbreak{}vF_i\mathfrak{I}_{\lambda_k}u)$ vanishes in $\mathcal{X}^{{{\underline{{\varepsilon}}}\underline{{\lambda}}}}_{\le{\underline{{\mu}}}}$ for any $u\in\mathcal{R}^{-\varepsilon_{k-1}}(\nu_{k-1})\mathbin{\BeginAccSupp{method=hex,unicode,ActualText=22A0}\mathpalette{\m}{\otimes}\EndAccSupp{}}\cdots\mathbin{\BeginAccSupp{method=hex,unicode,ActualText=22A0}\mathpalette{\m}{\otimes}\EndAccSupp{}}\mathcal{R}^{-\varepsilon_1}(\nu_1)$ and $v\in\mathcal{R}^{-\varepsilon_m}(\nu_m)\mathbin{\BeginAccSupp{method=hex,unicode,ActualText=22A0}\mathpalette{\m}{\otimes}\EndAccSupp{}}\cdots\mathbin{\BeginAccSupp{method=hex,unicode,ActualText=22A0}\mathpalette{\m}{\otimes}\EndAccSupp{}}\mathcal{R}^{-\varepsilon_k}(\nu_k-i)$. By the middle relation of \eqref{:A}, $\gamma$ factors through $v\mathfrak{I}_{\lambda_k}F_iu$. By the commutator relations of \thref{Q} and the first relation of \eqref{=A}\ensuremath{^{\prime}}, we may pull the strand $F_i$ in the identity morphism of $v\mathfrak{I}_{\lambda_k}F_iu$ to the left until it meets the region $k'<k$ with $\varepsilon_{k'}=+$ or reaches the leftmost area, at the cost of introducing additional 2-morphisms that factor through 1-morphisms with the $F_i$ removed. Thus, $1_{v\mathfrak{I}_{\lambda_k}F_iu}$ is a linear combination of morphisms that factor through an object in the image of $\tilde{\iota}^{{\underline{{\varepsilon}}}\underline{{\lambda}}}_{{\underline{{\nu}}}'}$ for some ${\underline{{\nu}}}'\in\mathbb{N}[I]^m$ with $\nu'_k=\nu_k-i$ and $\nu'_j=\nu_j$ for $j>k$. Since ${\underline{{\varepsilon}}}({\underline{{\lambda}}}-{\underline{{\nu}}}')\not\le{\underline{{\varepsilon}}}({\underline{{\lambda}}}-{\underline{{\nu}}})={\underline{{\mu}}}$, $\gamma$ vanishes in $\mathcal{X}^{{\underline{{\varepsilon}}}\underline{{\lambda}}}_{\le{\underline{{\mu}}}}$ as desired. The case $\varepsilon_k=-$ can be proved similarly, and the conclusion follows.\end{proof}\par Let $\mathcal{X}^{\mathbin{\BeginAccSupp{method=hex,unicode,ActualText=22A0}\mathpalette{\m}{\otimes}\EndAccSupp{}}{{\underline{{\varepsilon}}}\underline{{\lambda}}}}\colon\mathscr{U}(\mathfrak{g}^{\oplus m})\to\mathrm{Cat}_\Bbbk$ be the 2-functor such that $\mathcal{X}^{\mathbin{\BeginAccSupp{method=hex,unicode,ActualText=22A0}\mathpalette{\m}{\otimes}\EndAccSupp{}}{{\underline{{\varepsilon}}}\underline{{\lambda}}}}_{\underline{{\mu}}}\coloneqq\mathcal{X}^{\varepsilon_m\lambda_m}_{\mu_m}\mathbin{\BeginAccSupp{method=hex,unicode,ActualText=22A0}\mathpalette{\m}{\otimes}\EndAccSupp{}}\cdots\mathbin{\BeginAccSupp{method=hex,unicode,ActualText=22A0}\mathpalette{\m}{\otimes}\EndAccSupp{}}\mathcal{X}^{\varepsilon_1\lambda_1}_{\mu_1}$; generating 1-morphisms ${}_jE_i$ or ${}_jF_i$ act on the component $\mathcal{X}^{\varepsilon_j\lambda_j}_{\mu_j}$ by $E_i$ or $F_i$; generating 2-morphisms involving strands in a single component of $\coprod_{j=1}^mI$ act on the corresponding component; and the 2-morphism $\tau\colon{}_{j'}F_{i'}{}_{j}F_i\textnormal{\fontfamily{pbk}\selectfont1}_{\underline{{\mu}}}\to{}_{j}F_i{}_{j'}F_{i'}\textnormal{\fontfamily{pbk}\selectfont1}_{\underline{{\mu}}}$ for $j\ne j'$ acts by the identity. The canonical functors $\mathcal{R}^{\varepsilon_j\lambda_j}_{\mu_j}\to\mathcal{X}^{\varepsilon_j\lambda_j}_{\mu_j}$ are Morita equivalences \cite[Cor.~3.20]{K}. By transferring $\operatorname{mod}\!\operatorname{-}\hspace{-.8mm}\mathcal{X}^{\mathbin{\BeginAccSupp{method=hex,unicode,ActualText=22A0}\mathpalette{\m}{\otimes}\EndAccSupp{}}{{\underline{{\varepsilon}}}\underline{{\lambda}}}}$ along the equivalence of categories $\operatorname{mod}\!\operatorname{-}\hspace{-.8mm}\mathcal{X}^{\mathbin{\BeginAccSupp{method=hex,unicode,ActualText=22A0}\mathpalette{\m}{\otimes}\EndAccSupp{}}{{\underline{{\varepsilon}}}\underline{{\lambda}}}}_{\underline{{\mu}}}\xrightarrow{\sim}\operatorname{mod}\!\operatorname{-}\hspace{-.8mm}\mathcal{R}^{\mathbin{\BeginAccSupp{method=hex,unicode,ActualText=22A0}\mathpalette{\m}{\otimes}\EndAccSupp{}}{{\underline{{\varepsilon}}}\underline{{\lambda}}}}_{\underline{{\mu}}}$ where $\mathcal{R}^{\mathbin{\BeginAccSupp{method=hex,unicode,ActualText=22A0}\mathpalette{\m}{\otimes}\EndAccSupp{}}{{\underline{{\varepsilon}}}\underline{{\lambda}}}}_{\underline{{\mu}}}\coloneqq\mathcal{R}^{\varepsilon_m\lambda_m}_{\mu_m}\mathbin{\BeginAccSupp{method=hex,unicode,ActualText=22A0}\mathpalette{\m}{\otimes}\EndAccSupp{}}\cdots\mathbin{\BeginAccSupp{method=hex,unicode,ActualText=22A0}\mathpalette{\m}{\otimes}\EndAccSupp{}}\mathcal{R}^{\varepsilon_1\lambda_1}_{\mu_1}$, we obtain a pseudofunctor $\operatorname{mod}\!\operatorname{-}\hspace{-.8mm}\mathcal{R}^{\mathbin{\BeginAccSupp{method=hex,unicode,ActualText=22A0}\mathpalette{\m}{\otimes}\EndAccSupp{}}{{\underline{{\varepsilon}}}\underline{{\lambda}}}}\colon\mathscr{U}(\mathfrak{g}^{\oplus m})\to\mathrm{Ab}_\Bbbk$.\par\begin{remark}\thlabel{BB} Alternatively, one can construct $\operatorname{mod}\!\operatorname{-}\hspace{-.8mm}\mathcal{R}^{\mathbin{\BeginAccSupp{method=hex,unicode,ActualText=22A0}\mathpalette{\m}{\otimes}\EndAccSupp{}}{{\underline{{\varepsilon}}}\underline{{\lambda}}}}$ by applying \cite[\S5]{L} and \cite{H}, noting the equivalence of categories $\mathcal{R}^{\lambda_m}_{\mu_m}\mathbin{\BeginAccSupp{method=hex,unicode,ActualText=22A0}\mathpalette{\m}{\otimes}\EndAccSupp{}}\cdots\mathbin{\BeginAccSupp{method=hex,unicode,ActualText=22A0}\mathpalette{\m}{\otimes}\EndAccSupp{}}\mathcal{R}^{\lambda_1}_{\mu_1}\cong\mathcal{R}^{\underline{{\lambda}}}_{\underline{{\mu}}}$---see the proof of \thref{3A} for over a field, and note that $\mathcal{R}^{\underline{{\lambda}}}_{\underline{{\mu}}}$ is a projective $\Bbbk$-module by \cite[Thm.~4.5]{L}.\end{remark}\par\begin{theorem}\thlabel{.A} Let $\Bbbk$ be a commutative Henselian local ring. \begin{enumerate}[(i),font=\normalfont]\item The category $\mathcal{C}_\mu\coloneqq\operatorname{mod}\!\operatorname{-}\hspace{-.8mm}\mathcal{X}^{{\underline{{\varepsilon}}}\underline{{\lambda}}}_\mu$ has the structure of an upper finite fully stratified category with the poset $\Xi_\mu$, such that $\mathcal{C}_{\le{\underline{{\mu}}}}$ \textnormal{(resp. $\mathcal{C}_{<{\underline{{\mu}}}}$)} coincides with $\operatorname{mod}\!\operatorname{-}\hspace{-.8mm}\mathcal{X}^{{\underline{{\varepsilon}}}\underline{{\lambda}}}_{\le{\underline{{\mu}}}}$ \textnormal{(resp. $\operatorname{mod}\!\operatorname{-}\hspace{-.8mm}\mathcal{X}^{{\underline{{\varepsilon}}}\underline{{\lambda}}}_{<{\underline{{\mu}}}}$)} regarded as a full subcategory of $\operatorname{mod}\!\operatorname{-}\hspace{-.8mm}\mathcal{X}^{{\underline{{\varepsilon}}}\underline{{\lambda}}}_\mu$.\label{CB}\item $\iota^{{\underline{{\varepsilon}}}\underline{{\lambda}}}_{\underline{{\mu}}}$ is fully faithful, and hence induces an equivalence of categories \label{DB} $\mathcal{C}_{\underline{{\mu}}}\cong\operatorname{mod}\!\operatorname{-}\hspace{-.8mm}\allowbreak\mathcal{R}^{\varepsilon_m\lambda_m}_{\mu_m}\mathbin{\BeginAccSupp{method=hex,unicode,ActualText=22A0}\mathpalette{\m}{\otimes}\EndAccSupp{}}\cdots\mathbin{\BeginAccSupp{method=hex,unicode,ActualText=22A0}\mathpalette{\m}{\otimes}\EndAccSupp{}}\mathcal{R}^{\varepsilon_1\lambda_1}_{\mu_1}$.\item The $\mathscr{U}(\mathfrak{g})$-action $\operatorname{mod}\!\operatorname{-}\hspace{-.8mm}\mathcal{X}^{{\underline{{\varepsilon}}}\underline{{\lambda}}}$ and the $\mathscr{U}(\mathfrak{g}^{\oplus m})$-action $\operatorname{mod}\!\operatorname{-}\hspace{-.8mm}\mathcal{R}^{\mathbin{\BeginAccSupp{method=hex,unicode,ActualText=22A0}\mathpalette{\m}{\otimes}\EndAccSupp{}}{{\underline{{\varepsilon}}}\underline{{\lambda}}}}$ make $\{\operatorname{mod}\!\operatorname{-}\hspace{-.8mm}\mathcal{X}^{{\underline{{\varepsilon}}}\underline{{\lambda}}}_\mu\}_{\mu\in X}$ a generalized tensor product categorification of type ${{\underline{{\varepsilon}}}\underline{{\lambda}}}$. \label{EB}\end{enumerate}\end{theorem}\par\begin{remark}The construction of the stratification function in \ref{CB} follows from \cite[Lem.~5.9]{V}. Lemma 6.5 of \cite{V} outlines a modification of the proof of \cite[Prop.~5.5]{K} needed to establish the condition \ref{!A} of \thref{sA}. Parts \ref{CB} and \ref{DB} follow by extending \cite[Cor.~5.22 and Prop.~5.26]{K} to the case of arbitrary ${\underline{{\varepsilon}}}$. We will give a separate self-contained account of \thref{.A} in a paper in preparation \cite{r}, as some details are not given in \cite{V,K}, including the construction and the well-definedness of the isomorphism between $\Delta_{{\underline{{\mu}}}+{}_ji}({}_jE_iM)$ and a subquotient of $E_i\Delta_{\underline{{\mu}}}(M)$. The results in the rest of this paper depend on \thref{.A}.\end{remark}\par The bar-involution of $\mathscr{T}$ (\thref{>A} \ref{^A}) induces functors $\mathscr{T}(0,\mu)\to\mathscr{T}(0,\mu)^{\operatorname{op}}$, which restrict to $\mathscr{T}^{{{\underline{{\varepsilon}}}\underline{{\lambda}}}}(0,\mu)\to\mathscr{T}^{{{\underline{{\varepsilon}}}\underline{{\lambda}}}}(0,\mu)^{\operatorname{op}}$, and preserve the ideals $I(\mu)$. We have a diagram analogous to \eqref{bA}, and hence a duality of the $\mathscr{U}$-representation $\mathcal{X}^{{{\underline{{\varepsilon}}}\underline{{\lambda}}}}$. By \cite[Thm.~5.17]{V}, $\mathcal{X}^{{\underline{{\varepsilon}}}\underline{{\lambda}}}_\mu$ is a humorous category with respect to this duality (see also \thref{CA}). Hence, by \thref{aA}, we have a crystal $\mathbf{G}^{\textup{o}}(\dot{\mathcal{X}}^{{{\underline{{\varepsilon}}}\underline{{\lambda}}}})$. \thref{-A} and \thref{.A} give the following.\par\begin{corollary}\thlabel{FB} We have an isomorphism of crystals $G\colon B({{\underline{{\varepsilon}}}\underline{{\lambda}}})\to\mathbf{G}^{\textup{o}}(\dot{\mathcal{X}}^{{{\underline{{\varepsilon}}}\underline{{\lambda}}}})$.\end{corollary}\section{The bicrystal of $\mathscr{U}$}\label{G}\par\subsection{Radical of $\mathscr{U}(\lambda,\mu)$}Note that $\operatorname{End}_q(\textnormal{\fontfamily{pbk}\selectfont1}_\lambda)$ is a polynomial algebra generated by positive degree clockwise bubbles. For a 1-morphism $u$ of $\mathscr{U}(\lambda,\mu)$, $\gamma\in\operatorname{End}_q(\textnormal{\fontfamily{pbk}\selectfont1}_\lambda)$ induces a 2-morphism $\textnormal{\fontfamily{pbk}\selectfont1}_u\mathbin{\boldsymbol{\cdot}}\gamma\colon u\to u$ (resp.\ $\gamma\mathbin{\boldsymbol{\cdot}}\textnormal{\fontfamily{pbk}\selectfont1}_u\colon u\to u$) which corresponds to placing a bubble to the left (resp. right) of the identity diagram of $u$.\par\begin{lemma}\thlabel{GB} Let $\gamma\in\operatorname{Hom}(q^d\textnormal{\fontfamily{pbk}\selectfont1}_\lambda,\textnormal{\fontfamily{pbk}\selectfont1}_\lambda)$ with $d>0$. Then $1_u\mathbin{\boldsymbol{\cdot}}\gamma\in\big(\kern-.15em\operatorname{rad}\mathscr{U}(\lambda,\mu)\big)(q^du,u)$ for any $u\in\mathscr{U}(\lambda,\mu)$.\end{lemma} \begin{proof}Let $L$ be a simple right $\mathscr{U}(\lambda,\mu)$-module. $L(1_u\mathbin{\boldsymbol{\cdot}}\gamma)\colon L(u)\to L(q^du)$ for all $u$ forms a morphism $L\to q^dL$ of right $\mathscr{U}(\lambda,\mu)$-modules, and since $L(u)$ is concentrated in finitely many degrees, its kernel is a nontrivial submodule of $L$. Hence, it equals $L$, and $L(1_u\mathbin{\boldsymbol{\cdot}}\gamma)=0$. Since this holds for all simple $L$, it follows that $1_u\mathbin{\boldsymbol{\cdot}}\gamma\in\big(\kern-.15em\operatorname{rad}\mathscr{U}(\lambda,\mu)\big)(q^du,u)$ \cite[Prop.~9]{k}.\end{proof}\par For $N\in\mathbb{Z}_{\ge0}$, let $I^{\pm i,N}\subset\mathscr{U}(\lambda,\mu)$ be a two-sided ideal generated by morphisms $(\textnormal{\fontfamily{pbk}\selectfont1}_u\mathbin{\boldsymbol{\cdot}}x^N)\colon q_i^{2N}uE^\pm_i\to uE^\pm_i$ for all $u\in\mathscr{U}(\lambda\pm i,\mu)$. This agrees with $\mathcal{I}^{\pm i,N}$ in \S \ref{MA} associated to a 2-representation $\mathscr{U}(-,\mu)\circ\ast$.\par\begin{lemma}\thlabel{HB} Let $u\in\mathscr{U}(\lambda,\mu)$ and $i\in I$ be given. \begin{enumerate}[(a),font=\normalfont]\item For all sufficiently large $N$, we have $I^{\pm i,N}(v,u)\subset\big(\kern-.15em\operatorname{rad}\mathscr{U}(\lambda,\mu)\big)(v,u)$ for all $v\in\mathscr{U}(\lambda,\mu)$. \label{IB}\item If $\langle i,\lambda\rangle=0$ and $E^\pm_i$ does not appear in an expression $u=q^dE^{\pm}_{i_1}\cdots E^{\pm}_{i_k}\textnormal{\fontfamily{pbk}\selectfont1}_\lambda$, then $N$ in \textnormal{\ref{IB}} can be taken to be $0$. \label{JB}\end{enumerate}\end{lemma} \begin{proof}\ref{IB} follows from \thref{NA} applied to $\mathscr{U}(-,\mu)\circ\ast$. Let us assume the conditions in \ref{JB}. The proof of \thref{NA} applied to $\mathscr{U}(-,\mu)\circ\ast$ shows that the map \begin{equation}\label{KB}\psi\colon\mathscr{U}(\lambda-i,\mu)(q_i^{-1}uE_i,q_iuE_i)\longrightarrow I^{-,0}(u,u)\end{equation} that sends $f$ to the composition $(1_u\mathbin{\boldsymbol{\cdot}}\varepsilon')\circ(f\mathbin{\boldsymbol{\cdot}}1_{F_i})\circ(1_u\mathbin{\boldsymbol{\cdot}}\eta)$ is surjective.\par As an $\operatorname{End}_q(\textnormal{\fontfamily{pbk}\selectfont1}_{\lambda-i})$-module, the domain of $\psi$ is generated by 2-morphisms represented by a single tangle diagram without bubbles and such that every two strands intersect at most once \cite[\S3.2.3]{G}. Any such diagram must connect its unique top and bottom endpoints labeled $i$ by a strand that does not intersect other strands. Thus, as a $\Bbbk$-module, the domain of $\psi$ is generated by 2-morphisms of the form $g\mathbin{\boldsymbol{\cdot}}x^n\mathbin{\boldsymbol{\cdot}}\gamma$ for $g\in\operatorname{End}_q(u)$, $n\in\mathbb{Z}_{\ge0}$ and $\gamma\in\operatorname{End}_q(\textnormal{\fontfamily{pbk}\selectfont1}_{\lambda-i})$, where $x^n$ denotes the 2-morphism $x^n\colon q_i^{2n}E_i\textnormal{\fontfamily{pbk}\selectfont1}_{\lambda-i}\to E_i\textnormal{\fontfamily{pbk}\selectfont1}_{\lambda-i}$.\par We have $\psi(g\mathbin{\boldsymbol{\cdot}}x^n\mathbin{\boldsymbol{\cdot}}\gamma)=g\mathbin{\boldsymbol{\cdot}}\gamma'$ where $\gamma'\coloneqq\varepsilon'\circ(x^n\mathbin{\boldsymbol{\cdot}}\gamma\mathbin{\boldsymbol{\cdot}}1_{F_i\textnormal{\fontfamily{pbk}\selectfont1}_\lambda})\circ\eta\in\operatorname{End}_q(\textnormal{\fontfamily{pbk}\selectfont1}_\lambda)$, i.e., $\gamma'$ is obtained by enclosing $\gamma$ in a counter-clockwise bubble labeled $i$ with $n$ dots, whose degree is $(n+1)(i\cdot i)>0$ by the assumption $\langle i,\lambda\rangle=0$. \thref{GB} implies $g\mathbin{\boldsymbol{\cdot}}\gamma\in\big(\kern-.15em\operatorname{rad}\mathscr{U}(\lambda,\mu)\big)(u,u)$. Therefore, $I^{-,0}(u,u)=\operatorname{im}\psi(-\mathbin{\boldsymbol{\cdot}}1_{E_i})\subset\big(\kern-.15em\operatorname{rad}\mathscr{U}(\lambda,\mu)\big)(u,u)$, and we conclude $I^{-,0}(v,u)\subset\big(\kern-.15em\operatorname{rad}\mathscr{U}(\lambda,\mu)\big)(v,u)$ as in the proof of \thref{NA}. The case of $+$ sign can be proved similarly. \end{proof}\par For $\lambda_1,\lambda_2\in X^+$ with $-\lambda_1+\lambda_2=\lambda$, let $\mathcal{DR}^{-\lambda_1,\lambda_2}_\mu$ be the quotient of $\mathscr{U}(\lambda,\mu)$ by the two-sided ideal generated by positive degree bubbles placed on the leftmost region (i.e., $\textnormal{\fontfamily{pbk}\selectfont1}_u\mathbin{\boldsymbol{\cdot}}\gamma$ in \thref{GB}), plus $I^{+i,\langle i,\lambda_1\rangle}$ and $I^{-i,\langle i,\lambda_2\rangle}$ for all $i\in I$. Clearly, we have a 2-functor $\mathcal{DR}^{-\lambda_1,\lambda_2}\colon\mathscr{U}\to\mathrm{Cat}_\Bbbk$ and 2-natural transformations \begin{equation}\label{LB}\mathscr{U}(\lambda,-)\longrightarrow\mathcal{DR}^{-\xi',\lambda+\xi'}\longrightarrow\mathcal{DR}^{-\xi,\lambda+\xi}\end{equation} for $\xi,\xi'\in X^+$ with $\lambda+\xi,\xi'-\xi\in X^+$. We may identify $\operatorname{mod}\!\operatorname{-}\hspace{-.8mm}\mathcal{DR}^{-\xi,\lambda+\xi}_\mu$ as a full subcategory of $\operatorname{mod}\!\operatorname{-}\hspace{-.8mm}\mathscr{U}(\lambda,\mu)$ closed under taking subobjects and quotients. The duality of $\mathscr{U}(\lambda,-)$ (\thref{EA}) descends to $\mathcal{DR}^{-\xi,\lambda+\xi}$. By this we obtain crystals and their embeddings \[\mathbf{B}^{\textup{o}}(\operatorname{mod}\!\operatorname{-}\hspace{-.8mm}\mathcal{DR}^{-\xi,\lambda+\xi})\lhook\joinrel\longrightarrow\mathbf{B}^{\textup{o}}(\operatorname{mod}\!\operatorname{-}\hspace{-.8mm}\mathcal{DR}^{-\xi',\lambda+\xi'})\lhook\joinrel\longrightarrow\mathbf{B}^{\textup{o}}(\operatorname{mod}\!\operatorname{-}\hspace{-.8mm}\mathscr{U}(\lambda,-)).\]\par For the rest of the paper, we assume that there exists $\Lambda_i\in X$ for each $i\in I$ such that $\langle j,\Lambda_i\rangle=\delta_{ji}$ for all $j\in I$. Let us regard $X^+$ as a poset by $\xi\le\xi'$ if and only if $\xi'-\xi\in X^+$.\par\begin{corollary}\thlabel{MB} The canonical embedding below is an isomorphism of crystals. \begin{equation}\label{NB}\varinjlim_{\xi\in X^+}\mathbf{B}^{\textup{o}}(\operatorname{mod}\!\operatorname{-}\hspace{-.8mm}\mathcal{DR}^{-\xi,\lambda+\xi})\lhook\joinrel\longrightarrow\mathbf{B}^{\textup{o}}(\operatorname{mod}\!\operatorname{-}\hspace{-.8mm}\mathscr{U}(\lambda,-))\end{equation}\end{corollary} \begin{proof}Let $L\in\mathbf{B}^{\textup{o}}(\operatorname{mod}\!\operatorname{-}\hspace{-.8mm}\mathscr{U}(\lambda,\mu))$ and $P\in\mathscr{U}(\lambda,\mu)$ be such that $L(P)\ne0$. We show that $\big(\kern-.15em\operatorname{rad}\mathscr{U}(\lambda,\mu)\big)(-,P)$ contains the kernel of $\mathscr{U}(\lambda,\mu)(-,P)\to\mathcal{DR}^{-\xi,\lambda+\xi}_\mu(-,P)$ for $\xi\gg0$. By \thref{GB}, positive degree bubbles on the leftmost region lie in $\big(\kern-.15em\operatorname{rad}\mathscr{U}(\lambda,\mu)\big)(-,P)$. By \thref{HB}, there exist $N_i^+,N_i^-\in\mathbb{Z}_{\ge0}$ for each $i\in I$, all but finitely many of which are zero, such that $I^{\pm i,N_i^{\pm}}(-,P)\subset\big(\kern-.15em\operatorname{rad}\mathscr{U}(\lambda,\mu)\big)(-,P)$. We have $\langle i,\xi\rangle\ge N_i^+$ and $\langle i,\lambda+\xi\rangle\ge N_i^-$ for all $i\in I$ for $\xi\gg0$, and hence the claim. This shows that the embedding \eqref{NB} is surjective, and thus an isomorphism of crystals.\end{proof}\subsection{Stratification of $\mathscr{U}(\lambda,\mu)$}\label{OB}\par In the rest of the paper, $\Bbbk$ is a commutative Henselian local ring.\par For $\lambda\in X^+$ and $\mu\in X$, we choose $G(b)\in\mathbf{G}^{\textup{o}}(\dot{\mathcal{X}}^\lambda_\mu)$ for $b\in B(\lambda)_\mu$ (\thref{FB}) as an object of $\dot{\mathcal{R}}^\lambda_\mu$ via the equivalence $\dot{\iota}^\lambda_\mu\colon\dot{\mathcal{R}}^\lambda_\mu\xrightarrow{\sim}\dot{\mathcal{X}}^\lambda_\mu$. Let $\pi_\lambda\colon\mathcal{R}^-(\nu)\twoheadrightarrow\mathcal{R}^\lambda_{\lambda-\nu}$ be the canonical projection functor. The directed system $\{\mathbf{B}^{\textup{o}}(\operatorname{mod}\!\operatorname{-}\hspace{-.8mm}\mathcal{R}^\lambda_{\lambda-\nu})\}_{\lambda\in X^+}$ induced by canonical projection functors is isomorphic to $\{B(\lambda)_{\lambda-\nu}\}_{\lambda\in X^+}$ with transition maps $B(\lambda)\to B(\lambda')$ that commute with crystal operators $\widetilde{e}_i$'s for $\lambda<\lambda'$ in $X^+$ (see the proof of \thref{rA}). Since $\varinjlim_{\lambda\in X^+}B(\lambda)_{\lambda-\nu}\cong B(\infty)_{-\nu}$ and $\varinjlim_{\lambda\in X^+}\mathbf{B}^{\textup{o}}(\operatorname{mod}\!\operatorname{-}\hspace{-.8mm}\mathcal{R}^\lambda_{\lambda-\nu})\cong\mathbf{B}^{\textup{o}}(\operatorname{mod}\!\operatorname{-}\hspace{-.8mm}\mathcal{R}^-(\nu))$ (e.g.\ \cite[Prop.~2.4]{E}), we obtain a bijection $G\colon B(\infty)_{-\nu}\to\mathbf{G}^{\textup{o}}(\dot{\mathcal{R}}^-(\nu))$ satisfying \begin{equation}\label{PB}\dot{\pi}_\lambda(G(b))\cong G(\bar{\pi}_\lambda(b))\quad\text{for all }b\in B(\infty)\text{ with }\bar{\pi}_\lambda(b)\ne0,\end{equation} where $\dot{\pi}_\lambda$ is the functor induced from $\pi_\lambda$ by taking the additive Karoubi envelope, and $\bar{\pi}_\lambda\colon B(\infty)\to B(\lambda)$ is the canonical map. Similarly, we obtain a bijection $G\colon B(-\infty)_{\nu}\to\mathbf{G}^{\textup{o}}(\dot{\mathcal{R}}^+(\nu))$.\par Let $\lambda,\mu\in X$ with $\lambda-\mu\in\mathbb{Z}[I]$ be given, as $\mathscr{U}(\lambda,\mu)$ is otherwise empty. Let $\nu\in\mathbb{N}[I]$ be such that $\lambda-\mu+\nu\in\mathbb{N}[I]$.\par\begin{lemma}\thlabel{QB} Let $b_1\in B(-\infty)_\nu$ and $b_2\in B(\infty)_{\mu-\lambda-\nu}$. Let $\xi\in X^+$ be such that $\lambda+\xi\in X^+$, $\bar{\pi}_{-\xi}(b_1)\ne0$, and $\bar{\pi}_{\lambda+\xi}(b_2)\ne0$. Then we have an isomorphism \begin{equation}\label{RB}G(\bar{\pi}_{-\xi}(b_1)\otimes\bar{\pi}_{\lambda+\xi}(b_2))\cong G(b_2)G(b_1)\mathfrak{I}_{\lambda+\xi}\mathfrak{I}_{-\xi}\quad\text{ in }\dot{\mathcal{X}}^{-\xi,\lambda+\xi}_{\le(\nu-\xi,\:\mu+\xi-\nu)}.\end{equation}\end{lemma} \begin{proof}{} Under the equivalence $\iota^{(-\xi,\lambda+\xi)}_{(\nu-\xi,\:\mu+\xi-\nu)}\colon\mathcal{R}^{-\xi}_{\nu-\xi}\mathbin{\BeginAccSupp{method=hex,unicode,ActualText=22A0}\mathpalette{\m}{\otimes}\EndAccSupp{}}\mathcal{R}^{\lambda+\xi}_{\mu+\xi-\nu}\to\mathcal{X}^{-\xi,\lambda+\xi}_{(\nu-\xi,\:\mu+\xi-\nu)}$, the simple object $\bar{L}(\bar{\pi}_{-\xi}(b_1)\otimes\bar{\pi}_{\lambda+\xi}(b_2))\in\operatorname{mod}\!\operatorname{-}\hspace{-.8mm}\mathcal{X}^{-\xi,\lambda+\xi}_{(\nu-\xi,\:\mu+\xi-\nu)}$ corresponds to $L(\bar{\pi}_{-\xi}(b_1))\mathbin{\BeginAccSupp{method=hex,unicode,ActualText=22A0}\mathpalette{\m}{\otimes}\EndAccSupp{}}L(\bar{\pi}_{\lambda+\xi}(b_2))\in\operatorname{mod}\!\operatorname{-}\hspace{-.8mm}\mathcal{R}^{-\xi}_{\nu-\xi}\mathbin{\BeginAccSupp{method=hex,unicode,ActualText=22A0}\mathpalette{\m}{\otimes}\EndAccSupp{}}\mathcal{R}^{\lambda+\xi}_{\mu+\xi-\nu}$ (cf.\ \thref{BB}). Since $L(\bar{\pi}_{-\xi}(b_1))$ and $L(\bar{\pi}_{\lambda+\xi}(b_2))$ are absolutely simple, $\bar{P}(\bar{\pi}_{-\xi}(b_1)\otimes\bar{\pi}_{\lambda+\xi}(b_2))\in\operatorname{mod}\!\operatorname{-}\hspace{-.8mm}\mathcal{X}^{-\xi,\lambda+\xi}_{(\nu-\xi,\:\mu+\xi-\nu)}$ is represented by the object \begin{equation}\begin{split}G(b_2)\mathfrak{I}_{\lambda+\xi}G(b_1)\mathfrak{I}_{-\xi}&\cong\dot{\iota}^{(-\xi,\lambda+\xi)}_{(\nu-\xi,\:\mu+\xi-\nu)}\big(\dot{\pi}_{-\xi}(G(b_1))\mathbin{\BeginAccSupp{method=hex,unicode,ActualText=22A0}\mathpalette{\m}{\otimes}\EndAccSupp{}}\dot{\pi}_{\lambda+\xi}(G(b_2))\big)\\&\cong\dot{\iota}^{(-\xi,\lambda+\xi)}_{(\nu-\xi,\:\mu+\xi-\nu)}\big(G(\bar{\pi}_{-\xi}(b_1))\mathbin{\BeginAccSupp{method=hex,unicode,ActualText=22A0}\mathpalette{\m}{\otimes}\EndAccSupp{}}G(\bar{\pi}_{\lambda+\xi}(b_2))\big)\end{split}\end{equation} in $\dot{\mathcal{X}}^{-\xi,\lambda+\xi}_{(\nu-\xi,\:\mu+\xi-\nu)}$, where the second isomorphism uses \eqref{PB}. It follows that \begin{equation}\begin{split}\Delta(\bar{\pi}_{-\xi}(b_1)\otimes\bar{\pi}_{\lambda+\xi}(b_2))&\cong\bar{P}(\bar{\pi}_{-\xi}(b_1)\otimes\bar{\pi}_{\lambda+\xi}(b_2))\otimes_{\mathcal{X}^{-\xi,\lambda+\xi}_{(\nu-\xi,\:\mu+\xi-\nu)}}\dot{\mathcal{X}}^{-\xi,\lambda+\xi}_{\le(\nu-\xi,\:\mu+\xi-\nu)}\\&\cong\dot{\mathcal{X}}^{-\xi,\lambda+\xi}_{\le(\nu-\xi,\:\mu+\xi-\nu)}(-,G(b_2)\mathfrak{I}_{\lambda+\xi}G(b_1)\mathfrak{I}_{-\xi}).\end{split}\end{equation}\par Therefore, the projective $\mathcal{X}^{-\xi,\lambda+\xi}_\mu$-module represented by $G(b_2)\mathfrak{I}_{\lambda+\xi}G(b_1)\mathfrak{I}_{-\xi}$ admits a surjection onto $\Delta(\bar{\pi}_{-\xi}(b_1)\otimes\bar{\pi}_{\lambda+\xi}(b_2))$, which in turn surjects onto $L(\bar{\pi}_{-\xi}(b_1)\otimes\bar{\pi}_{\lambda+\xi}(b_2))$. Hence, $G(\bar{\pi}_{-\xi}(b_1)\otimes\bar{\pi}_{\lambda+\xi}(b_2))$ is a direct summand of $G(b_2)\mathfrak{I}_{\lambda+\xi}G(b_1)\mathfrak{I}_{-\xi}$ in $\dot{\mathcal{X}}^{-\xi,\lambda+\xi}_\mu$. Now, by \cite[Lem.~5.10]{V}, $G(b_2)\mathfrak{I}_{\lambda+\xi}G(b_1)\mathfrak{I}_{-\xi}$ is indecomposable in $\dot{\mathcal{X}}^{-\xi,\lambda+\xi}_{\le(\nu-\xi,\:\mu+\xi-\nu)}$, so its summands other than $G(\bar{\pi}_{-\xi}(b_1)\otimes\bar{\pi}_{\lambda+\xi}(b_2))$ vanish. The isomorphism \eqref{RB} follows from this and $G(b_2)\mathfrak{I}_{\lambda+\xi}G(b_1)\mathfrak{I}_{-\xi}\cong G(b_2)G(b_1)\mathfrak{I}_{\lambda+\xi}\mathfrak{I}_{-\xi}$, which holds by \thref{|A}.\end{proof}\par Let $\varepsilon\in\{-,+\}$ and $\nu_1,\nu_2\in\mathbb{N}[I]$ be such that $\lambda-\mu=\varepsilon(\nu_2-\nu_1)$. Let \[\iota^{\varepsilon,\lambda}_{\nu_1,\nu_2}\coloneqq c_{\lambda,\lambda+\nu_1,\mu}\circ(\iota^{\varepsilon,\lambda}_{\nu_1}\mathbin{\BeginAccSupp{method=hex,unicode,ActualText=22A0}\mathpalette{\m}{\otimes}\EndAccSupp{}}\iota^{-\varepsilon,\mu+\varepsilon\nu_2}_{\nu_2})\colon\mathcal{R}^\varepsilon(\nu_1)\mathbin{\BeginAccSupp{method=hex,unicode,ActualText=22A0}\mathpalette{\m}{\otimes}\EndAccSupp{}}\mathcal{R}^{-\varepsilon}(\nu_2)\longrightarrow\mathscr{U}(\lambda,\mu)\] be a graded $\Bbbk$-linear functor (so $\iota^{\varepsilon,\lambda}_{\nu_1,\nu_2}(u_1\mathbin{\BeginAccSupp{method=hex,unicode,ActualText=22A0}\mathpalette{\m}{\otimes}\EndAccSupp{}}u_2)=u_2u_1\textnormal{\fontfamily{pbk}\selectfont1}_\lambda$).\par Let $\mathscr{U}(\lambda,\mu)_{\le\nu}$ be the quotient of $\mathscr{U}(\lambda,\mu)$ by the two-sided ideal generated by identity morphisms of objects lying in the image of $\iota^{+,\lambda}_{\nu',\lambda-\mu+\nu'}$ for $\nu'\in\mathbb{N}[I]$ with $\lambda-\mu+\nu'\in\mathbb{N}[I]$ and $\nu'-\nu\not\in\mathbb{N}[I]$.\par Consider the functor $\mathscr{U}(\lambda,\mu)\to\mathcal{X}^{-\xi,\lambda+\xi}_\mu$ given by $u\mapsto u\mathfrak{I}_{\lambda+\xi}\mathfrak{I}_{-\xi}$. By \cite[Prop.~5.6]{V}, it factors through $\mathcal{DR}^{-\xi,\lambda+\xi}_\mu$, and the induced functor $\mathcal{DR}^{-\xi,\lambda+\xi}_\mu\to\mathcal{X}^{-\xi,\lambda+\xi}_\mu$ induces an equivalence of categories between their additive Karoubi envelopes. On the other hand, the image of the defining ideal of $\mathscr{U}(\lambda,\mu)_{\le\nu}$ is precisely the defining ideal of $\mathcal{X}^{-\xi,\lambda+\xi}_{\le(\nu-\xi,\:\mu+\xi-\nu)}$ (cf.\ \S\ref{+A}), since we have $u_2u_1\mathfrak{I}_{\lambda+\xi}\mathfrak{I}_{-\xi}\cong u_2\mathfrak{I}_{\lambda+\xi}u_1\mathfrak{I}_{-\xi}$ for $u_1\in\mathcal{R}^+(\nu')$ and $u_2\in\mathcal{R}^-(\lambda-\mu+\nu')$. By letting $(\mathcal{DR}^{-\xi,\lambda+\xi}_\mu)_{\le\nu}$ be the quotient of $\mathcal{DR}^{-\xi,\lambda+\xi}_\mu$ by the image of this two-sided ideal, we have a commutative diagram \begin{equation}\label{SB}\begin{tikzcd}\mathcal{DR}^{-\xi,\lambda+\xi}_\mu\arrow[r,hook,"-\cdot\x_{\lambda+\xi}\x_{-\xi}"]\arrow[d,two heads]&\mathcal{X}^{-\xi,\lambda+\xi}_\mu\arrow[d,two heads]\\(\mathcal{DR}^{-\xi,\lambda+\xi}_\mu)_{\le\nu}\arrow[r,hook]&\mathcal{X}^{-\xi,\lambda+\xi}_{\le(\nu-\xi,\:\mu+\xi-\nu)}\end{tikzcd}\end{equation} such that the two horizontal arrows become equivalences of categories after taking additive Karoubi envelopes.\par Let $T_\lambda=\{t_\lambda\}$ be the crystal with $\operatorname{wt}(t_\lambda)=\lambda$ and $\varepsilon_i(t_\lambda)=-\infty$ for all $i\in I$ (cf.\ \cite[Ex.~1.5.3]{O}).\par\begin{proposition}\thlabel{DA} There exists an isomorphism of crystals \[G\colon B(-\infty)\otimes T_\lambda\otimes B(\infty)\to\bigsqcup_{\mu\in X}\mathbf{G}^{\textup{o}}(\dot{\mathscr{U}}(\lambda,\mu))\vspace{-.2em}\] such that for all $b_1\in B(-\infty)_\nu$ and $b_2\in B(\infty)_{\mu-\lambda-\nu}$, we have an isomorphism \begin{equation}\label{TB}G(b_1\otimes t_\lambda\otimes b_2)\cong G(b_2)G(b_1)\textnormal{\fontfamily{pbk}\selectfont1}_\lambda\quad\text{ in }\dot{\w}(\lambda,\mu)_{\le\nu}.\end{equation} \end{proposition} \begin{proof}Let $G(b_2)G(b_1)\textnormal{\fontfamily{pbk}\selectfont1}_\lambda\cong\bigoplus_{k=1}^nP_k$ be a decomposition into indecomposable objects in $\dot{\mathscr{U}}(\lambda,\mu)$, and let $L_k\coloneqq\operatorname{hd}\dot{\mathscr{U}}(\lambda,\mu)(-,P_k)\in\mathbf{B}(\operatorname{mod}\!\operatorname{-}\hspace{-.8mm}\mathscr{U}(\lambda,\mu))$. By \thref{MB}, there exists $\xi\in X^+$ such that every $L_k$ factors through $\mathcal{DR}^{-\xi,\lambda+\xi}_\mu$, $\bar{\pi}_{-\xi}(b_1)\ne0$, and $\bar{\pi}_{\lambda+\xi}(b_2)\ne0$. By \thref{QB}, the image of $G(b_2)G(b_1)\textnormal{\fontfamily{pbk}\selectfont1}_\lambda$ in $(\mathcal{DR}^{-\xi,\lambda+\xi}_\mu)_{\le\nu}$ is isomorphic to $G(\bar{\pi}_{-\xi}(b_1)\otimes\bar{\pi}_{\lambda+\xi}(b_2))\in\dot{\mathcal{X}}^{-\xi,\lambda+\xi}_{\le(\nu-\xi,\:\mu+\xi-\nu)}$, which is indecomposable. This implies that there exists a unique $k$ such that $L_k$ factors through $(\mathcal{DR}^{-\xi,\lambda+\xi}_\mu)_{\le\nu}$ (cf.\ \eqref{SB}), or equivalently, factors through $\mathscr{U}(\lambda,\mu)_{\le\nu}$. Hence, we have $P_l\cong0$ in $\dot{\mathscr{U}}(\lambda,\mu)_{\le\nu}$ for $l\ne k$, and $G(b_1\otimes t_\lambda\otimes b_2)$ satisfying \eqref{TB} is the unique indecomposable summand of $G(b_2)G(b_1)\textnormal{\fontfamily{pbk}\selectfont1}_\lambda$ that does not vanish in $\dot{\w}(\lambda,\mu)_{\le\nu}$.\par Let $\xi\in X^+$ be any element with $\lambda+\xi\in X^+$. Let $b_1$ and $b_2$ be such that $\bar{\pi}_{-\xi}(b_1)\ne0$ and $\bar{\pi}_{\lambda+\xi}(b_2)\ne0$. The image of $G(b_1\otimes t_\lambda\otimes b_2)$ in $\dot{\mathcal{X}}^{-\xi,\lambda+\xi}_{\le(\nu-\xi,\:\mu+\xi-\nu)}$ is isomorphic to $G(\bar{\pi}_{-\xi}(b_1)\otimes\bar{\pi}_{\lambda+\xi}(b_2))$, and both are indecomposable in $\dot{\mathcal{X}}^{-\xi,\lambda+\xi}_\mu$, so they are isomorphic in $\dot{\mathcal{X}}^{-\xi,\lambda+\xi}_\mu$. Thus, we have a commutative diagram \[\begin{tikzcd}B(-\infty)\otimes T_\lambda\otimes B(\infty)\arrow[r,"G"]&\bigsqcup_{\mu\in X}\mathbf{G}^{\textup{o}}(\dot{\mathscr{U}}(\lambda,\mu))\\B(-\xi)\otimes B(\lambda+\xi)\arrow[r,"G"]\arrow[u,hook]&\bigsqcup_{\mu\in X}\mathbf{G}^{\textup{o}}(\dot{\mathcal{X}}^{-\xi,\lambda+\xi}_\mu)\arrow[u,hook]\end{tikzcd}\] of sets, where the left (resp.\ right) vertical arrows are the embeddings of crystals given by $\bar{\pi}_{-\xi}(b_1)\otimes\bar{\pi}_{\lambda+\xi}(b_2)\mapsto b_1\otimes t_\lambda\otimes b_2$ (resp.\ induced by $\mathscr{U}(\lambda,-)\twoheadrightarrow\mathcal{DR}^{-\xi,\lambda+\xi}\hookrightarrow\mathcal{X}^{-\xi,\lambda+\xi}$). Since the bottom arrow is an isomorphism of crystals and the top arrow is their direct limit over $\xi\in X^+$, the same holds for the top arrow.\end{proof}\par\begin{remark}The isomorphisms \eqref{RB} and \eqref{TB} can be viewed as analogues of Lemma 2.2.3 and equation (3.1.1) of \cite{O}, respectively.\end{remark}\par\subsection{The star crystal}Let $\mathscr{U}(\lambda,\mu)_{\not\ge\nu}$ be the quotient of $\mathscr{U}(\lambda,\mu)$ by the two-sided ideal generated by identity morphisms of objects lying in the image of $\iota^{+,\lambda}_{\nu',\lambda-\mu+\nu'}$ for $\nu'\in\mathbb{N}[I]$ with $\lambda-\mu+\nu'\in\mathbb{N}[I]$ and $\nu-\nu'\in\mathbb{N}[I]\setminus\{0\}$. Let $\pi_{\not\ge\nu}\colon\mathscr{U}(\lambda,\mu)\to\mathscr{U}(\lambda,\mu)_{\not\ge\nu}$ be the canonical projection.\par\begin{lemma}\thlabel{UB} Let $u=E_{i_l}^{\epsilon_l}\cdots E_{i_1}^{\epsilon_1}\textnormal{\fontfamily{pbk}\selectfont1}_\lambda\in\mathscr{U}(\lambda,\mu)$ $(\epsilon_k\in\{-,+\})$ and $\nu=\sum_{k,\epsilon_k=+}i_k$. \begin{enumerate}[(i),font=\normalfont]\item $\pi_{\not\ge\nu}(u)\cong\pi_{\not\ge\nu}(E_{i_l}^{\epsilon_l}\cdots E_{i_{k}}^{\epsilon_{k}}E_{i_{k+1}}^{\epsilon_{k+1}}\cdots E_{i_1}^{\epsilon_1}\textnormal{\fontfamily{pbk}\selectfont1}_\lambda)$ for all $1\le k<l$ with $\epsilon_k\ne\epsilon_{k+1}$. \label{VB}\item $\pi_{\not\ge\nu'}(u)\cong0$ for all $\nu'\in\mathbb{N}[I]$ with $\nu'-\nu\in\mathbb{N}[I]\setminus\{0\}$. \label{WB}\end{enumerate}\end{lemma} \begin{proof}The statement \ref{VB} implies \ref{WB} since one can move all $E^+_i$'s to the right and all $E^-_i$'s to the left. For \ref{VB}, the sideways crossings \eqref{i} provide a map $u\to E_{i_l}^{\epsilon_l}\cdots E_{i_{k}}^{\epsilon_{k}}E_{i_{k+1}}^{\epsilon_{k+1}}\cdots E_{i_1}^{\epsilon_1}\allowbreak\textnormal{\fontfamily{pbk}\selectfont1}_\lambda$. By \eqref{j}--\eqref{l}, the composition of two sideways crossings at the $k$-th position equals $\pm1_u$ modulo the morphisms factoring through objects with one fewer $E^+_i$ than $u$, and applying \ref{WB} to $(\nu,\nu')=(\nu-i,\nu)$ shows that such morphisms vanish in $\mathscr{U}(\lambda,\mu)_{\not\ge\nu}$. Therefore, an induction on $|\nu|$ completes the proof.\end{proof}\par\begin{lemma}There exists a natural isomorphism $\alpha\colon\pi_{\not\ge\nu}\circ\iota^{+,\lambda}_{\nu,\lambda-\mu+\nu}\to\pi_{\not\ge\nu}\circ\iota^{-,\lambda}_{\lambda-\mu+\nu,\nu}\circ\sigma$ of functors $\mathcal{R}^+(\nu)\mathbin{\BeginAccSupp{method=hex,unicode,ActualText=22A0}\mathpalette{\m}{\otimes}\EndAccSupp{}}\mathcal{R}^-(\lambda-\mu+\nu)\to\mathscr{U}(\lambda,\mu)_{\not\ge\nu}$, where $\sigma(x\mathbin{\BeginAccSupp{method=hex,unicode,ActualText=22A0}\mathpalette{\m}{\otimes}\EndAccSupp{}}y)=y\mathbin{\BeginAccSupp{method=hex,unicode,ActualText=22A0}\mathpalette{\m}{\otimes}\EndAccSupp{}}x$.\end{lemma} \begin{proof}Define the component of $\alpha$ at $u_1\mathbin{\BeginAccSupp{method=hex,unicode,ActualText=22A0}\mathpalette{\m}{\otimes}\EndAccSupp{}}u_2\in\mathcal{R}^+(\nu)\mathbin{\BeginAccSupp{method=hex,unicode,ActualText=22A0}\mathpalette{\m}{\otimes}\EndAccSupp{}}\mathcal{R}^-(\lambda-\mu+\nu)$ as the tangle diagram that represents a reduced expression of the permutation $\omega\in S_{|\nu|+|\lambda-\mu+\nu|}$ defined by $\omega(k)=k+|\lambda-\mu+\nu|$ for $1\le k\le|\nu|$ and $\omega(k)=k-|\nu|$ for $|\nu|<k$. \begin{equation}\label{XB}\alpha_{u_1\mathbin{\BeginAccSupp{method=hex,unicode,ActualText=22A0}\mathpalette{\m}{\otimes}\EndAccSupp{}}u_2}\coloneqq{\tikzstyle{every picture}=[baseline=-2pt,thick,scale=.4]\begin{tikzpicture}\node(0)at(0,-2.5){};\node(1)at(1,-2.5){};\node(2)at(2,-2.5){};\node(3)at(4,-2.5){};\node(4)at(5,-2.5){};\node(5)at(-5,2.5){};\node(6)at(-4,2.5){};\node(7)at(-3,2.5){};\node(8)at(-1,2.5){};\node(9)at(0,2.5){};\node(10)at(-5,-2.5){};\node(11)at(-4,-2.5){};\node(12)at(-2,-2.5){};\node(13)at(-1,-2.5){};\node(14)at(1,2.5){};\node(15)at(2,2.5){};\node(16)at(4,2.5){};\node(17)at(5,2.5){};\draw[postaction={ decorate, decoration={ markings, mark=at position .2 with {{\arrow[scale=.8]{>}}} } }](5.center)to(0.center);\draw[postaction={ decorate, decoration={ markings, mark=at position .2 with {{\arrow[scale=.8]{>}}} } }](6.center)to(1.center);\draw[postaction={ decorate, decoration={ markings, mark=at position .2 with {{\arrow[scale=.8]{>}}} } }](7.center)to(2.center);\draw[postaction={ decorate, decoration={ markings, mark=at position .85 with {{\arrow[scale=.8]{>}}} } }](8.center)to(3.center);\draw[postaction={ decorate, decoration={ markings, mark=at position .85 with {{\arrow[scale=.8]{>}}} } }](9.center)to(4.center);\draw[postaction={ decorate, decoration={ markings, mark=at position .2 with {{\arrow[scale=.8]{>}}} } }](10.center)to(14.center);\draw[postaction={ decorate, decoration={ markings, mark=at position .2 with {{\arrow[scale=.8]{>}}} } }](11.center)to(15.center);\draw[postaction={ decorate, decoration={ markings, mark=at position .85 with {{\arrow[scale=.8]{>}}} } }](12.center)to(16.center);\draw[postaction={ decorate, decoration={ markings, mark=at position .85 with {{\arrow[scale=.8]{>}}} } }](13.center)to(17.center);\node[scale=.8]at(0.3,0.2){$\cdots$};\node[scale=.8]at(1.98,1.6){$\cdots$};\node[scale=.8]at(-1.86,-1.6){$\cdots$};\node[scale=.8]at(-1.1,1.6){$\cdots$};\node[scale=.8]at(2.1,-1.6){$\cdots$};\end{tikzpicture}}\end{equation}\par The coherence condition for generating morphisms of $\mathcal{R}^+(\nu)$ and $\mathcal{R}^-(\lambda-\mu+\nu)$ amounts to sliding dots and crossings between strands of the same orientation through the middle crossing region. Since every strand flows to the right in \eqref{XB}, this can be done by applying the KLR relations \eqref{c}--\eqref{d} rotated counterclockwise by $90^\circ$, modulo the remainder terms that contain one cap and one cup. Such terms factor through objects with one fewer $E^+_i$ than $u_1$, and thus vanish in $\mathscr{U}(\lambda,\mu)_{\not\ge\nu}$ by \thref{UB}.\par Next, we can apply relations \eqref{j}--\eqref{l} to iteratively resolve bigons in $\overline{\alpha_{u_1\mathbin{\BeginAccSupp{method=hex,unicode,ActualText=22A0}\mathpalette{\m}{\otimes}\EndAccSupp{}}u_2}}\alpha_{u_1\mathbin{\BeginAccSupp{method=hex,unicode,ActualText=22A0}\mathpalette{\m}{\otimes}\EndAccSupp{}}u_2}$ while introducing morphisms that factor through objects with one fewer $E^+_i$ than $u_1$, which again vanish in $\mathscr{U}(\lambda,\mu)_{\not\ge\nu}$ by \thref{UB}. This shows $\overline{\alpha_{u_1\mathbin{\BeginAccSupp{method=hex,unicode,ActualText=22A0}\mathpalette{\m}{\otimes}\EndAccSupp{}}u_2}}\alpha_{u_1\mathbin{\BeginAccSupp{method=hex,unicode,ActualText=22A0}\mathpalette{\m}{\otimes}\EndAccSupp{}}u_2}=1_{u_1\mathbin{\BeginAccSupp{method=hex,unicode,ActualText=22A0}\mathpalette{\m}{\otimes}\EndAccSupp{}}u_2}$, and similarly we have $\alpha_{u_1\mathbin{\BeginAccSupp{method=hex,unicode,ActualText=22A0}\mathpalette{\m}{\otimes}\EndAccSupp{}}u_2}\overline{\alpha_{u_1\mathbin{\BeginAccSupp{method=hex,unicode,ActualText=22A0}\mathpalette{\m}{\otimes}\EndAccSupp{}}u_2}}=1_{u_2\mathbin{\BeginAccSupp{method=hex,unicode,ActualText=22A0}\mathpalette{\m}{\otimes}\EndAccSupp{}}u_1}$. \end{proof}\par As with \thref{|A}, we conclude from the universal property of additive Karoubi envelopes that \begin{corollary}\thlabel{YB} For $u_1\in\dot{\mathcal{R}}^+(\nu)$ and $u_2\in\dot{\mathcal{R}}^-(\lambda-\mu+\nu)$, we have $u_2u_1\textnormal{\fontfamily{pbk}\selectfont1}_\lambda\cong u_1u_2\textnormal{\fontfamily{pbk}\selectfont1}_\lambda$ in $\dot{\mathscr{U}}(\lambda,\mu)_{\not\ge\nu}$.\end{corollary}\par Recall the involution $\ast\colon B(\infty)\to B(\infty)$ induced by the $\mathbb{Q}(q)$-algebra antiautomorphism of $U^-$ fixing each $F_i$ (cf.\ \cite[\S2]{Q}), and the operators $\widetilde{e}_i^\ast\coloneqq\ast\circ\widetilde{e}_i\circ\ast$ and $\widetilde{f}_i^\ast\coloneqq\ast\circ\widetilde{f}_i\circ\ast$ on $B(\infty)$.\par\begin{lemma}\thlabel{ZB} $G(b)^\ast\cong G(b^\ast)$ for $b\in B(\infty)$.\end{lemma} \begin{proof}Regard $B\coloneqq\bigsqcup_{\nu\in\mathbb{N}[I]}\mathbf{B}^{\textup{o}}(\operatorname{mod}\!\operatorname{-}\hspace{-.8mm}\mathcal{R}^-(\nu))$ as a crystal by identifying each component with a direct limit of $\{\mathbf{B}^{\textup{o}}(\operatorname{mod}\!\operatorname{-}\hspace{-.8mm}\mathcal{R}^\lambda_{\lambda-\nu})\}_{\lambda\in X^+}$ (cf.\ \S\ref{OB}) and let $L\colon B(\infty)\to B$ be the identification of crystals. By \cite[Prop.~7.3]{E}, there exists a strict embedding $\Psi_i\colon B\hookrightarrow B\otimes B_i$ that sends $L(b)$ to $(\tilde{e}_i^\vee)^cL(b)\otimes b_i(-c)$, where $B_i=\{b_i(n)\}_{n\in\mathbb{Z}}$ is the crystal of \cite[Ex.~1.2.6]{Q}, $c=\varepsilon_i(L(b)\circ\ast)$, and $\tilde{e}_i^\vee L(b)\coloneqq\big(\widetilde{e}_i(L(b)\circ\ast)\big)\circ\ast$, which lies in $B$ since $\ast$ commutes with $\mathop{-}\nolimits$.\footnote{\cite{E} uses an involution $\sigma$ defined with a different sign convention than $\ast$, but $M\circ\sigma\cong M\circ\ast$ holds for all $M\in\operatorname{mod}\!\operatorname{-}\hspace{-.8mm}\mathcal{R}^-(\nu)$.} The strict embedding $B(\infty)\hookrightarrow B(\infty)\otimes B_i$ that sends $u_\infty$ to $u_\infty\otimes b_i(0)$ is unique, so by comparing $\Psi_i$ with \cite[Thm.~2.2.1]{Q}, we have \begin{equation}\label{aB}\varepsilon_i(L(b)\circ\ast)=\varepsilon_i^\ast(b)\quad\text{and}\quad(\tilde{e}_i^\vee)^cL(b)\cong L\big((\widetilde{e}^\ast_i)^cb\big).\end{equation} We show $L(b)\circ\ast\cong L(b^\ast)$ for $b\in B(\infty)_{-\nu}$ by induction on $|\nu|$. It is clear for $b=u_\infty$. Otherwise, take $i\in I$ such that $\varepsilon_i^\ast(b)>0$. Then $\widetilde{e}_i(L(b)\circ\ast)\cong(\tilde{e}_i^\vee L(b))\circ\ast\cong L(\widetilde{e}_i^\ast b)\circ\ast\cong L(\widetilde{e}_ib^\ast)\cong\widetilde{e}_iL(b^\ast)$, where the second and the third isomorphisms use \eqref{aB} and the induction hypothesis, respectively. Applying $\widetilde{f}_i$ completes the induction. Since $G(b)^\ast$ represents the projective cover of $L(b)\circ\ast$, the statement follows.\end{proof}\par\begin{lemma}\thlabel{bB} $G(b_1\otimes t_\lambda\otimes b_2)^\ast\cong G(b_1^\ast\otimes t_{-\mu}\otimes b_2^\ast)$ for $b_1\in B(-\infty)_\nu$ and $b_2\in B(\infty)_{\mu-\lambda-\nu}$.\end{lemma} \begin{proof}Note that $\ast\colon\mathscr{U}(\lambda,\mu)\to\mathscr{U}(-\mu,-\lambda)$ induces a functor $\mathscr{U}(\lambda,\mu)_{\le\nu}\to\mathscr{U}(-\mu,-\lambda)_{\le\nu}$, since $(u_2u_1\textnormal{\fontfamily{pbk}\selectfont1}_\lambda)^\ast\cong u_1^\ast u_2^\ast\textnormal{\fontfamily{pbk}\selectfont1}_{-\mu}$ for $u_1\in\mathcal{R}^+(\nu)$ and $u_2\in\mathcal{R}^-(\lambda-\mu+\nu)$ vanishes in $\mathscr{U}(-\mu,-\lambda)_{\le\nu}$ by \thref{UB}. We have \begin{equation}\begin{split}G(b_1\otimes t_\lambda\otimes b_2)^\ast&\cong\big(G(b_2)G(b_1)\textnormal{\fontfamily{pbk}\selectfont1}_\lambda\big)^\ast\cong G(b_1)^\ast G(b_2)^\ast\textnormal{\fontfamily{pbk}\selectfont1}_{-\mu}\\&\cong G(b_1^\ast)G(b_2^\ast)\textnormal{\fontfamily{pbk}\selectfont1}_{-\mu}\cong G(b_2^\ast)G(b_1^\ast)\textnormal{\fontfamily{pbk}\selectfont1}_{-\mu}\qquad\text{ in }\mathscr{U}(-\mu,-\lambda)_{\le\nu}\:,\end{split}\end{equation} where we applied \eqref{RB}, \thref{ZB}, and \thref{YB} at the first, third, and fourth isomorphisms, respectively. Since \eqref{RB} uniquely characterizes $G\colon B(-\infty)\otimes T_\lambda\otimes B(\infty)\to\bigsqcup_{\mu\in X}\mathbf{G}^{\textup{o}}(\dot{\mathscr{U}}(\lambda,\mu))$, we obtain the desired conclusion. \end{proof}\par By \cite[Thm.~3.1.1, Cor.~4.3.3, and \S5]{O}, the crystal $B(\widetilde{U}(\mathfrak{g}))$ of the modified quantized enveloping algebra $\widetilde{U}(\mathfrak{g})$ is isomorphic to $\bigsqcup_{\lambda\in X}B(-\infty)\otimes T_\lambda\otimes B(\infty)$ as a $\mathfrak{g}$-crystal, on which the involution $\ast$ induced by the antiautomorphism of $\widetilde{U}(\mathfrak{g})$ given by $E_i\mapsto E_i$, $F_i\mapsto F_i$, and $1_\lambda\mapsto1_{-\lambda}$ acts as $b_1\otimes t_\lambda\otimes b_2\mapsto b_1^\ast\otimes t_{-\lambda-\operatorname{wt}b_1-\operatorname{wt}b_2}\otimes b_2^\ast$, and on which the second $\mathfrak{g}$-crystal structure is given by $\varepsilon_i^\ast(b)\coloneqq\varepsilon_i(b^\ast)$, $\varphi_i^\ast(b)\coloneqq\varphi_i(b^\ast)$, $\widetilde{e}_i^\ast(b)\coloneqq(\widetilde{e}_ib^\ast)^\ast$, and $\widetilde{f}_i^\ast(b)\coloneqq(\widetilde{f}_ib^\ast)^\ast$, referred to as the star crystal structure.\par Recall that $\bigsqcup_{\lambda,\mu\in X}\mathbf{G}^{\textup{o}}(\dot{\mathscr{U}}(\lambda,\mu))$ has two $\mathfrak{g}$-crystal structures provided by \thref{EA,cA}, and hence can be regarded as a $\mathfrak{g}\oplus\mathfrak{g}$-crystal.\par\begin{theorem}\thlabel{_} There exists an isomorphism of $\mathfrak{g}\oplus\mathfrak{g}$-crystals \[B(\widetilde{U}(\mathfrak{g}))\cong\bigsqcup_{\lambda,\mu\in X}\mathbf{G}^{\textup{o}}(\dot{\mathscr{U}}(\lambda,\mu))\:.\]\end{theorem} \begin{proof}Let $G\colon B(\widetilde{U}(\mathfrak{g}))\to\bigsqcup_{\lambda,\mu\in X}\mathbf{G}^{\textup{o}}(\dot{\mathscr{U}}(\lambda,\mu))$ be the disjoint union of the maps $G$ of \thref{DA} over $\lambda\in X$. It remains to show that it intertwines the star crystal structure on both sides. By the characterization of the crystal on the orthodox basis in \thref{JA} (see also \thref{aA}), we have $\varepsilon_i^\ast(P)=\varepsilon_i(P^\ast)$ and $\widetilde{e}_i^\ast(P)=(\widetilde{e}_iP^\ast)^\ast$ for $P\in\mathbf{G}^{\textup{o}}(\dot{\w}(\lambda,\mu))$. Thus, $\widetilde{e}_i^\ast G(b)\cong(\widetilde{e}_iG(b)^\ast)^\ast\cong(\widetilde{e}_iG(b^\ast))^\ast\cong G(\widetilde{e}_ib^\ast)^\ast\cong G(\widetilde{e}_i^\ast b)$, where the second and the fourth isomorphisms follow from \thref{bB}. Hence, $G$ intertwines $\widetilde{e}_i^\ast$, and similarly for $\widetilde{f}_i^\ast$. Therefore, $G$ is an isomorphism of $\mathfrak{g}\oplus\mathfrak{g}$-crystals.\end{proof}\section{Categorification of extremal weight modules}\label{H}\par\subsection{The modules $V(\lambda)$}In what follows, when $n<0$, we write $E_i^n$ (resp.\ $E_i^{(n)}$) for $F_i^{-n}$ (resp.\ $F_i^{(-n)}$), and similarly for $F_i$. For $i\in I$, set $S_i1_\lambda\coloneqq F_i^{(\langle i,\lambda\rangle)}1_\lambda\in\widetilde{U}_\mathbb{Z}$, and similarly $S_i\textnormal{\fontfamily{pbk}\selectfont1}_\lambda$ as a 1-morphism of $\mathscr{U}$. Recall the braid group action $S_i$ on normal crystals given by $S_ib\coloneqq\widetilde{f}_i^{\langle i,\operatorname{wt}b\rangle}b$ if $\langle i,\operatorname{wt}b\rangle\ge0$ and $\widetilde{e}_i^{-\langle i,\operatorname{wt}b\rangle}b$ if $\langle i,\operatorname{wt}b\rangle<0$. Set $S_i^\ast b\coloneqq(S_ib^\ast)^\ast$ for $b\in B(\widetilde{U}(\mathfrak{g}))$.\par Let $\lambda\in X$. Set $\lambda_+\coloneqq\sum_{\langle i,\lambda\rangle>0}\langle i,\lambda\rangle\Lambda_i$, $\lambda_-\coloneqq\lambda_+-\lambda$, and $\mathcal{DR}^\lambda\coloneqq\mathcal{DR}^{-\lambda_-,\lambda_+}$. Note that $K_0(\mathcal{DR}^\lambda)\cong V_\mathbb{Z}(-\lambda_-)\otimes V_\mathbb{Z}(\lambda_+)$ as a $\widetilde{U}_\mathbb{Z}(\mathfrak{g})$-module (see \thref{?A}), and $G(b)\in\dot{\w}(\lambda,\mu)$ for $b\in B(\widetilde{U}(\mathfrak{g})1_{\lambda})=B(-\infty)\otimes T_\lambda\otimes B(\infty)$ vanishes in $\mathcal{DR}^\lambda$ if and only if $d_i(b^\ast)\ne|\langle i,\lambda\rangle|$ for some $i\in I$, i.e., $b^\ast$ is not $i$-extremal (see \thref{DA} and \cite[\S 8]{O}).\par\begin{lemma}\thlabel{cB} For $b\in B(\widetilde{U}(\mathfrak{g})1_{\lambda})$, we have $G(b)S_i\textnormal{\fontfamily{pbk}\selectfont1}_{s_i\lambda}\cong G(S_i^\ast b)$ in $\dot{\mathcal{DR}}^{s_i\lambda}_{\operatorname{wt}b}$. \end{lemma} \begin{proof}Assume $\langle i,\lambda\rangle\ge0$. \thref{JA} applied to $\mathscr{U}(-,\operatorname{wt}b)\circ\ast$ gives $G(b)E_i^{(\langle i,\lambda\rangle)}\cong\genfrac{[}{]}{0pt}{}{\varepsilon_i(b^\ast)}{\langle i,\lambda\rangle}_iG(S_i^\ast b)\oplus\bigoplus_{b'}c_{b'}G(b')$, where $b'\in B(\widetilde{U}(\mathfrak{g})1_{s_i\lambda})$ in the summation satisfies $d_i(b'^\ast)>d_i(b^\ast)$. Any such $G(b')$ vanishes in $\dot{\mathcal{DR}}^{s_i\lambda}$, and so does $G(S_i^\ast b)$ if $b^\ast$ is not $i$-extremal. Otherwise, the coefficient of $G(S_i^\ast b)$ is 1. The case for $\langle i,\lambda\rangle<0$ is similar. \end{proof}\par Set $I^{\ast}\coloneqq\bigsqcup_{k=0}^\infty I^k$. For $\underline{i}=(i_1,\ldots,i_l)\in I^{\ast}$, write $S_{\underline{i}}\coloneqq S_{i_l}\cdots S_{i_1}$. An element $b\in B(\widetilde{U}(\mathfrak{g}))$ is called extremal if $S_{\underline{i}}b$ is $j$-extremal for all $\underline{i}\in I^{\ast}$ and $j\in I$. The crystal of the extremal weight module $V(\lambda)$ is given by \cite[\S 8]{O} \begin{equation}\label{dB}B(\lambda)\coloneqq\big\{\mkern1mub\in B(\widetilde{U}(\mathfrak{g})1_{\lambda})\,|\,b^\ast\text{ is extremal}\mkern1mu\big\}.\end{equation}\par There exists an isomorphism of $\mathbb{Z}[q^{\pm}]$-algebras (cf.\ \cite[Thm.~1.2]{G}, \cite[Thm.~A]{J}, \cite[Thm.~6.18]{I}) \[\widetilde{U}_\mathbb{Z}(\mathfrak{g})\xrightarrow{\kern.5em\sim\kern.5em}\bigoplus_{\lambda,\mu\in X}K_0\big(\dot{\w}(\lambda,\mu)\big)\] that sends $E^{(n)}_i1_\lambda$ to $[E^{(n)}_i\textnormal{\fontfamily{pbk}\selectfont1}_\lambda]$ ($i\in I$, $n\in\mathbb{Z}$). Let \[I_\lambda\coloneqq\sum_{b\in B(\widetilde{U}(\mathfrak{g})1_{\lambda})\setminus B(\lambda)}\mathbb{Z}[q^\pm][G(b)]\subset\bigoplus_{\mu\in X}K_0(\dot{\w}(\lambda,\mu)).\]\par\begin{lemma}\thlabel{eB} \begin{enumerate*}[(i),font=\normalfont]\item $I_\lambda S_i\subset I_{s_i\lambda}$. \label{fB}\item $I_\lambda$ is a $\widetilde{U}_\mathbb{Z}(\mathfrak{g})$-submodule of $\widetilde{U}_\mathbb{Z}(\mathfrak{g})1_\lambda$.\end{enumerate*}\end{lemma}\par\begin{proof}Let $b\in B(\widetilde{U}(\mathfrak{g})1_{\lambda})\setminus B(\lambda)$, and let $q^cG(b')$ be a summand of $G(b)S_i$. Since $S_ib^\ast$ is not extremal, \thref{cB} implies $b'\not\in B(s_i\lambda)$, proving \ref{fB}. Let $\underline{i}=(i_1,\ldots,i_l)\in I^{\ast}$ be such that $S_{i_{k-1}}\cdots S_{i_1}b^\ast$ is $i_{k}$-extremal for $1\le k<l$, but $S_{i_{l-1}}\cdots S_{i_1}b^\ast$ is not $i_l$-extremal. We prove $[uG(b)]\in I_\lambda$ for any 1-morphism $u$ of $\mathscr{U}$ by induction on $l$. If $l=1$, then $G(b)\cong0$ in $\dot{\mathcal{DR}}^\lambda$. Since $\dot{\z}^\lambda$ is a 2-functor, $uG(b)\cong0$ in $\dot{\z}^\lambda$, and hence $[uG(b)]\in I_\lambda$. Assume $l>1$, and let $q^cG(b')$ be a summand of $G(b)S_{i_1}$. By \thref{cB}, either $b'=S_{i_1}^\ast b$ or ${b'}^\ast$ is not $j$-extremal for some $j\in I$. In each case, $[uG(b')]\in I_{s_{i_1}\lambda}$ by the induction hypothesis for $l=l-1$ or $l=1$, respectively. Hence $[uG(b)S_{i_1}]\in I_{s_{i_1}\lambda}$, and $[uG(b)S_{i_1}S_{i_1}]\in I_{s_{i_1}\lambda}S_{i_1}\subset I_\lambda$ by \ref{fB}. Since $\textnormal{\fontfamily{pbk}\selectfont1}_\lambda$ is a retract of $S_{i_1}S_{i_1}\textnormal{\fontfamily{pbk}\selectfont1}_\lambda$ \eqref{u}, $uG(b)$ is a direct summand of $uG(b)S_{i_1}S_{i_1}$, and so $[uG(b)]\in I_\lambda$.\end{proof}\par For $\epsilon\in\{-,+\}$ and $\lambda\in X$, let $I^{\ast,\epsilon}_\lambda\coloneqq\left\{\kern.07em(\underline{i},j)\in I^{\ast}\times I\,\middle|\,\epsilon\langle j,w_{\underline{i}}(\lambda)\rangle\ge0\kern.07em\right\}$, where $w_{\underline{i}}\coloneqq s_{i_l}\cdots s_{i_1}\in W$ for $\underline{i}=(i_1,\ldots,i_l)$. We have a presentation $V(\lambda)\cong\widetilde{U}(\mathfrak{g})1_\lambda/I'_\lambda$, where \[I'_\lambda\coloneqq\sum_{(\underline{i},j)\in I^{\ast,+}_\lambda}\widetilde{U}(\mathfrak{g})E_jS_{\underline{i}}1_\lambda+\sum_{(\underline{i},j)\in I^{\ast,-}_\lambda}\widetilde{U}(\mathfrak{g})F_jS_{\underline{i}}1_\lambda.\]\par\begin{lemma}\thlabel{gB} $I'_\lambda=\mathbb{Q}(q)I_\lambda$.\end{lemma} \begin{proof}Let $(\underline{i},j)\in I^{\ast,\epsilon}_\lambda$ and $u$ be any 1-morphism of $\mathscr{U}$. If $\underline{i}$ is of length 0, $uE^\epsilon_j\textnormal{\fontfamily{pbk}\selectfont1}_\lambda$ vanishes in $\mathcal{DR}^\lambda$ by the defining relations. This implies $[uE^\epsilon_j\textnormal{\fontfamily{pbk}\selectfont1}_\lambda]\in I_\lambda$. For general $\underline{i}$, we have $[uE^\epsilon_jS_{\underline{i}}\textnormal{\fontfamily{pbk}\selectfont1}_\lambda]\in I_\lambda$ by the case $l=0$ and \thref{eB} \ref{fB}. This shows $I'_\lambda\subset\mathbb{Q}(q)I_\lambda$. Let $b\in B(\widetilde{U}(\mathfrak{g})1_{\lambda})\setminus B(\lambda)$. For the opposite inclusion, we prove $[G(b)]\in I'_\lambda$ by induction on the length of $\underline{i}$ (for all $\lambda\in X$), where $\underline{i}=(i_1,\ldots,i_l)\in I^{\ast}$ is chosen as in the proof of \thref{eB}. If $l=1$, $G(b)$ vanishes in $\mathcal{DR}^\lambda$, so $[G(b)]\in K_0(\dot{\mathscr{U}}(\lambda,\mu))\cong1_\mu\widetilde{U}_\mathbb{Z}(\mathfrak{g})1_\lambda$ vanishes in $V_\mathbb{Z}(-\lambda_-)\otimes V_\mathbb{Z}(\lambda_+)$. Since \[V(-\lambda_-)\otimes V(\lambda_+)\cong\widetilde{U}(\mathfrak{g})1_\lambda\Big/\sum_{j\in I}\big(\widetilde{U}(\mathfrak{g})F_j^{1+\langle j,\lambda_+\rangle}1_\lambda+\widetilde{U}(\mathfrak{g})E_j^{1+\langle j,\lambda_-\rangle}1_\lambda\big)\,,\] $[G(b)]$ lies in the denominator, which is a subspace of $I'_\lambda$. Assume $l>1$. By the induction hypothesis, $[G(S_{i_1}^\ast b)]\in I'_{s_{i_1}(\lambda)}$, so $[G(S_{i_1}^\ast b)S_{i_1}]\in I'_{s_{i_1}(\lambda)}S_{i_1}\subset I'_\lambda$. By \thref{cB} and the case $l=1$, we have $[G(S_{i_1}^\ast b)S_{i_1}]-[G(b)]\in I'_\lambda$. Hence $[G(b)]\in I'_\lambda$.\end{proof}\par Let $\widetilde{\mathcal{V}}^\lambda_\mu$ be the quotient of $\mathscr{U}(\lambda,\mu)$ by the two-sided ideal generated by identity morphisms of objects \begin{equation}\label{hB}uE_j^\epsilon F_{i_l}^{\langle i_l,s_{i_{l-1}}\cdots s_{i_1}(\lambda)\rangle}\cdots F_{i_1}^{\langle i_1,\lambda\rangle}\textnormal{\fontfamily{pbk}\selectfont1}_\lambda\quad\text{ for }\epsilon\in\{-,+\}\text{ and }(\underline{i},j)\in I^{\ast,\epsilon}_\lambda,\end{equation} where $u$ is any 1-morphism of $\mathscr{U}$. Let $\mathcal{V}^\lambda_\mu$ be the quotient of $\widetilde{\mathcal{V}}^\lambda_\mu$ obtained by killing far-left positive degree bubbles. Clearly, $\mathscr{U}(\lambda,-)\colon\mathscr{U}\to\mathrm{Cat}_\Bbbk$ induces 2-representations $\widetilde{\mathcal{V}}^\lambda,\mathcal{V}^\lambda\colon\mathscr{U}\to\mathrm{Cat}_\Bbbk$.\par Let $V_\mathbb{Z}(\lambda)$ be the $\widetilde{U}_\mathbb{Z}(\mathfrak{g})$-submodule of $V(\lambda)$ generated by the extremal weight vector $u_\lambda$.\par\begin{theorem}\thlabel{A} There exist isomorphisms of $\widetilde{U}_\mathbb{Z}(\mathfrak{g})$-modules \[\bigoplus_{\mu\in X}K_0(\smash{\dot{\F}}\vphantom{\dot{\mathcal{V}}}^\lambda_\mu)\cong\bigoplus_{\mu\in X}K_0(\dot{\E}^\lambda_\mu)\cong V_\mathbb{Z}(\lambda)\ \] and isomorphisms of crystals $\mathbf{G}^{\textup{o}}(\smash{\dot{\F}}\vphantom{\dot{\mathcal{V}}}^\lambda)\cong\mathbf{G}^{\textup{o}}(\dot{\E}^\lambda)\cong B(\lambda)$. \end{theorem} \begin{proof}By \thref{gB}, $G(b)\in\mathbf{G}^{\textup{o}}(\dot{\w}(\lambda,\mu))$ vanishes in $\smash{\dot{\F}}\vphantom{\dot{\mathcal{V}}}^\lambda_\mu$ if and only if $[G(b)]$ vanishes in $V(\lambda)$, and thus $\mathbb{Q}(q)\otimes_{\mathbb{Z}[q^\pm]}\bigoplus_{\mu\in X}K_0(\smash{\dot{\F}}\vphantom{\dot{\mathcal{V}}}^\lambda_\mu)\cong V(\lambda)$. Since $K_0(\smash{\dot{\F}}\vphantom{\dot{\mathcal{V}}}^\lambda)$ is generated by $[\textnormal{\fontfamily{pbk}\selectfont1}_\lambda]$ as a $\bigoplus_{\mu,\mu'\in X}K_0(\dot{\w}(\mu,\mu'))\cong\widetilde{U}_\mathbb{Z}(\mathfrak{g})$-module, we have $K_0(\smash{\dot{\F}}\vphantom{\dot{\mathcal{V}}}^\lambda)\cong V_\mathbb{Z}(\lambda)$. By \thref{GB}, passing to $\mathcal{V}^\lambda_\mu$ does not change the class of indecomposable objects. The crystal isomorphism is clear from \thref{DA} and \eqref{dB}.\end{proof}\par\begin{corollary}\thlabel{iB} There exists a nonzero symmetric bilinear form $(\mathop{-}\nolimits,\mathop{-}\nolimits)\colon\allowbreak V(\lambda)\times V(\lambda)\to\mathbb{Q}(q)$ satisfying $(xu,v)=(u,\rho(x)v)$ for $x\in\widetilde{U}(\mathfrak{g})$ and $u,v\in V(\lambda)$.\end{corollary} \begin{proof}Take $\Bbbk$ to be a field. For $P,Q\in\dot{\E}^\lambda_\mu$, set $([P],[Q])\coloneqq\sum_{d\in\mathbb{Z}}q^{-d}\dim\dot{\E}^\lambda_\mu(q^dP,\overline{Q})\in\mathbb{Z}\mathopen{(\mkern-2mu(}q^{-1}\mathclose{)\mkern-2mu)}$. Clearly, this defines a $\mathbb{Z}[q^\pm]$-bilinear form on $K_0(\dot{\E}^\lambda_\mu)$, which extends to a $\mathbb{Q}(q)$-bilinear form on $V(\lambda)$ by \thref{A}. This is symmetric since the bar-involution gives an isomorphism $\dot{\E}^\lambda_\mu(q^dP,\overline{Q})\cong\dot{\E}^\lambda_\mu(q^dQ,\overline{P})$. The adjunction $u\dashv\overline{\rho(u)}$ in $\mathscr{U}$ gives $\dot{\E}^\lambda_{\mu'}(q^dxP,\overline{Q})\cong\mathcal{V}^\lambda_\mu(q^dP,\overline{\rho(x)Q})$ ($x\in\dot{\w}(\mu,\mu')$, $P\in\dot{\E}^\lambda_\mu$, $Q\in\dot{\E}^\lambda_{\mu'}$), which implies the contravariance by \thref{r}. We have $(u_\lambda,u_\lambda)=1$ by $\operatorname{End}_q(\textnormal{\fontfamily{pbk}\selectfont1}_\lambda)\cong\Bbbk$.\par Let $P=E_{i_l}^{\epsilon_l}\cdots E_{i_1}^{\epsilon_1}\textnormal{\fontfamily{pbk}\selectfont1}_\lambda$ and $Q=E_{j_m}^{\epsilon'_m}\cdots E_{j_1}^{\epsilon'_1}\textnormal{\fontfamily{pbk}\selectfont1}_\lambda$. Then $\mathscr{U}(\lambda,\mu)|_{q=1}(q^dP,Q)$ is finitely generated as a module over $\operatorname{End}_q(\textnormal{\fontfamily{pbk}\selectfont1}_\lambda)\otimes\Bbbk[x_1,\ldots,x_l,y_1,\ldots,y_m]$, where $x_k$ (resp. $y_k$) acts by placing a dot at the bottom (resp. top) of the $k$-th strand \cite[\S3.2.3]{G}. Hence, $\mathcal{V}^\lambda_\mu|_{q=1}(q^dP,\overline{Q})$ is a finitely generated module over $\Bbbk[x_1,\ldots,x_l,y_1,\ldots,y_m]$, and its Hilbert series is a rational function in $q$. Since the classes $[P]$ span $V(\lambda)$ over $\mathbb{Q}(q)$, the bilinear form takes values in $\mathbb{Q}(q)$.\end{proof}\par\begin{remark}\thlabel{jB} Note that a choice of a base field and a quiver Hecke datum determines a point $x\in\operatorname{Spec}\mathbf{k}_C$, where $\mathbf{k}_C$ is the ring of indeterminates (cf.\ \thref{_B}). If the Hom-spaces of $\dot{\E}^\lambda_\mu$ are finite free over all Henselian local rings $\Bbbk$, then taking $\Bbbk$ to be the \'etale stalk at $x$ shows that the form of \thref{iB} is independent of $x$. We do not know whether the Hom-spaces of $\dot{\E}^\lambda_\mu$ are free in general.\end{remark}\par\begin{remark}In \cite[Conj.~2.12]{s}, Kashiwara conjectured that there is a non-degenerate symmetric bilinear form on $V(\lambda)$, characterized by $(u_\lambda,G(b))=\delta_{b,u_\lambda}$ ($b\in B(\lambda)$) and $(xu,v)=(u,\rho(x)v)$ ($x\in\widetilde{U}(\mathfrak{g})$), which satisfies $(G(b),G(b'))\in\delta_{b,b'}+q^{-1}\mathbb{Z}[[q^{-1}]]$ for $b,b'\in B(\lambda)$ and $(G(b),G(b'))=\delta_{b,b'}$ for $b,b'\in B(\lambda)_\lambda$. Here $G(b)$ denotes the global basis. For $\mathfrak{g}$ of affine type, the bilinear form exists \cite[Prop.~4.1]{t}, and the almost orthonormality and the orthonormality on $V(\lambda)_\lambda$ are proved in \cite{R}; for basic weights, these are due to \cite{u}.\par If the canonical basis of $\widetilde{U}$ and the orthodox basis of $\dot{\w}$ coincide (cf.\ \cite[\S8]{V}), then $\{[G(b)]\}_{b\in B(\lambda)}$ is the global basis of $V(\lambda)$, and \thref{iB} gives the form in the conjecture if and only if \begin{equation}\label{kB}\dot{\E}^\lambda_\lambda|_{q=1}(\textnormal{\fontfamily{pbk}\selectfont1}_\lambda,G(b))=0\quad\text{for }b\in B(\lambda)_\lambda\setminus\{u_\lambda\}.\end{equation} Note that \eqref{kB} implies the freeness of the Hom-spaces in \thref{jB}. We do not know whether \eqref{kB} holds in general.\end{remark}\subsection{Fundamental weight modules in affine type}Let $(I,\cdot)$ be of affine type, and $X$ be its standard weight lattice of rank $|I|+1$. Let $\delta\in\mathbb{N}[I]$ be such that $\mathbb{Z}\delta$ is the radical of the bilinear form on $\mathbb{Z}[I]$, $X^0\coloneqq X\cap\mathbb{Q}[I]$ the level-zero weight lattice, $X^0_{\operatorname{cl}}\coloneqq X^0/\mathbb{Z}\delta$ the classical weight lattice, and ${\operatorname{cl}}\colon X^0\to X^0_{\operatorname{cl}}$ the canonical projection.\par Suppose that the underlying quiver Hecke datum satisfies $\prod_{j\in I}t_{ij}^{a_j}=1$ for all $i\in I$, where $\delta=\sum_{i\in I}a_ii$. Then the maps $\langle i,-\rangle$ and $c_i$ ($i\in I$) restrict to $X^0$ and factor through $X^0_{\operatorname{cl}}$. Let $\widetilde{U}'$ be the algebra $\widetilde{U}$ associated with the weight datum $X^0_{\operatorname{cl}}$. Recall that the definition of the 2-category $\mathscr{U}$ (\thref{Q}) and that of $\mathcal{V}^\lambda$ apply to an arbitrary weight datum. Let $\mathscr{U}'$ be the 2-category associated with the same Cartan datum and quiver Hecke datum, \bytes